\documentclass[reqno]{amsart}
\usepackage{amsthm, amssymb, amsfonts}
\usepackage{lmodern}
\usepackage{color}
\usepackage[T5]{fontenc}
\usepackage{amscd,amssymb}
\usepackage[v2,cmtip]{xy}
\usepackage{mathrsfs}
\theoremstyle{plain}
\newtheorem{thm}{Theorem}[section]
\newtheorem{prop}[thm]{Proposition}
\newtheorem{lem}[thm]{Lemma}
\newtheorem{con}[thm]{Conjecture}
\newtheorem{corl}[thm]{Corollary}
\theoremstyle{definition}
\newtheorem{defn}[thm]{Definition}

\newtheorem{nota}[thm]{Notation}
\theoremstyle{plain}
\newtheorem{thms}{Theorem}[subsection]
\newtheorem{props}[thms]{Proposition}
\newtheorem{lems}[thms]{Lemma}
\newtheorem{corls}[thms]{Corollary}
\newtheorem{khi}[thms]{Case}
\theoremstyle{definition}
\newtheorem{exa}{Example}[subsection]

\renewcommand{\theequation}{\arabic{section}.\arabic{equation}}
\def\vdvh{\vrule height 11pt depth 5pt width 0.5pt}
\def\vdq{\vdvh\ }
\def\DD{D\kern-.7em\raise0.4ex\hbox{\char '55}\kern.33em}
\begin{document} 
\title[On the Peterson hit problem ]
{On the Peterson hit problem}
\author{Nguy\~\ecircumflex n Sum}

\footnotetext[1]{\ 2010 {\it Mathematics Subject Classification}. Primary 55S10; 55S05, 55T15.}
\footnotetext[2]{\ {\it Keywords and phrases}. Steenrod squares, polynomial algebra, Peterson hit problem.}
\footnotetext[3]{\ This version is a revision of a preprint of Quy Nh\ohorn n University, Vi\^{\d e}t Nam, 2011.}

\begin{abstract}
We study the {\it hit problem}, set up by F. Peterson, of finding a minimal set 
of generators for the polynomial algebra $P_k := \mathbb F_2[x_1,x_2,\ldots,x_k]$ as a module over the  mod-2 Steenrod algebra, $\mathcal{A}$. In this paper, we study a minimal set of generators 
for $\mathcal A$-module $P_k$ in some so-called generic degrees and apply these results  to explicitly determine the hit problem for $k=4$. 
\end{abstract}
\maketitle

\begin{center}
{\it Dedicated to Prof. N. H. V. H\uhorn ng on the occasion of his sixtieth birthday}
\end{center}

\bigskip

\section{Introduction and statement of results}\label{s1} 
\setcounter{equation}{0}

Let $V_k$ be an elementary abelian 2-group of rank $k$. Denote by $BV_k$ the classifying space of $V_k$. 
It may be thought of as the product of $k$ copies of the real projective space $\mathbb RP^\infty$.  
Then  $$P_k:= H^*(BV_k) \cong \mathbb F_2[x_1,x_2,\ldots ,x_k],$$ 
a polynomial algebra in  $k$ variables $x_1, x_2, \ldots , x_k$, each of degree 1. Here the cohomology is taken with coefficients in the prime field $\mathbb F_2$ of two elements. 

Being the cohomology of a space, $P_k$ is a module over the mod-2 Steenrod algebra $\mathcal{A}$. 
The action of $\mathcal A$ on $P_k$ is explicitly given by the formula
$$Sq^i(x_j) = \begin{cases} x_j, &i=0,\\ x_j^2, &i=1,\\ 0, &\text{otherwise,}
\end{cases}$$
and subject to the Cartan formula
$$Sq^n(fg) = \sum_{i=0}^nSq^i(f)Sq^{n-i}(g),$$
for $f, g \in P_k$ (see Steenrod and Epstein~\cite{st}).

A polynomial $f$ in $P_k$ 
is called {\it hit} if it can be written as a finite sum $f = \sum_{i>0}Sq^i(f_i)$ 
for some polynomials $f_i$.  That means $f$ belongs to  $\mathcal{A}^+P_k$, 
where $\mathcal{A}^+$ denotes the augmentation ideal in $\mathcal A$.
We are interested in the {\it hit problem}, set up by F. Peterson, of finding a minimal set 
of generators for the polynomial algebra $P_k$ as a module over the Steenrod algebra. 
In other words, we want to find a basis of the $\mathbb F_2$-vector space 
$QP_k := P_k/\mathcal A^+P_k = \mathbb F_2 \otimes_{\mathcal A} P_k$.

The hit problem was first studied by Peterson~\cite{pe,pe1}, Wood~\cite{wo}, Singer~\cite {si1}, 
and Priddy~\cite{pr}, who showed its relation to several classical problems respectively in cobordism theory, 
modular representation theory, Adams spectral sequence for the stable homotopy of spheres, and
stable homotopy type of classifying spaces of finite groups.
The vector space $QP_k$ was explicitly calculated by 
Peterson~\cite{pe} for $k=1, 2,$ by Kameko~\cite{ka} for $k=3$. The case $k = 4$ has been treated by Kameko~\cite{ka2} and by the present author \cite{su2}. 

\medskip
Several aspects of the hit problem were then investigated by many authors. (See Boardman~\cite{bo}, 
Bruner, H\`a and H\uhorn ng~\cite{br}, Carlisle and Wood~\cite{cw}, 
Crabb and Hubbuck~\cite{ch}, Giambalvo and Peterson~\cite{gp}, H\` a~\cite{ha}, H\uhorn ng~\cite{hu}, H\uhorn ng and Nam~\cite{hn1, hn2}, H\uhorn ng and Peterson~\cite{hp,hp2}, Janfada and Wood~\cite{jw1, jw2}, Kameko~\cite{ka,ka1}, Minami~\cite{mi}, Mothebe \cite{mo,mo1}, Nam~\cite{na,na2}, 
Repka and Selick~\cite{res}, Silverman~\cite{sl}, Silverman and Singer~\cite{ss}, Singer~\cite{si2}, Walker and Wood~\cite{wa, wa1,wa2}, Wood~\cite{wo2,wo3} and others.)

\medskip
The $\mu$-function is one of the numerical functions that have much been used  in the context of the hit problem. For a positive integer $n$, by $\mu(n)$ one means the smallest number $r$ for which it is possible to write $n = \sum_{1\leqslant i\leqslant r}(2^{d_i}-1),$ where $d_i >0$.
 A routine computation shows that  $\mu(n) = s$ if and only if there exists uniquely a sequence of integers $d_1 > d_2 >\ldots > d_{s-1}\geqslant d_s>0$ such that
\begin{equation} \label{ct1.1}n =  2^{d_1} + 2^{d_2}+ \ldots + 2^{d_{s-1}}+ 2^{d_{s}} - s.
\end{equation}
From this it implies $n-s$ is even and $\mu(\frac{n-s}2) \leqslant s$.

Denote by $(P_k)_n$ the subspace of $P_k$ consisting of all the homogeneous polynomials of degree $n$ in $P_k$ and by $(QP_k)_n$ the subspace of $QP_k$ consisting of all the classes represented by the elements in $(P_k)_n$.

\medskip
Peterson~\cite{pe} made the following conjecture, which was subsequently proved by Wood~\cite{wo}.    
 
\begin{thm}[Wood~\cite{wo}]\label{dlmd1} 
If $\mu(n) > k$, then $(QP_k)_n = 0$.
\end{thm} 

One of the main tools in the study of the hit problem is  Kameko's homomorphism 
$\widetilde{Sq}^0_*: QP_k \to QP_k$. 
This homomorphism is induced by the $\mathbb F_2$-linear map, also denoted by  $\widetilde{Sq}^0_*:P_k\to P_k$, given by
$$
\widetilde{Sq}^0_*(x) = 
\begin{cases}y, &\text{if }x=x_1x_2\ldots x_ky^2,\\  
0, & \text{otherwise,} \end{cases}
$$
for any monomial $x \in P_k$. Note that $\widetilde{Sq}^0_*$ is not an $\mathcal A$-homomorphism. However, 
$\widetilde{Sq}^0_*Sq^{2t} = Sq^{t}\widetilde{Sq}^0_*,$ and $\widetilde{Sq}^0_*Sq^{2t+1} = 0$ for any non-negative integer $t$. 

\begin{thm}[Kameko~\cite{ka}]\label{dlmd2} 
Let $m$ be a positive integer. If $\mu(2m+k)=k$, then 
$(\widetilde{Sq}^0_*)_m: (QP_k)_{2m+k}\to (QP_k)_m$ 
is an isomorphism of the $\mathbb F_2$-vector spaces.
\end{thm}

Based on Theorems~\ref{dlmd1} and \ref{dlmd2}, the hit problem is reduced to the case of degree $n$ with $\mu(n) = s < k$. 

\medskip
The hit problem in the case of degree $n$ of the form (\ref{ct1.1}) with $s=k-1$,  $d_{i-1}-d_i>1$ for $2\leqslant i <k$ and $d_{k-1} >1$ was partially studied by Crabb and Hubbuck~\cite{ch}, Nam~\cite{na}, 
Repka and Selick~\cite{res} and the present author \cite{su}.

\medskip
In this paper, we explicitly determine the hit problem for the case $k=4$. First, we study the hit problem for the case of degree $n$ of the form (\ref{ct1.1}) for  $s = k-1$. The following theorem gives an inductive formula for the dimension of $(QP_k)_n$ in this case.

\begin{thm}\label{dl1} Let $n =\sum_{1 \leqslant i \leqslant k-1}(2^{d_i}-1)$ 
with $d_i$ positive integers such that $d_1 > d_2 > \ldots >d_{k-2} \geqslant d_{k-1},$ and let $m = \sum_{1 \leqslant i \leqslant k-2}(2^{d_i-d_{k-1}}-1)$.
If $d_{k-1} \geqslant k-1 \geqslant 3$, then
$$\dim (QP_k)_n = (2^k-1)\dim (QP_{k-1})_m.$$
\end{thm}

For $d_{k-2} > d_{k-1} \geqslant k$, the theorem follows from a result in Nam~\cite{na}. For $d_{k-2} = d_{k-1} > k$, it has been proved in \cite{su}. However, for either $d_{k-1} = k-1$ or $d_{k-2} = d_{k-1} = k$, the theorem is new.

 From the results in Peterson \cite{pe} and Kameko \cite{ka}, we see that if $k=3$, then this theorem  is true for either $d_{1} > d_{2} \geqslant 2$ or $d_1=d_2\geqslant 3$; if $k=2$, then it is true for $d_1 \geqslant 2$. 

The main tool in the proof of the theorem is Singer's criterion  on the hit monomials (Theorem \ref{dlsig}.) So, the condition $d_1 > d_2 > \ldots >d_{k-2} \geqslant d_{k-1}>0$ is used in our proof when we use this criterion.

\medskip
Based on Theorem \ref{dl1}, we explicitly compute $QP_4$. 

\begin{thm}\label{dl3} Let $n $ be an  arbitrary positive integer with $\mu(n) < 4$. The dimension of the $\mathbb F_2$-vector space $(QP_4)_n$ is given by the following table:

\medskip{\rm
\centerline{\begin{tabular}{llcccc}
\ \hskip2.5cm$n$ &\vdq $s=1$ & $s=2$ & $s=3$ & $s=4$  & $s\geqslant 5$\cr
\hline
$2^{s+1}-3$ &\vdq\ \ 4 & 15 & 35 & 45 & 45  \cr
$2^{s+1}-2$&\vdq \ \ 6 & 24 & 50 & 70 & 80  \cr
$2^{s+1}-1$ &\vdq\ 14 & 35 & 75 &  89 & 85 \cr 
$2^{s+2}+2^{s+1}-3$ &\vdq\ 46 & 94 & 105  &105&105\cr
$2^{s+3}+2^{s+1}-3$&\vdq\ 87 & 135 & 150  &150&150\cr
$2^{s+4}+2^{s+1}-3$ &\vdq 136& 180 & 195 &195&195\cr 
$2^{s+t+1}+2^{s+1}-3, t\geqslant 4$ &\vdq 150& 195 & 210 & 210 & 210\cr
 $2^{s+1}+2^s-2$ &\vdq\ 21 & 70 & 116 & 164 & 175 \cr
 $2^{s+2}+2^s-2$&\vdq \ 55 & 126 & 192 &240 & 255  \cr
$2^{s+3}+2^s-2$ &\vdq \ $73$& 165 & 241 & 285 &  300  \cr 
 $2^{s+4}+2^s-2$ &\vdq \ 95& 179 & 255 & 300&  315  \cr
 $2^{s+5}+2^s-2$ &\vdq 115& 175 & 255 & 300 & 315 \cr
$2^{s+t}+2^s-2, t\geqslant 6$ &\vdq 125& 175 & 255 & 300 & 315  \cr
 $2^{s+2}+2^{s+1}+2^s-3$ &\vdq\  64&120  &120 &120&120\cr
 $2^{s+3}+2^{s+2}+2^s-3$  &\vdq 155 & 210 &210&210&210 \cr
 $2^{s+t+1}+2^{s+t}+2^s-3, t\geqslant 3$  &\vdq 140&210  &210&210&210 \cr
$2^{s+3}+2^{s+1}+2^s-3$ &\vdq 140& 225 &225&225&225 \cr
 $2^{s+u+1}+2^{s+1}+2^s-3, u\geqslant 3$ &\vdq  120& 210 &210 &210&210\cr
$2^{s+u+2}+2^{s+2}+2^s-3, u\geqslant 2$&\vdq 225 & 315 &315&315&315 \cr
 $2^{s+t+u}+2^{s+t}+2^s-3, u\geqslant 2, t\geqslant 3$  &\vdq 210 &315& 315 &315 &315 \cr
\end{tabular}}}
\end{thm}

The vector space $QP_4$ was also computed in Kameko \cite{ka2} by using computer calculation. However the manuscript is unpublished at the time of the writing.

Carlisle and Wood showed in \cite{cw} that the dimension of the vector space  $(QP_k)_n$ is uniformly bounded by a number depended only on $k$. In 1990, Kameko made the following conjecture in his Johns Hopkins University PhD thesis \cite{ka}.

\begin{con}[Kameko~\cite{ka}]\label{kac} For every non-negative integer $n$,
$$ \dim (QP_k)_n
\leqslant \prod_{1\leqslant i \leqslant k} (2^i-1).$$ 
\end{con}

The conjecture was shown by Kameko himself for $k\leqslant 3$ in \cite{ka}.
From Theorem \ref{dl3}, we see that the conjecture is also true for $k=4$.

By induction on $k$, using Theorem \ref{dl1}, we obtain the following.
\begin{corl}\label{hqq} Let $n =\sum_{1 \leqslant i \leqslant k-1}(2^{d_i}-1)$ 
with $d_i$ positive integers. If
$d_1-d_2\geqslant 2, d_{i-1} - d_i \geqslant i-1, 3 \leqslant i \leqslant k-1, d_{k-1}\geqslant k-1\geqslant 2$, then
$$\dim (QP_k)_n = \prod_{1 \leqslant i \leqslant k}(2^i-1).$$
\end{corl}
For the case $d_{i-1} - d_i \geqslant i, 2 \leqslant i \leqslant k-1,$ and $d_{k-1}\geqslant k$, this result is due to Nam~\cite{na}. This corollary also shows that  Kameko's conjecture is true for the degree $n$ as given in the corollary.

\medskip
By induction on $k$, using Theorems \ref{dl1}, \ref{dl3} and the fact that   Kameko's homomorphism is an epimorphism, one gets the following.

\begin{corl}\label{hhq}Let  $n =\sum_{1 \leqslant i \leqslant k-2}(2^{d_i}-1)$ with $d_i$ positive integers  and let
$d_{k-1} = 1, \ n_r = \sum_{1 \leqslant i \leqslant r-2}(2^{d_i- d_{r-1}}-1) -1$
with $r= 5,6, \ldots , k$. If  $d_1-d_2 \geqslant 4,\ d_{i-2}-d_{i-1}\geqslant i,$ for $4\leqslant i \leqslant k$ and $k  \geqslant 5,$ then 
$$
\dim (QP_k)_{n} = \prod_{1 \leqslant i \leqslant k}(2^i-1)
+ \sum_{5\leqslant r \leqslant k}\Big(\prod_{r+1 \leqslant i \leqslant k}(2^i-1)\Big)
\dim\text{\rm Ker}(\widetilde{Sq}^0_*)_{n_r},
$$
where $(\widetilde{Sq}^0_*)_{n_r}: (QP_r)_{2n_r+r} \to (QP_r)_{n_r}$ denotes Kameko's homomorphism $\widetilde{Sq}^0_*$ in degree $2n_r+r$. Here, by convention, $\prod_{r+1 \leqslant i \leqslant k}(2^i-1) = 1$ for $r=k$.
\end{corl}
This corollary has been proved in \cite{su} for the case $d_{i-2} - d_{i-1} > i+1$ with  $3 \leqslant i \leqslant k$.

\medskip
Obviously $2n_r + r = \sum_{1\leqslant i \leqslant r-2}(2^{e_i}-1),$ where $e_i = d_i - d_{r-1} + 1$, for $1 \leqslant i \leqslant r - 2$. So, in degree $2n_r + r$ of $P_r$, there is a so-called spike $x = x_1^{2^{e_1}-1}x_2^{2^{e_2}-1}\ldots x_{r-2}^{2^{e_{r-2}}-1},$ i.e. a monomial whose exponents are all of the form $2^e-1$ for some $e$. Since the class $[x]$ in $(QP_k)_{2n_r+r}$ represented by the spike $x$ is nonzero and $\widetilde{Sq}^0_*([x]) = 0$, we have Ker$(\widetilde{Sq}^0_*)_{n_r} \ne 0$, for any $5 \leqslant r \leqslant k$. Therefore, by Corollary \ref{hhq}, Kameko's conjecture is not true in degree $n= 2n_k+k$ for any $k \geqslant 5$, where $n_k = 2^{d_1-1} + 2^{d_2-1} + \ldots + 2^{d_{k-2}-1} - k + 1$.

\medskip
This paper is organized as follows.
 In Section \ref{s2}, we recall  some needed information on the admissible monomials in $P_k$ and Singer's criterion on the hit monomials.  We prove Theorem \ref{dl1} in Section \ref{s3} by describing a basis of  $(QP_k)_n$  in terms of a given basis  of $(QP_{k-1})_m$. In Section \ref{s3a}, we recall the results on the hit problem for $k \leqslant 3$.
 Theorem \ref{dl3} will be proved in Section \ref{s4} by explicitly determining all of the admissible monomials in $P_4$.

\medskip
The first formulation of this paper was given in a 240-page 
preprint in 2007 \cite{su2}, which was then publicized to a remarkable
number of colleagues. One year latter, we found the negative answer
to Kameko's conjecture on the hit problem \cite{su1,su}.
Being led by the insight of this new study, we have remarkably reduced
the length of the paper.

The main results of the present paper have already been announced in \cite{su4}. However, we correct Theorem 3 in \cite{su4} by replacing the condition  $d_{k-1} \geqslant k-1 \geqslant 1$ with $d_{k-1} \geqslant k-1 \geqslant 3$.
\section{Preliminaries}\label{s2}
\setcounter{equation}{0}

In this section, we recall some results in Kameko ~\cite{ka} and Singer ~\cite{si2} which will be used in the next sections.

\begin{nota} Throughout the paper, we use the following notations.
\begin{align*}
\mathbb N_k &= \{1,2, \ldots , k\},\\
X_{\mathbb J} &= X_{\{j_1,j_2,\ldots , j_s\}} =
 \prod_{j\in \mathbb N_k\setminus \mathbb J}x_j , \ \ \mathbb J = \{j_1,j_2,\ldots , j_s\}\subset \mathbb N_k,
\end{align*}
In particular, we have
\begin{align*}
&X_{\mathbb N_k} =1,\\
&X_\emptyset = x_1x_2\ldots x_k, \\
&X_j = X_{\{j\}}= x_1\ldots \hat x_j \ldots x_k, \ 1 \leqslant j \leqslant k.
\end{align*}

Let $\alpha_i(a)$ denote the $i$-th coefficient in dyadic expansion of a non-negative integer $a$. That means
$a= \alpha_0(a)2^0+\alpha_1(a)2^1+\alpha_2(a)2^2+ \ldots ,$
for $ \alpha_i(a) =0$ or 1 and $i\geqslant 0.$ Denote by $\alpha(a)$ the number of 1's in dyadic expansion of $a$.

Let $x=x_1^{a_1}x_2^{a_2}\ldots x_k^{a_k} \in P_k$. Denote $\nu_j(x) = a_j, 1 \leqslant j \leqslant k$.  
Set $$\mathbb J_i(x) = \{j \in \mathbb N_k :\alpha_i(\nu_j(x)) =0\},$$ 
for $i\geqslant 0$. Then we have
$$x = \prod_{i\geqslant 0}X_{\mathbb J_i(x)}^{2^i}.$$ 

For a  polynomial $f$ in $P_k$, we denote by $[f]$ the class in $QP_k$ represented by $f$. For a subset $S \subset P_k$, we denote 
$$[S] = \{[f] : f \in S\} \subset QP_k.$$
\end{nota}
\begin{defn}
For a monomial  $x$,  define two sequences associated with $x$ by
\begin{align*} 
\omega(x)&=(\omega_1(x),\omega_2(x),\ldots , \omega_i(x), \ldots),\\
\sigma(x) &= (\nu_1(x),\nu_2(x),\ldots ,\nu_k(x)),
\end{align*}
where
$\omega_i(x) = \sum_{1\leqslant j \leqslant k} \alpha_{i-1}(\nu_j(x))= \deg X_{I_{i-1}(x)},\ i \geqslant 1.$

The sequence $\omega(x)$ is called  the weight vector of $x$ (see Wood~\cite{wo2}). The weight vectors and the sigma vectors can be ordered by the left lexicographical order. 

Let $\omega=(\omega_1,\omega_2,\ldots , \omega_i, \ldots)$ be a sequence of non-negative integers such that $\omega_i = 0$ for $i \gg 0$. Define $\deg \omega = \sum_{i > 0}2^{i-1}\omega_i$. 
Denote by   $P_k(\omega)$ the subspace of $P_k$ spanned by all monomials $y$ such that
$\deg y = \deg \omega$, $\omega(y) \leqslant \omega$, and by  $P_k^-(\omega)$ the subspace of $P_k$ spanned by all monomials $y \in P_k(\omega)$  such that $\omega(y) < \omega$. Denote by $\mathcal A^+_s$ the subspace of $\mathcal A$ spanned by all $Sq^j$ with $1\leqslant j < 2^s$.
\end{defn}
\begin{defn}\label{dfn2} Let $\omega$ be a sequence of non-negative integers and $f, g$ two polynomials  of the same degree in $P_k$. 

i) $f \equiv g$ if and only if $f - g \in \mathcal A^+P_k$. 

ii) $f \simeq_{(s,\omega)} g$ if and only if $f - g \in \mathcal A^+_sP_k+P_k^-(\omega)$. 
\end{defn}

Since $\mathcal A_0^+P_k = 0$, $f \simeq_{(0,\omega)} g$ if and only if $f - g \in P_k^-(\omega)$. If $x$ is a monomial in $P_k$ and $\omega = \omega(x)$, then we denote $x \simeq_{s}g$  if and only if $x \simeq_{(s,\omega(x))}g$. 

Obviously, the relations $\equiv$ and $\simeq_{(s,\omega)}$ are equivalence
relations.

We recall some relations on the action of the Steenrod squares on $P_k$.

\begin{prop}\label{mdcb1} Let $f$ be  a  polynomial in $P_k$. 

{\rm i)} If $i > \deg f$, then $Sq^i(f) =0$. If $i = \deg f$, then $Sq^i(f) =f^2$.

{\rm ii)} If $i$ is not divisible by $2^s$, then $Sq^i(f^{2^s}) = 0$ while $Sq^{r2^s}(f^{2^s}) = (Sq^r(f))^{2^s}$.
\end{prop}

\begin{prop}\label{mdcb4}  Let $x, y$ be  monomials and let $f, g $ be polynomials in $P_k$ such that $\deg x = \deg f $, $\deg y = \deg g$.

{\rm i)} If $\omega_i(x) \leqslant 1$ for $i > s$ and $x\simeq_{t}f$ with $t \leqslant s$, then $xy^{2^s}\simeq_{t}f y^{2^s}$. 

{\rm ii)} If  $\omega_i(x) = 0$ for $i > s $, $x\simeq_{s} f$ and $y \simeq_r g$, then $xy^{2^s} \simeq_{s+r} fg^{2^s}$.
\end{prop}
\begin{proof} Suppose that
\begin{equation}\label{ct00}
x+f + \sum_{1\leqslant u < 2^{t}}Sq^u(z_u) = h \in P_k^-(\omega(x)),
\end{equation}
where $z_u \in P_k$. From this and Proposition 2.4, we have $Sq^u(z_u) y^{2^s} = Sq^u(z_u y^{2^s})$ for $1\leqslant u < 2^t \leqslant 2^s$. Observe that $\omega_v(xy^{2^s}) = \omega_v(x)$ for $1\leqslant v \leqslant s$. If $z$ is a monomial and $z \in P_k^-(\omega(x))$, then there exists an index $i \geqslant 1$ such that $\omega_j(z) = \omega_j(x)$ for $j \leqslant i-1$ and $\omega_i(z) < \omega_i(x)$. If $i>s$, then $\omega_i(x)=1, \omega_i(z)=0$. Then we have
 $$\alpha_{i-1}\Big(\deg x - \sum_{1\leqslant j\leqslant i-1}2^{j-1}\omega_j(x)\Big) = \alpha_{i-1}\Big(2^{i-1}+ \sum_{j>i}2^{j-1}\omega_j(x)\Big) =1.$$

On the other hand, since $\deg x = \deg z,\ \omega_i(z)=0$ and $\omega_j(z)=\omega_j(x),$ for $j \leqslant i-1,$ one gets
\begin{align*}\alpha_{i-1}\Big(\deg x - \sum_{1\leqslant j\leqslant i-1}2^{j-1}\omega_j(x)\Big) &=\alpha_{i-1}\Big(\deg z  - \sum_{1\leqslant j\leqslant i-1}2^{j-1}\omega_j(z)\Big)\\ &=  \alpha_{i-1}\Big(\sum_{j>i}2^{j-1}\omega_j(z)\Big) =0.
\end{align*}
This is a contradiction. Hence, $1 \leqslant i \leqslant s.$

From the above equalities and the fact that  $h \in P_k^-(\omega(x))$, one gets
$$
xy^{2^s} + fy^{2^s} + \sum_{1\leqslant i < 2^{t}}Sq^i(z_iy^{
2^s}) = hy^{2^s} \in P_k^-(\omega(xy^{2^s})).
$$
The first part of the proposition is proved.

Suppose that $y+g + \sum_{1\leqslant j < 2^r}Sq^j(u_j) = h_1 \in P_k^-(\omega(y)),$ where $u_j \in P_k$. Then 
$$xy^{2^s} =  xg^{2^s} + xh_1^{2^s} + \sum_{1\leqslant j < 2^r}xSq^{j2^s}(u_j^{2^s}).$$
Since $\omega_i(x) = 0$ for $i>s$ and $h_1 \in P_k^-(\omega(y))$, we get $xh_1^{2^s} \in P_k^-(\omega(xy^{2^s}))$. Using the Cartan formula and Proposition \ref{mdcb1}, we obtain
$$xSq^{j2^s}(u_j^{2^s}) = Sq^{j2^s}(xu_j^{2^s}) + \sum_{0< b  \leqslant j}Sq^{b2^s}(x)(Sq^{j-b}(u_j))^{2^s}.$$
Since $\omega_i(x) = 0$ for $i>s$, we have $x = \prod_{0\leqslant i<s}X_{\mathbb J_i(x)}^{2^i}$. Using the Cartan formula and Proposition \ref{mdcb1}, we see that $Sq^{b2^s}(x)$ is a sum of polynomials of the form
$$  \prod_{0\leqslant i<s}(Sq^{b_i}(X_{\mathbb J_i(x)}))^{2^i},$$
where $\sum_{0\leqslant i<s}b_i2^i = b2^s$ and $0\leqslant b_i \leqslant \deg X_{\mathbb J_i(x)}$. Let $\ell$ be the smallest index  such that $b_\ell >0$ with $0\leqslant \ell <s$. Suppose that a monomial $z$ appears as a term of the polynomial $\Big(\prod_{0\leqslant i<s}(Sq^{b_i}(X_{\mathbb J_i(x)}))^{2^i}\Big)(Sq^{j-b}(u_j))^{2^s}$. Then $\omega_u(z) = \deg X_{\mathbb J_{u-1}}(x) = \omega_u(x) = \omega_u(xy^{2^s})$ for $u \leqslant \ell,$ and $\omega_{\ell+1}(z) = \deg X_{\mathbb J_\ell(x)} - b_\ell < \deg X_{\mathbb J_\ell(x)} = \omega_{\ell + 1}(x) =  \omega_{\ell + 1}(xy^{2^s})$. Hence, 
$$\Big(\prod_{0\leqslant i<s}(Sq^{b_i}(X_{\mathbb J_i(x)}))^{2^i}\Big)(Sq^{j-b}(u_j))^{2^s}\in P_k^-(\omega(xy^{2^s})).$$
This implies  $Sq^{b2^s}(x)(Sq^{j-b}(u_j))^{2^s}\in P_k^-(\omega(xy^{2^s}))$ for $0< b \leqslant j$.  So, one gets
$$xy^{2^s}+xg^{2^s} + \sum_{1\leqslant j < 2^r}Sq^{j2^s}(xu_j^{2^s}) \in P_k^-(\omega(xy^{2^s})).$$
Since $1 \leqslant j2^s < 2^{r+s}$ for $1 \leqslant j < 2^r$, we obtain $xy^{2^s} \simeq_{r+s} xg^{2^s}$.  

Since $\omega_i(x) = 0$ for $i>s$ and $h \in P_k^-(\omega(x))$, we have $hg^{2^s} \in P_k^-(\omega(xy^{2^s})) $. Using Proposition \ref{mdcb1}, the Cartan formula and the relation (\ref{ct00}) with $t=s$, we get
$$xg^{2^s} + fg^{2^s} +  \sum_{1\leqslant u < 2^{s}}Sq^{u}(z_ug^{2^s}) = hg^{2^s} \in P_k^-(\omega(xy^{2^s})).$$

 Combining the above equalities gives $xy^{2^s} + fg^{2^s} \in \mathcal A_{r+s}P_k+P_k^-(\omega(xy^{2^s}))$. This implies $xy^{2^s} \simeq_{r+s} fg^{2^s}$. The proposition follows.
\end{proof}

\begin{defn}\label{defn3} 
Let $x, y$ be monomials of the same degree in $P_k$. We say that $x <y$ if and only if one of the following holds:  

i) $\omega (x) < \omega(y)$;

ii) $\omega (x) = \omega(y)$ and $\sigma(x) < \sigma(y).$
\end{defn}

\begin{defn}
A monomial $x$ is said to be inadmissible if there exist monomials $y_1,y_2,\ldots, y_t$ such that $y_j<x$ for $j=1,2,\ldots , t$ and $x - \sum_{j=1}^ty_j \in \mathcal A^+P_k.$ 

A monomial $x$ is said to be admissible if it is not inadmissible.
\end{defn}

Obviously, the set of all the admissible monomials of degree $n$ in $P_k$ is a minimal set of $\mathcal{A}$-generators for $P_k$ in degree $n$. 

\begin{defn} 
 A monomial $x$ is said to be strictly inadmissible if and only if there exist monomials $y_1,y_2,\ldots, y_t$ such that $y_j<x,$ for $j=1,2,\ldots , t$ and $x - \sum_{j=1}^t y_j \in \mathcal A_s^+P_k$ with $s = \max\{i \ :\ \omega_i(x) > 0\}$.
\end{defn}

It is easy to see that if $x$ is strictly inadmissible, then it is inadmissible.
The following theorem is a modification of a result in \cite{ka}.

\begin{thm}[Kameko \cite{ka}, Sum \cite{su}]\label{dlcb1}  
 Let $x, y, w$ be monomials in $P_k$ such that $\omega_i(x) = 0$ for $i > r>0$, $\omega_s(w) \ne 0$ and   $\omega_i(w) = 0$ for $i > s>0$.

{\rm i)}  If  $w$ is inadmissible, then  $xw^{2^r}$ is also inadmissible.

{\rm ii)}  If $w$ is strictly inadmissible, then $xw^{2^r}y^{2^{r+s}}$ is inadmissible.
\end{thm}

\begin{prop}[\cite{su}]\label{mdcb3} Let $x$ be an admissible monomial in $P_k$. Then we have

{\rm i)} If there is an index $i_0$ such that $\omega_{i_0}(x)=0$, then $\omega_{i}(x)=0$ for all $i > i_0$.

{\rm ii)} If there is an index $i_0$ such that $\omega_{i_0}(x)<k$, then $\omega_{i}(x)<k$ for all $i > i_0$.
\end{prop}

Now, we recall a result in \cite{si2} on the hit monomials in $P_k$. 

\begin{defn}  A monomial $z$ in $P_k$   is called a spike if $\nu_j(z)=2^{s_j}-1$ for $s_j$ a non-negative integer and $j=1,2, \ldots , k$. 
If $z$ is a spike with $s_1>s_2>\ldots >s_{r-1}\geqslant s_r>0$ and $s_j=0$ for $j>r,$ then it is called a minimal spike.
\end{defn}

The following is a criterion for the hit monomials in $P_k$.

\begin{thm}[Singer~\cite{si2}]\label{dlsig} Suppose $x \in P_k$ is a monomial of degree $n$, where $\mu(n) \leqslant k$. Let $z$ be the minimal spike of degree $n$. If $\omega(x) < \omega(z)$, then $x$ is hit.
\end{thm}

From this theorem, we see that if $z$ is a minimal spike, then $P_k^-(\omega(z)) \subset \mathcal A^+P_k$.

The following lemma has been proved in \cite{su}.
\begin{lem}[\cite{su}]\label{bdbs1} Let $n =\sum_{1 \leqslant i \leqslant k-1}(2^{d_i}-1)$ with $d_i$ positive integers such that $d_1 > d_2 > \ldots >d_{k-2} \geqslant d_{k-1}>0,$ and let $x$ be a monomial of degree $n$ in $P_k$. If $[x] \ne 0$, then $\omega_i(x) = k-1$ for $1 \leqslant i \leqslant d_{k-1}$.
\end{lem}
 The following is a modification of a result in \cite{su}.
\begin{lem}\label{bdcb3} 
Let $n$ be as in Lemma \ref{bdbs1} and let $\omega = (\omega_1,\omega_2, \ldots)$ be a sequence of non-negative integers such that $\omega_i = k-1,$ for $1 \leqslant i \leqslant s \leqslant d_{k-1}$, $\omega_i \leqslant 1$ for $i > s$, $\omega_i =0$ for $i \gg 0$, and $\deg \omega < n$.  Suppose  $f, g, h, p \in P_k$ with $\deg f = \deg g = \deg\omega$,  $\deg h = \deg p =  (n - \deg\omega)/2^s = \sum_{i=1}^{k-1}(2^{d_i-s}-1) - \sum _{j \geqslant 1}2^{j-1}\omega_{s+j}$.

\smallskip
{\rm i)} If $f \simeq_{(s,\omega)} g$, then $fh^{2^s} \equiv gh^{2^s}$.

{\rm ii)} If $\omega_i =0$ for $i>s$, and $h \equiv p$, then $fh^{2^s} \equiv fp^{2^s}$. 
\end{lem}
This lemma can easily be proved by using Proposition  \ref{mdcb4}, Theorem \ref{dlsig} and  Lemma \ref{bdbs1}.

For latter use, we set 
\begin{align*} 
P_k^0 &=\langle\{x=x_1^{a_1}x_2^{a_2}\ldots x_k^{a_k} \ : \ a_1a_2\ldots a_k=0\}\rangle,
\\ P_k^+ &= \langle\{x=x_1^{a_1}x_2^{a_2}\ldots x_k^{a_k} \ : \ a_1a_2\ldots a_k>0\}\rangle. 
\end{align*}

It is easy to see that $P_k^0$ and $P_k^+$ are the $\mathcal{A}$-submodules of $P_k$. Furthermore, we have the following.

\begin{prop}\label{2.7} We have a direct summand decomposition of the $\mathbb F_2$-vector spaces
$$QP_k =QP_k^0 \oplus  QP_k^+.$$
Here $QP_k^0 = P_k^0/\mathcal A^+.P_k^0$ and  $QP_k^+ = P_k^+/\mathcal A^+.P_k^+.$
\end{prop}

\section{Proof of Theorem \ref{dl1}}\label{s3}
\setcounter{equation}{0}

We denote
$$\mathcal N_k =\{(i;I)\ :\ I=(i_1,i_2,\ldots,i_r),1 \leqslant  i < i_1 <  \ldots < i_r\leqslant  k,\ 0\leqslant r <k\}.$$
Let $(i;I) \in \mathcal N_k$ and $j \in \mathbb N_k$. Denote  by $r = \ell(I)$ the length of $I$, and
$$I\cup j = \begin{cases}I,&  \text { if  } j \in I,\\  
 (j,i_1,\ldots ,i_r), &\text { if  } 0<j<i_1, \\
(i_1,\ldots ,i_{t-1},j,i_t,\ldots ,i_r), &\text { if  } i_{t-1}<j<i_t, \ 2 \leqslant t \leqslant r+1.
 \end{cases}$$
Here $i_{r+1} = k+1$. 
For $1 \leqslant i \leqslant k$, define the homomorphism $f_i= f_{k;i}: P_{k-1} \to P_k$ of algebras by substituting
$$f_i(x_j) = \begin{cases} x_j, &\text{ if } 1 \leqslant j <i,\\
x_{j+1}, &\text{ if } i \leqslant j <k.
\end{cases}$$
\begin{defn} Let $(i;I) \in \mathcal N_k$, let $r = \ell(I)$, and let $u$
be an integer with $1 \leqslant  u \leqslant r$. A monomial $x \in P_{k-1}$ is said to be $u$-compatible with $(i;I)$ if all of the following hold:

\medskip
i) $\nu_{i_1-1}(x)= \nu_{i_2-1}(x)= \ldots = \nu_{i_{(u-1)}-1}(x)=2^{r} - 1$,

ii) $\nu_{i_u-1}(x) > 2^{r} - 1$,

iii) $\alpha_{r-t}(\nu_{i_u-1}(x)) = 1,\ \forall t,\ 1 \leqslant t \leqslant  u$,

iv) $\alpha_{r-t}(\nu_{i_t-1} (x)) = 1,\ \forall t,\ u < t \leqslant r$.
\end{defn}

Clearly,  a monomial $x \in P_k$ can be $u$-compatible with a given $(i;I) \in \mathcal N_k$, $r =\ell(I) >0,$ for at most one value of $u$. By convention, $x$ is $1$-compatible with $(i;\emptyset)$.

\begin{defn}\label{dfn1} Let $(i;I) \in \mathcal N_k$. Denote $x_{(I,u)} = x_{i_u}^{2^{r-1}+\ldots + 2^{r-u}}\prod_{u< t \leqslant r}x_{i_t}^{2^{r-t}}$ for $1\leqslant u \leqslant r = \ell(I)$, and $x_{(\emptyset,1)} = 1$.  For a monomial $x$ in $P_{k-1}$, 
we define the monomial  $\phi_{(i;I)}(x)$ in $P_k$ by setting
$$ \phi_{(i;I)}(x) = \begin{cases} (x_i^{2^r-1}f_i(x))/x_{(I,u)}, &\text{if there exists $u$ such that}\\
&\text{$x$ is $u$-compatible with $(i, I)$,}\\
0, &\text{otherwise.}
\end{cases}$$

Then we have an $\mathbb F_2$-linear map $\phi_{(i;I)}:P_{k-1}\to P_k$.
In particular, $\phi_{(i;\emptyset)} = f_i$.  

For example, let $I=(j)$ and $1 \leqslant i <j \leqslant k$. A monomial $x\in P_{k-1}$ is $1$-compatible with $(i;I)$ if and only if $\alpha_0(\nu_{j-1}(x)) =1$ and $\phi_{(i;I)}(x) = (x_if_i(x))/x_j$.

Let $k=4$ and $I = (2,3,4)$. The monomial $x=x_1^{12}x_2^6x_3^9$ is  $1$-compatible with $(1;I)$; $y = x_1^7x_2^{15}x_3^7$ is $2$-compatible with $(1;I)$; $z = x_1^7x_2^7x_3^{15}$  is $3$-compatible with $(1;I)$ and $\phi_{(1;I)}(x) = x_1^7x_2^8x_3^4x_4^8$, $\phi_{(1;I)}(y) = x_1^7x_2^7x_3^9x_4^6$, $\phi_{(1;I)}(z) = x_1^7x_2^7x_3^7x_4^8$.

Let  $x = X^{2^d-1}y^{2^d}$, with $y$ a monomial in $P_{k-1}$ and $X= x_1x_2\ldots , x_{k-1}\in P_{k-1}$. 

If $r < d$, then $x$ is $1$-compatible with $(i;I)$ and
\begin{equation}\label{rela22}\phi_{(i;I)}(x) = \phi_{(i;I)}(X^{2^{d}-1})f_i(y)^{2^d} =  x_i^{2^r-1}\prod_{1 \leqslant t \leqslant r}x_{i_t}^{2^d-2^{r-t}-1} X_{i,i_1,\ldots, i_r}^{2^d-1}f_i(y)^{2^d}.
\end{equation}

If $d = r$, $\nu_{j-1}(y) = 0, j = i_1, i_2, \ldots , i_{u-1}$ and $\nu_{i_u-1}(y) > 0$,  then $x$ is $u$-compatible with $(i;I)$ and
\begin{equation}\label{rela2}\phi_{(i;I)}(x) = \phi_{(i_u;J_u)}(X^{2^{d}-1})f_{i}(y)^{2^d} ,
\end{equation}
where $J_u = (i_{u+1},\ldots , i_r)$. 
\end{defn}

Let $B$ be a finite subset of $P_{k-1}$ consisting of  some  polynomials in degree $n$. We set
\begin{align*}\Phi^0(B) &= \bigcup_{1\leqslant i \leqslant k}\phi_{(i;\emptyset)}(B) = \bigcup_{1\leqslant i \leqslant k}f_i(B).\\
\Phi^+(B) &= \bigcup_{(i;I)\in \mathcal N_k, 0<\ell(I)\leqslant k-1}\phi_{(i;I)}(B)\setminus P_k^0.\\
\Phi(B) &= \Phi^0(B)\bigcup \Phi^+(B). 
\end{align*} 

 It is easy to see that if $B_{k-1}(n)$ is a minimal set of generators for $\mathcal A$-module $P_{k-1}$ in degree $n$, then  $\Phi^0(B_{k-1}(n))$  is a minimal set of generators for $\mathcal A$-module $P_k^0$ in degree $n$ and $\Phi^+(B_{k-1}(n)) \subset P_k^+$.

\begin{prop}\label{mdc1} Let $n =\sum_{1 \leqslant i \leqslant k-1}(2^{d_i}-1)$ with $d_i$ positive integers such that $d_1 > d_2 > \ldots >d_{k-2} \geqslant d_{k-1} \geqslant k-1 \geqslant 3$. If $B_{k-1}(n)$ is a minimal set of generators for  $\mathcal A$-module $P_{k-1}$ in degree $n$, then $B_k(n) = \Phi(B_{k-1}(n))$ is also a minimal set of generators for  $\mathcal A$-module $P_k$ in degree $n$. 
\end{prop}

For $d_{k-1}\geqslant k$, this proposition is a modification of a result in Nam \cite{na}. For $d_{k-2} = d_{k-1} > k$, it has been proved in \cite{su}.

\medskip
The proposition will be proved in Subsection \ref{sub31}. We need some lemmas. 

\begin{lem}\label{hq0} Let $j_0, j_1, \ldots , j_{d-1} \in \mathbb N_k$, with $d$ a positive integer and let $i = \min\{j_0,\ldots , j_{d-1}\}$, $I = (i_1, \ldots, i_r)$ with $\{i_1, \ldots, i_r\} = \{j_0,\ldots , j_{d-1}\}\setminus \{i\}$. Then,  
$$x := \prod_{0 \leqslant t <d}X_{j_t}^{2^t} \simeq_{d-1} \phi_{(i;I)}(X^{2^d-1}).$$
\end{lem}

\begin{lem}\label{bdcbs}Let $n =\sum_{1 \leqslant i \leqslant k-1}(2^{d_i}-1)$ with $d_i$ positive integers such that $d_1 > d_2 > \ldots >d_{k-2} \geqslant d_{k-1}=d > 0, m =  \sum_{1 \leqslant i \leqslant k-2}(2^{d_i-d}-1),$  and let $y_0$ be   a monomial in $(P_k)_{m-1}$, $y_j = y_0x_j$ for $1\leqslant j \leqslant k$, and $(i;I) \in \mathcal N_k$.

\medskip
{\rm i)} If \ $r = \ell(I) < d$, then
$$ \phi_{(i;I)}(X^{2^{d}-1})y_i^{2^d} \equiv 
 \sum _{1\leqslant j < i} \phi_{(j;I)}(X^{2^{d}-1})y_j^{2^d} + \sum _{i < j \leqslant k}\phi_{(t_j;I^ {(j)})}(X^{2^{d}-1})y_j^{2^d},$$
where $t_j = \min(j,I)$,  and $I^ {(j)} = (I\cup j)\setminus \{t_j\}$ for $j > i$.

\medskip
{\rm ii)} If \ $r +1<d$, then
$$ \phi_{(i;I)}(X^{2^{d}-1})y_i^{2^d} \equiv 
 \sum _{1\leqslant j < i} \phi_{(j;I\cup i)}(X^{2^{d}-1})y_j^{2^d} + \sum _{i < j \leqslant k}\phi_{(i;I\cup j)}(X^{2^{d}-1})y_j^{2^d}.$$
\end{lem}

Denote  $I_t =(t+1, t+2, \ldots, k)$ for $1 \leqslant t  \leqslant k$. Set  
$$Y_t = \sum_{u=t}^k\phi_{(t;I_t)}(X^{2^{d}-1})x_u^{2^{d}},\  d \geqslant k-t+1.$$ 

\begin{lem}\label{hq4} For $1 \leqslant t \leqslant k$, 
$$Y_t \simeq_{(k-t+1,\omega)} \sum_{(j;J)}\phi_{(j;J)}(X^{2^d-1})x_j^{2^d},$$ 
where the sum runs over all $(j;J)\in \mathcal N_k$ with $1 \leqslant j < t$, $J \subset I_{t-1}$, $J \ne I_{t-1}$ and $\omega=\omega(X_1^{2^d-1}x_1^{2^d})$. 
\end{lem}

We assume that all elements of $B_{k-1}(n)$ are monomials. Denote  $\mathcal B = B_{k-1}(n)$. We set
\begin{align*}
\mathcal C &=\{z \in \mathcal B : \nu_1(z) > 2^{k-1}-1\},\\
\mathcal D &=\{z \in \mathcal B : \nu_1(z) = 2^{k-1}-1, \nu_2(z) > 2^{k-1}-1 \},\\
\mathcal E &=\{z \in \mathcal B : \nu_1(z) = \nu_2(z) = 2^{k-1}-1 \}.
\end{align*}

Since $\omega_k(z) \geqslant k-3$ for all $z \in \mathcal B$, we have $\mathcal B = \mathcal C\cup \mathcal D \cup \mathcal E$.  If $d= d_{k-1}> k-1$, then $\mathcal D = \mathcal E = \emptyset.$  If $d_{k-2} > d_{k-1} = k-1$, then $\mathcal E = \emptyset.$ 

We set 
$\bar {\mathcal B} = \{\bar z \in P_{k-1} : X^{2^d-1}\bar z^{2^d} \in \mathcal B\}.$
If either $d \geqslant k$ or $I\ne I_1$, then $\phi_{(i;I)}(z) = \phi_{(i;I)}(X^{2^d-1})f_i(\bar z)^{2^d}$. If $d = d_{k-1} = k-1$, then using the relation (\ref{rela2}), we have
\begin{equation}\label{ctth}\phi_{(1;I_1)}(z) = \begin{cases}\phi_{(2;I_2)}(X^{2^d-1})f_1(\bar z)^{2^d}, &\text{ if }  z \in \mathcal C,\\
\phi_{(3;I_3)}(X^{2^d-1})f_2(\bar z)^{2^d}, &\text{ if }  z \in \mathcal D,\\
\phi_{(4;I_4)}(X^{2^d-1})f_3(\bar z)^{2^d}, &\text{ if }  z \in \mathcal E.
\end{cases}\end{equation}

For any $(i;I) \in \mathcal N_k$, we define the homomorphism $p_{(i;I)}: P_k \to P_{k-1}$ of algebras by substituting
$$p_{(i;I)}(x_j) =\begin{cases} x_j, &\text{ if } 1 \leqslant j < i,\\
\sum_{s\in I}x_{s-1}, &\text{ if }  j = i,\\  
x_{j-1},&\text{ if } i< j \leqslant k.
\end{cases}$$
Then $p_{(i;I)}$ is a homomorphism of $\mathcal A$-modules.  In particular, for $I =\emptyset$, we have $p_{(i;\emptyset)}(x_i)= 0$. 

\begin{lem}\label{bdc1} Let $z \in \mathcal B$ and $(i;I), (j;J) \in \mathcal N_k$ with $\ell(J) \leqslant \ell (I)$.

\medskip
{\rm i)} If either $d \geqslant k$ or $d = k -1$ and $I\ne I_1$, then
$$p_{(j;J)}(\phi_{(i;I)}(z)) \equiv \begin{cases}z, &\text{if } (j;J) = (i;I),\\
0,  &\text{if } (j;J) \ne (i;I).\end{cases}$$

{\rm ii)} If $z \in \mathcal C$ and $d = k -1$,  then
$$p_{(i;I)}(\phi_{(1;I_1)}(z)) \equiv \begin{cases}z, &\text{if } (i;I) = (1;I_1), \\
0 \ \text{\rm mod} \langle \mathcal D \cup \mathcal E \rangle, &\text{if } (i;I) = (2;I_2),\\
0,  &\text{otherwise } .
\end{cases}$$

{\rm iii)} If $z \in \mathcal D$ and $d = k -1$,  then
$$p_{(i;I)}(\phi_{(1;I_1)}(z)) \equiv \begin{cases}z, &\text{if } (i;I) = (1;I_1), (1;I_2), (2;I_2),\\
0 \ \text{\rm mod} \langle \mathcal E \rangle, &\text{if } (i;I) = (3;I_3),\\
0,  &\text{otherwise } .\end{cases}$$

{\rm iv)} If $z \in \mathcal E$ and $d = k -1$,  then
$$p_{(i;I)}(\phi_{(1;I_1)}(z)) \equiv \begin{cases}z &\text{if } I_3 \subset I,\\
0,  &\text{otherwise}.\end{cases}$$
\end{lem}

The above lemmas will be proved in Subsections \ref{sub32}-\ref{sub34}. In particular, for $d >k$, the first part of Lemma \ref{bdc1} has been proved in \cite{su}.

\begin{lem}[Nam \cite{na}]\label{bdmu} 
Let $x$ be a monomial in $P_k$. Then $x \equiv \sum\bar x$, where  $\bar x$ are  monomials with $\nu_1(\bar x) = 2^t-1$ and $t = \alpha(\nu_1(x))$. 
\end{lem}

Now, based on Proposition \ref{mdc1} and Lemma \ref{bdc1}, we prove Theorem \ref{dl1}.

\begin{proof}[Proof of Theorem \ref{dl1}] Denote by $|S|$ the cardinal of a set $S$. It is easy to check that $|\mathcal N_k| = 2^k-1$. Let $(i;I), (j;J) \in \mathcal N_k$ with $\ell(J) \leqslant \ell(I)$ and $y, z \in B_{k-1}(n).$ Suppose that $\phi_{(j;J)}(y) = \phi_{(i;I)}(z)$. Using Lemma \ref{bdc1}, we have $y \equiv p_{(j;J)}(\phi_{(i;I)}(z))\not \equiv 0$. From this, Lemma \ref{bdc1} and the relation (\ref{ctth}), we get $y = z$ and $(i;I) = (j;J)$. Hence,
$$\phi_{(i;I)}(B_{k-1}(n))\cap \phi_{(j;J)}(B_{k-1}(n)) = \emptyset.$$ 
for $(i;I) \ne (j;J)$ and $|\phi_{(i;I)}(B_{k-1}(n))| = |B_{k-1}(n)|$.  From Proposition \ref{mdc1}, we have
\begin{align*}\dim (QP_k)_n &= |B_k(n)| = \sum_{(i;I)\in \mathcal N_k}|B_{k-1}(n)|\\
&= |\mathcal N_k|\dim (QP_{k-1})_n\\
&= (2^k-1)\dim (QP_{k-1})_n.
\end{align*}
Set $h_u = 2^{d_1-u}+\ldots + 2^{d_{k-2}-u} + 2^{d_{k-1}-u}-k+1,$ for $0\leqslant u \leqslant d$. We have $h_0 = n,\ h_d = m,\ 2h_{u}+k-1 = h_{u-1}$ and $\mu(2h_{u}+k-1) = k-1$ for $1\leqslant u \leqslant d$. By Theorem \ref{dlmd2}, Kameko's homomorphism $(\widetilde{Sq}^0_*)_{h_u}: (QP_{k-1})_{h_{u-1}} \to (QP_{k-1})_{h_{u}}$ is an isomorphism of the $\mathbb F_2$-vector spaces. So, the iterated homomorphism
$$(\widetilde{Sq}^0_*)^d = (\widetilde{Sq}^0_*)_{h_d}\ldots (\widetilde{Sq}^0_*)_{h_1}: (QP_{k-1})_{n} \to (QP_{k-1})_{m}$$
is an isomorphism of the $\mathbb F_2$-vector spaces. Hence, 
$\dim (QP_{k-1})_n = \dim (QP_{k-1})_m.$ The theorem is proved.  
\end{proof}

In the remaining part of the section, we prove Proposition \ref{mdc1} and Lemmas \ref{hq0}-\ref{bdc1}.

\subsection{Proof of Proposition \ref{mdc1}}\label{sub31}\
\setcounter{exa}{3}

\medskip
Denote by $\mathcal P_k(n)$ the subspace of $(P_k)_n$ spanned by all elements of the set $B_k(n)$.
Let $x$ be a monomial of degree $n$ in $P_k$ and $[x] \ne 0$. By Lemma \ref{bdbs1}, we have $\omega_i(x) = k-1$ for $1 \leqslant i \leqslant d_{k-1}= d$. Hence, we obtain
$x = \Big(\prod_{0 \leqslant t < d}X_{j_t}^{2^t}\Big)\bar y^{2^d},$
for suitable monomial $\bar y \in (P_k)_m$, with $m =  \sum_{1 \leqslant i \leqslant k-2}(2^{d_i-d}-1)$.  

According to Lemmas \ref{hq0} and \ref{bdcb3}, there is $(i;I) \in \mathcal N_k$ such that
\begin{equation*}x = \Big(\prod_{0 \leqslant t < d}X_{j_t}^{2^t}\Big)\bar y^{2^d}\equiv \phi_{(i;I)}(X^{2^d-1})\bar y^{2^d},
\end{equation*}
where $r = \ell(I) < d$. 

Now, we prove $[x]\in [\mathcal P_k(n)]$. The proof is divided into many cases.
 
\begin{khi}\label{th01}
 If  $\bar y = f_i(y)$ with $y \in (P_{k-1})_m$, then $[x]\in [\mathcal P_k(n)]$.
\end{khi} 
Since the iterated homomorphism
$(\widetilde{Sq}^0_*)^d : (QP_{k-1})_{n} \to (QP_{k-1})_{m}$
is an isomorphism of the $\mathbb F_2$-vector spaces,
$$\bar{\mathcal B} = (\widetilde{Sq}^0_*)^d(B_{k-1}(n)) = \{\bar z \in (P_{k-1})_m : X^{2^d-1}\bar z^{2^d} \in B_{k-1}(n)\}$$
is a minimal set of $\mathcal A$-generators for $P_{k-1}$ in degree $m$.

 Since $y \in (P_{k-1})_m$,  we have $y \equiv \bar z_1+ \bar z_2 + \ldots + \bar z_s$ with $\bar z_t$  monomials in $\bar{\mathcal B}$. Using Lemma \ref{bdcb3}, we have
$$x\equiv  \phi_{(i;I)}(X^{2^d-1})f_i(y)^{2^d}\equiv \sum_{1\leqslant t \leqslant s} \phi_{(i;I)}(X^{2^d-1})f_i(\bar z_t)^{2^d}.$$

 Since $\phi_{(i;I)}(X^{2^d-1})f_i(\bar z_t)^{2^d}= \phi_{(i;I)}(X^{2^d-1}\bar z_t^{2^d})$ and $X^{2^d-1}\bar z_t^{2^d} \in B_{k-1}(n)$, we get $[x]\in [\mathcal P_k(n)]$. 

\begin{khi}\label{th02} If $d \geqslant k$,  then $[x]\in [\mathcal P_k(n)]$ for all $\bar y \in (P_{k-1})_{m}$.
\end{khi}
We have $\bar y = x_i^af_i(y)$ with $a= \nu_i(\bar y)$ and $y = p_{(i;\emptyset)}(\bar y/x_i^a) \in (P_{k-1})_{m-a}$.  The proof proceeds by double induction on $(i,a)$. If $a=0$, then by Case \ref{th01}, $[x]\in [\mathcal P_k(n)]$ for any $i$. Suppose that $a > 0$ and this case is true for $a-1$ and any $i$.

 If  $i = 1$ and either $I \ne I_1$ or $d>k$, then $d-r-1\geqslant 1$. Applying Lemma \ref{bdcbs}(ii) with $y_0 = x_1^{a-1}f_1(y)$, we get
$$
x \equiv \sum_{2\leqslant j \leqslant k}\phi_{(1;I\cup j)}(X^{2^{d}-1})(x_1^{a-1}f_1(x_{j-1}y))^{2^{d}}.
$$
From this and the inductive hypothesis, we obtain $[x] \in [\mathcal P_k(n)]$. 

If  $I = I_1$ and $d = k$, then $r = d-1$. Using Lemma \ref{bdcbs}(i) with $y_0 = x_1^{a-1}f_1(y)$ and Lemmas \ref{hq4}, \ref{bdcb3}, we get
\begin{align*}
x &\equiv \sum_{j =2}^k\phi_{(2;I_2)}(X^{2^{k}-1})(x_jy_0)^{2^{k}}= Y_2y_0^{2^k}
\equiv \sum\phi_{(1;J)}(X^{2^k-1})(x_1^af_1(y))^{2^k},
\end{align*}
where the last sum runs over all $J \ne I_1$. Hence, $[x] \in [\mathcal P_k(n)]$.

Suppose  $i > 1$ and assume this case has been proved  in  the  subcases $1,2, \ldots  , i-1$. Then,  $r + 1 \leqslant k-i + 1< k \leqslant d$. Applying Lemma \ref{bdcbs}(ii) with $y_0 = x_i^{a-1}f_i(y)$, we obtain
\begin{align*} 
x \equiv \sum _{1\leqslant j < i} \phi_{(j;I\cup i)}(X^{2^{d}-1})y_j^{2^d} + \sum _{i < j \leqslant k}\phi_{(i;I\cup j)}(X^{2^{d}-1})(x_i^{a-1}f_i(x_{j-1}y)^{2^d}.
\end{align*}
Using the inductive hypothesis, we have $\phi_{(j;I\cup i)}(X^{2^{d}-1})y_j^{2^d} \in [\mathcal P_k(n)]$ for $j<i$, and 
$\phi_{(i;I\cup j)}(X^{2^{d}-1})(x_i^{a-1}f_i(x_{j-1}y)^{2^d}\in [\mathcal P_k(n)] $ for $j>i$. Hence, $[x] \in [\mathcal P_k(n)]$.

\medskip
So, the proposition is proved for $d \geqslant k$.
In the remaining part of the proof, we assume that $d = k - 1$. 

\begin{khi}\label{th03} If\ $I=I_i$, $\bar y = f_{i-1}(y)$ with $y \in (P_{k-1})_{m}$, $\nu_j(y)=0$ for $j \leqslant i-2$, and $2\leqslant i \leqslant4$, then  $[x]\in [\mathcal P_k(n)]$.
\end{khi}
 Since $y \in (P_{k-1})_m$,  we have $y \equiv \bar z_1+ \bar z_2 + \ldots + \bar z_s$ with $\bar z_t$  monomials in $\bar{\mathcal B}$ and $\nu_j(\bar z_t)=0$ for $j\leqslant i-2$. Using Lemma \ref{bdcb3}, we get
$$x\equiv  \phi_{(i;I_i)}(X^{2^d-1})f_{i-1}(y)^{2^d}\equiv \sum_{1\leqslant t \leqslant s} \phi_{(i;I_i)}(X^{2^d-1})f_{i-1}(\bar z_t)^{2^d}.$$

 If $\nu_{i-1}(\bar z_t) > 0$, then $\phi_{(i;I_i)}(X^{2^d-1})f_{i-1}(\bar z_t)^{2^d}= \phi_{(1;I_1)}(X^{2^d-1}\bar z_t^{2^d})$. 
If $\nu_{i-1}(\bar z_t) = 0$, then $f_{i-1}(\bar z_t) = f_i(\bar z_t)$ and $\phi_{(i;I_i)}(X^{2^d-1})f_{i-1}(\bar z_t)^{2^d}= \phi_{(i;I_i)}(X^{2^d-1}\bar z_t^{2^d})$. Hence, $[x]\in [\mathcal P_k(n)]$. 

\begin{khi}\label{th1}  
If $\bar y = x_i^{2^s}f_i(y)$ with $y \in (P_{k-1})_{m-2^s}$, $\nu_j(y) = 0$ for $j<i$, $r=\ell(I) < k-i-1$ and $i\leqslant 2$, then $[x] \in [\mathcal P_k(n)]$. 
\end{khi}

According to Lemmas \ref{bdmu} and \ref{bdcb3}, $x_i^{2^s}f_i(y)^{2^{d}}\equiv \sum x_if_i(z)$, where the sum runs over some monomials $z \in P_{k-1}$ with $\nu_j(z)=0$, $j<i$. So, by using Lemma \ref{bdcb3}, we can assume $s = 0$. 

Let $i=1$. Since $r+1 < k-1 =d$, using Lemma \ref{bdcbs}(ii) with $y_0 = f_1(y)$, we have
$$ x \equiv \sum_{u=2}^k \phi_{(1;I\cup u)}(X^{2^{d}-1})(f_1(x_{u-1}y))^{2^{d}}.$$
Hence,  by Case \ref{th01},  $[x] \in [\mathcal P_k(n)]$.

Let $i=2$. Since $r+1 < k-2 < d$, using Lemma \ref{bdcbs}(ii) with $y_0 = f_2(y)$, one gets
$$ x \equiv \phi_{(1;I\cup 2)}(X^{2^{d}-1})(x_{1}f_1(y))^{2^{d}} + \sum_{u=3}^k \phi_{(2;I\cup u)}(X^{2^{d}-1})(f_2(x_{u-1}y))^{2^{d}}.$$
Since $\nu_1(y)=0$, $f_2(y) = f_1(y)$ and $\ell(I\cup 2) < k-2$, from this equality, Case \ref{th01} and Case \ref{th1} with $i=1$, we obtain  $[x] \in [\mathcal P_k(n)]$.

\begin{exa}\label{vdu1} Let $k=4, d_1 =5, d_2 =d_3=3$. Then, we have $n =45, m =3$ and  $\omega = (3,3,3,1,1)$ is the minimal sequence. Let $B_3(n)$ be the set of all the admissible monomials of degree $n$ in $P_3$. Then $\bar B_3(m)$ is the set of all the admissible monomials of degree $m$ in $P_3$. Let $I=\emptyset, y_0 =x_4^2= f_1(x_3^2)= f_2(x_3^2)  \in P_4$. Denote $A =\phi_{(1;\emptyset)}(X^{7})(x_1y_0)^{8}$,  $B =\phi_{(2;\emptyset)}(X^{7})(x_2y_0)^{8}$ and $z = \phi_{(1;2)}(X^{7})(x_1y_0)^{8}$.  From the proof of this case, we obtain
\begin{align*}
A&\equiv \phi_{(1;2)}(X^{7})(x_2y_0)^{8} + \phi_{(1;3)}(X^{7})(x_3y_0)^{8} +\phi_{(1;4)}(X^{7})(x_4y_0)^{8},\\
B &\equiv z + \phi_{(2;3)}(X^{7})(x_3y_0)^{8} +\phi_{(2;4)}(X^{7})(x_4y_0)^{8},\\
z &\equiv \phi_{(1;2)}(X^{7})(x_2y_0)^{8} + \phi_{(1;(2,3))}(X^{7})(x_3y_0)^{8} +\phi_{(1;(2,4))}(X^{7})(x_4y_0)^{8}.
\end{align*} 
Since $x_2y_0 = x_2x_4^2, \ x_3y_0 = x_3x_4^2, \ x_4y_0 = x_4^3$ are the admissible monomials, we get $[A], [B], [x] \in  [\mathcal P_4(n)]$. Furthermore,
\begin{align*}  
A&\equiv    x_1x_2^{7}x_3^{7}x_4^{30} + x_1x_2^{7}x_3^{14}x_4^{23} + x_1x_2^{14}x_3^{7}x_4^{23},\\
B&\equiv  x_1x_2^{14}x_3^{7}x_4^{23} + x_1^{3}x_2^{5}x_3^{7}x_4^{30} + x_1^{3}x_2^{5}x_3^{14}x_4^{23} + x_1^{7}x_2x_3^{7}x_4^{30} + x_1^{7}x_2x_3^{14}x_4^{23}.
\end{align*} 
All monomials in the right hand sides of the last equalities are admissible.
\end{exa}

\begin{khi}\label{th3} 
If $\bar y = x_3^{2^s}f_3(y)$, with $y \in (P_{k-1})_{m-2^s}$, $\nu_1(y) = \nu_2(y) = 0$ and $i=3$, then $[x] \in [\mathcal P_k(n)]$. 
\end{khi}

According to Lemmas \ref{bdmu} and \ref{bdcb3}, we need only to prove this case for $s = 0$. Note that, since $\nu_1(y) = \nu_2(y) = 0$, we have $x_3f_3(y) = f_2(x_2y)$. 
If $I = I_3$, then by Case \ref{th03} with $i=3$, $[x] \in [\mathcal P_k(n)]$. Suppose $I\ne I_3$.

If $d_{k-2}> d_{k-1}$, then $\omega_k(x) = \omega_1(y) +1 = k-2$. Hence,  $\alpha_0(\nu_j(y)) = 1$ for $j = 3,\ldots, k-1$. Applying Lemma \ref{bdcbs}(i) with $y_0 = f_3(y)$ and Theorem \ref{dlsig}, we get
$$x \equiv \phi_{(1;I)}(X^{2^d-1})(x_1f_1(y))^{2^d} +  \phi_{(2;I)}(X^{2^d-1})(x_2f_2(y))^{2^d}.$$
Hence,  by  Case \ref{th1}, $[x] \in [\mathcal P_k(n)]$.

Suppose that $d_{k-2}= d_{k-1}$. If $\ell(I) < k-4$, then using Lemma \ref{bdcbs}(ii) with $y_0 = f_3(y) = f_1(y) = f_2(y)$, one gets
\begin{align*} x \equiv\phi_{(1;I\cup 3)}(X^{2^{d}-1})(x_1f_1(y))^{2^{d}} &+ \phi_{(2;I\cup 3)}(X^{2^{d}-1})(x_{2}f_2(y))^{2^{d}}\\  
& + \sum_{v=4}^k \phi_{(3;I\cup v)}(X^{2^{d}-1})(f_3(x_{v-1}y))^{2^{d}}.
\end{align*}
From this equality and Cases \ref{th01}, \ref{th1}, we obtain  $[x] \in [\mathcal P_k(n)]$.

If $\ell(I) = k-4$, then $I = (4,\ldots , \hat u , \dots , k)$ with $4 \leqslant u \leqslant k$. Since  $\omega_k(x) = \omega_1(y) +1 = k-3$, we have $\omega_1(y) = k-4$. Hence, there exists uniquely $3 \leqslant t < k$ such that $\alpha_0(\nu_{t}(y)) = 0$.
 
If $t=u-1$, then using Lemma \ref{bdcbs}(i) with $y_0 = f_3(y)$ and Theorem \ref{dlsig}, we obtain
\begin{align*} x \equiv \phi_{(1;I)}(X^{2^{d}-1})(x_{1}f_1(y))^{2^{d}} &+ \phi_{(2;I)}(X^{2^{d}-1})(x_{2}f_2(y))^{2^{d}}\\ 
&+ \phi_{(4;I_4)}(X^{2^{d}-1})(f_3(x_{t}y))^{2^{d}} .
\end{align*}
By Cases \ref{th03} and \ref{th1}, we get $[x] \in [\mathcal P_k(n)]$.

If $u=4 < t+1$, then using Lemma \ref{bdcbs}(i) with $y_0 = f_3(y)$ and Theorem \ref{dlsig}, we get
\begin{align*}x \equiv \phi_{(1;I)}(X^{2^{d}-1})(x_{1}f_1(y))^{2^{d}} &+ \phi_{(2;I)}(X^{2^{d}-1})(x_{2}f_2(y))^{2^{d}}\\ 
&+ \phi_{(5;I_5)}(X^{2^{d}-1})(x_5f_3(x_{t}y/x_4))^{2^{d}}.
\end{align*} 
Applying Lemma \ref{bdcbs}(i) with $y_0 = f_3(x_ty/x_4)$ and Theorem \ref{dlsig}, we have 
  $$ \phi_{(5;I_5)}(X^{2^{d}-1})(x_5f_3(x_{t}y/x_4))^{2^{d}}\equiv \sum_{1\leqslant v \leqslant 3} \phi_{(v;I_5)}(X^{2^{d}-1})(x_{v}f_3(x_ty/x_4))^{2^{d}}.$$
Since $\ell(I_5) = k-5<k-4$, $\phi_{(3;I_5)}(X^{2^{d}-1})(x_{3}f_3(x_ty/x_4))^{2^{d}} \in [\mathcal P_k(n)]$. So, combining Case \ref{th1}, the above equalities and the fact that $x_{v}f_3(x_ty/x_4)= x_{v}f_v(x_ty/x_4)$ for $v=1, 2$, one gets $[x] \in [\mathcal P_k(n)]$.

Suppose that  $4 < u \ne t+1$. Using Lemma \ref{bdcbs}(i) with $y_0 = f_3(y)$ and Theorem \ref{dlsig}, we obtain
\begin{align*}  x \equiv \phi_{(1;I)}(X^{2^{d}-1})(x_{1}f_1(y))^{2^{d}} &+ \phi_{(2;I)}(X^{2^{d}-1})(x_{2}f_2(y))^{2^{d}}\\ 
 &+ \phi_{(4;I\setminus 4)}(X^{2^{d}-1})(x_4f_3(x_{t}y/x_3))^{2^{d}}.
\end{align*}
Applying Lemma \ref{bdcbs}(i) with $y_0 = f_3(x_ty/x_3)$ and Theorem \ref{dlsig}, we have 
  $$ \phi_{(4;I\setminus 4)}(X^{2^{d}-1})(x_4f_3(x_{t}y/x_3))^{2^{d}}\equiv \sum_{1\leqslant v \leqslant 3} \phi_{(v;I\setminus 4)}(X^{2^{d}-1})(x_{v}f_3(x_ty/x_3))^{2^{d}}.$$
Since $\ell(I\setminus 4) = k-5<k-4$, $\phi_{(3;I\setminus 4)}(X^{2^{d}-1})(x_{3}f_3(x_ty/x_3))^{2^{d}}\in [\mathcal P_k(n)]$. So, from the above equalities, Case \ref{th1} and the fact that $x_{v}f_3(x_ty/x_3)= x_{v}f_v(x_ty/x_3)$, for $v=1,2$, we get $[x] \in [\mathcal P_k(n)]$.

\begin{exa}\label{vdu2} Let $k=4, n, m, B_3(n), A, B, y_0$ be as in Example \ref{vdu1}. Then,
\begin{align*}C =\phi_{(3;\emptyset)}(X^{15})(x_3y_0)^{8}\equiv A + B  +\phi_{(4;\emptyset)}(X^{7})(x_4y_0)^{8}.
\end{align*}
Since $\phi_{(4;\emptyset)}(X^{7})(x_4y_0)^{8} = \phi_{(1;I_1)}(X^{7}x_3^{24})$ and $X^{7}x_3^{24}= x_1^7x_2^7x_3^{31} \in B_3(n)$, one gets $[C] \in [\mathcal P_4(n)]$ and 
\begin{align*}C&\equiv x_1x_2^{7}x_3^{7}x_4^{30} + x_1x_2^{7}x_3^{14}x_4^{23} + x_1^{3}x_2^{5}x_3^{7}x_4^{30} + x_1^{3}x_2^{5}x_3^{14}x_4^{23}\\ &\quad+ x_1^{7}x_2x_3^{7}x_4^{30} + x_1^{7}x_2x_3^{14}x_4^{23} + x_1^{7}x_2^{7}x_3^{7}x_4^{24}.
\end{align*} 
\end{exa}
\begin{khi}\label{th4}   
If  $\nu_1(\bar y) = \nu_2(\bar y) = 0$ and $i=4$, then $[x] \in [\mathcal P_k(n)]$.  
\end{khi}

Since  $\nu_1(\bar y) = \nu_2(\bar y) = 0$, we have $\bar y =  x_3^{b}x_4^cf_4(y)$ for suitable $y \in (P_{k-1})_{m-b-c}$ with $\nu_j(y) =0, j \leqslant 3$, and $b =\nu_3(\bar y), c =\nu_4(\bar y)$. Using Lemmas \ref{bdmu} and \ref{bdcb3}, we assume that $b = 2^s-1$.  

We prove this case by induction on $c$. If $c = 0$, then by Case \ref{th01}, $[x] \in [\mathcal P_k(n)]$.  Suppose that $c>0$ and this case holds for $c-1$ and all $I \subset I_4$.

If $I \ne I_4$, then applying Lemma \ref{bdcbs}(ii) with  $y_0 = x_3^{b}x_4^{c-1}f_4(y)$, we have 
\begin{align*} x &\equiv \phi_{(1;I\cup 4)}(X^{2^{d}-1})(x_1f_1(x_2^{b}x_3^{c-1}y))^{2^{d}}+ \phi_{(2;I\cup 4)}(X^{2^{d}-1})(x_2f_2(x_2^{b}x_3^{c-1}y))^{2^{d}}\\ 
&+ \phi_{(3;I\cup 4)}(X^{2^{d}-1})(x_3^{2^s}f_3(x_3^{c-1}y))^{2^{d}} + \sum_{u=5}^k \phi_{(4;I\cup u)}(X^{2^{d}-1})(x_3^{b}x_4^{c-1}f_4(x_{u-1}y))^{2^{d}}.
\end{align*}

Combining this equality, Cases \ref{th1}, 
\ref{th3} and the inductive hypothesis gives  $[x] \in [\mathcal P_k(n)].$

If $I = I_4$, then applying Lemma \ref{bdcbs}(i) with  $y_0 = x_3^{b}x_4^{c-1}f_4(y)$ and using Cases \ref{th1}, \ref{th3}, we obtain 
\begin{align*} x &\equiv \phi_{(1;I_4)}(X^{2^{d}-1})(x_1f_1(x_2^{b}x_3^{c-1}y))^{2^{d}}+ \phi_{(2;I_4)}(X^{2^{d}-1})(x_2f_2(x_2^{b}x_3^{c-1}y))^{2^{d}}\\ 
&\quad \ + \phi_{(3;I_ 4)}(X^{2^{d}-1})(x_3^{2^s}f_3(x_3^{c-1}y))^{2^{d}}  + Y_5y_0^{2^{d}} \equiv Y_5y_0^{2^{d}}\ \text{mod}(\mathcal P_k(n)).
\end{align*}
By Lemmas \ref{hq4} and \ref{bdcb3}, 
\begin{align*} Y_5y_0^{2^{d}} \equiv  \sum\phi_{(j;J)}(X^{2^{d}-1})\big(x_jx_3^{b}x_4^{c-1}f_4(y)\big)^{2^{d}},
\end{align*}
where the sum runs over all $(j,J)$ with $1 \leqslant j <5, J\subset I_4$ and $J\ne I_4$. 
Since $J\ne I_4, [\phi_{(4;J)}(X^{2^{d}-1})\big(x_3^{b}x_4^{c}f_4(y)\big)^{2^{d}}] \in [\mathcal P_k(n)]$. By Case \ref{th1},
$$
[\phi_{(j;J)}(X^{2^{d}-1})\big(x_jx_3^{b}x_4^{c-1}f_4(y)\big)^{2^{d}} ] = [\phi_{(j;J)}(X^{2^{d}-1})\big(x_jf_j(x_2^{b}x_3^{c-1}y)\big)^{2^{d}}]  \in [\mathcal P_k(n)],$$
for $j=1,2$. By Case \ref{th3},
$$[\phi_{(3;J)}(X^{2^{d}-1})\big(x_3x_3^{b}x_4^{c-1}f_4(y)\big)^{2^{d}}] = [\phi_{(3;J)}(X^{2^{d}-1})\big(x_3^{2^s}f_3(x_3^{c-1}y)\big)^{2^{d}}] \in [\mathcal P_k(n)].$$
Hence, $[x] \in [\mathcal P_k(n)].$

\begin{exa}\label{vdu3} Let $k=4, n, m, B_3(n), C$ be as in Example \ref{vdu2}.  Let $I=\emptyset, y_0 =x_3x_4, \bar y = x_4y_0 = x_3x_4^2$.  From the proof of this case, we obtain
\begin{align*}
D=\phi_{(4;\emptyset)}(X^{7})(x_4y_0)^{8}\equiv x + y +\phi_{(3;\emptyset)}(X^{7})(x_3y_0)^{8},
\end{align*} 
where  $x =\phi_{(1;\emptyset)}(X^{7})(x_1y_0)^{8}$, $y =\phi_{(2;\emptyset)}(X^{7})(x_2y_0)^{8}$. Since $x_3y_0 = x_3^2x_4 \equiv x_3x_4^2$, $\phi_{(3;\emptyset)}(X^{7})(x_3y_0)^{8} \equiv C$. By Case \ref{th1}, $[x], [y] \in [\mathcal P_4(n)]$. Hence, $[D] = [C] + [x] + [y] \in [\mathcal P_4(n)]$.
By a computation analogous to the previous one, we obtain
\begin{align*}  
D &\equiv x_1x_2^{7}x_3^{7}x_4^{30} + x_1x_2^{7}x_3^{15}x_4^{22} + x_1^{3}x_2^{5}x_3^{7}x_4^{30} + x_1^{3}x_2^{5}x_3^{15}x_4^{22}\\ &\quad + x_1^{7}x_2x_3^{7}x_4^{30} + x_1^{7}x_2x_3^{15}x_4^{22} + x_1^{7}x_2^{7}x_3^{7}x_4^{24}. 
\end{align*} 
\end{exa}
\begin{khi}\label{th5} 
If $\nu_1(\bar y) = \nu_2(\bar y) = 0$ and $i=3$, then $[x] \in [\mathcal P(n)]$.
\end{khi}
 We have $\bar y = x_3^{b}f_3(y)$ for suitable $y \in (P_{k-1})_{m-b}$ with $\nu_1(y) = \nu_2(y) =0$, and $b = \nu_3(\bar y)$. We prove $[x] \in [\mathcal P(n)]$ by induction on $b$. 

If $b = 0$, then by Case \ref{th01}, $[x] \in [\mathcal P(n)].$ Suppose $b >0$ and this case holds for $b-1$. If $\alpha_0(b) =0$, then $\bar y = Sq^1(x_3^{b-1}f_3(y)) + x_3^{b-1}f_3(Sq^1(y)) \equiv x_3^{b-1}f_3(Sq^1(y))$. Hence, using Lemma \ref{bdcb3} and the inductive hypothesis, one gets  $[x] \in [\mathcal P(n)]$. Now, assume that $\alpha_0(b) = 1$.

Since $x_3^{b}f_3(y) = f_2(x_2^{b}y)$, if $I=I_3$, then  
by Case \ref{th03}, $[x] \in [\mathcal P(n)].$ 
If $\ell(I) <k-4$, then applying Lemma \ref{bdcbs}(ii) with  $y_0 = x_3^{b-1}f_3(y)$, we obtain  
\begin{align*}
x \equiv \phi_{(1;I\cup 3)}(X^{2^{d}-1})(x_1f_1(x_2^{b-1}y))^{2^{d}} &+ \phi_{(2;I\cup 3)}(X^{2^{d}-1})(x_2f_2(x_2^{b-1}y))^{2^{d}}\\ 
&+ \sum_{v=4}^k \phi_{(3;I\cup v)}(X^{2^{d}-1})(x_3^{b-1}f_3(x_{v-1}y))^{2^{d}}.
\end{align*}
Using Case \ref{th1} and the inductive hypothesis, one gets  $[x] \in [\mathcal P(n)].$

Suppose that $\ell(I) =k-4$. Then $I = (4, \ldots, \hat u, \ldots , k)$, with $4 \leqslant  u\leqslant k$. 

If $d_{k-2}> d_{k-1}$, then $\omega_k(x) = \omega_1(y) + \alpha_0(b) = k-2$. Hence,  $\alpha_0(\nu_j(y)) = 1$ for $j = 3,\ldots, k-1$. Applying Lemma \ref{bdcbs}(i) with $y_0 = x_3^{b-1}f_3(y)$ and Theorem \ref{dlsig}, we get
$$x \equiv \phi_{(1;I)}(X^{2^d-1})(x_1f_1(x_2^{b-1}y))^{2^d} +  \phi_{(2;I)}(X^{2^d-1})(x_2f_2(x_2^{b-1}y))^{2^d}.$$
Hence,  by  Case \ref{th1}, we obtain  $[x] \in [\mathcal P(n)]$.

Suppose  $d_{k-2}= d_{k-1}$. Since  $\omega_k(x) = \omega_1(y) +\alpha_0(b) = k-3$, we have $\omega_1(y) = k-4$. Hence, there exists uniquely $3 \leqslant t \leqslant k-1$ such that $\alpha_0(\nu_{t}(y)) = 0$.
 
If $t=u-1$, then using Lemma \ref{bdcbs}(i) with $y_0 = x_3^{b-1}f_3(y)$ and Theorem \ref{dlsig}, we have
\begin{align*} x \equiv \phi_{(1;I)}(X^{2^{d}-1})(x_{1}f_1(x_2^{b-1}y))^{2^{d}} &+ \phi_{(2;I)}(X^{2^{d}-1})(x_{2}f_2(x_2^{b-1}y))^{2^{d}} \\ 
&+ \phi_{(4;I_4)}(X^{2^{d}-1})(x_3^{b-1}f_3(x_{t}y))^{2^{d}} .
\end{align*}
From this equality, Case \ref{th1} and \ref{th4}, we get $[x] \in [\mathcal P(n)]$.

If $u=4 < t+1$, then using Lemma \ref{bdcbs}(i) with $y_0 = x_3^{b-1}f_3(y)$ and Theorem \ref{dlsig}, we obtain
\begin{align*}x \equiv \phi_{(1;I)}(X^{2^{d}-1})(x_{1}f_1(x_2^{b-1}y))^{2^{d}} &+ \phi_{(2;I)}(X^{2^{d}-1})(x_{2}f_2(x_2^{b-1}y))^{2^{d}} \\  
&+ \phi_{(5;I_5)}(X^{2^{d}-1})(x_3^{b-1}f_3(x_{t}y))^{2^{d}}.
\end{align*} 
Applying Lemma \ref{bdcbs}(i) with $y_0 = x_3^{b-1}f_3(x_ty/x_4)$ and Theorem \ref{dlsig}, we have 
  $$ \phi_{(5;I_5)}(X^{2^{d}-1})(x_3^{b-1}f_3(x_{t}y))^{2^{d}}\equiv \sum_{1\leqslant v \leqslant 3} \phi_{(v;I_5)}(X^{2^{d}-1})(x_{v}x_3^{b-1}f_3(x_ty/x_4))^{2^{d}}.$$
Since $\ell(I_5) = k-5<k-4$, $\phi_{(3;I_5)}(X^{2^{d}-1})(x_3^{b}f_3(x_ty/x_4))^{2^{d}}\in [\mathcal P(n)]$. Hence, combining the above equalities, Case \ref{th1} and the fact that $x_{v}x_3^{b-1}f_3(x_ty/x_4) = x_{v}f_v(x_2^{b-1}x_ty/x_4)$, for $v=1,2$, one gets $[x] \in [\mathcal P(n)]$.

Suppose that  $4 < u \ne t+1$. Using Lemma \ref{bdcbs}(i) with $y_0 = x_3^{b-1}f_3(y)$ and Theorem \ref{dlsig}, we obtain
\begin{align*} x  \equiv \phi_{(1;I)}(X^{2^{d}-1})(x_{1}f_1(x_2^{b-1}y))^{2^{d}} &+ \phi_{(2;I)}(X^{2^{d}-1})(x_{2}f_2(x_2^{b-1}y))^{2^{d}}\\ 
 &+ \phi_{(4;I\setminus 4)}(X^{2^{d}-1})(x_3^{b-1}f_3(x_{t}y))^{2^{d}}.
\end{align*}
From the above equalities, Cases \ref{th1} and \ref{th4}, we get $[x] \in [\mathcal P(n)]$.

\begin{exa}\label{vdu4} Let $k=4, n, m, B_3(n), D$ be as in Example  \ref{vdu3}.  Let $I=\emptyset, y_0 =x_3^2= x_3^2f_3(x_3^0)  \in P_4$. This case is entry with $b=3, u=4, t=3=u-1$. Then,
$E =\phi_{(3;\emptyset)}(X^{7})(x_3y_0)^{8}\equiv x + y + \phi_{(4;\emptyset)}(X^{7})(x_4y_0)^{8} ,
$
where  $x =\phi_{(1;\emptyset)}(X^{7})(x_1y_0)^{8}$, $y =\phi_{(2;\emptyset)}(X^{7})(x_2y_0)^{8}$. Since $x_4y_0 = x_3^2x_4 \equiv x_3x_4^2$, we have  $\phi_{(4;\emptyset)}(X^{7})(x_4y_0)^{8}\equiv D$. By Case \ref{th1}, $[x], [y] \in [\mathcal P_4(n)]$. Hence, $[E] = [D] + [x] + [y] \in [\mathcal P_4(n)]$.
By a computation analogous to the previous one, we obtain
\begin{align*} 
E&\equiv x_1x_2^{7}x_3^{7}x_4^{30} + x_1x_2^{7}x_3^{30}x_4^{7} + x_1^{3}x_2^{5}x_3^{7}x_4^{30} + x_1^{3}x_2^{5}x_3^{30}x_4^{7}\\ &\quad + x_1^{7}x_2x_3^{7}x_4^{30} + x_1^{7}x_2x_3^{30}x_4^{7} + x_1^{7}x_2^{7}x_3^{7}x_4^{24}. 
\end{align*} 
Here, the monomials in the right hand sides of the last relation are admissible.
\end{exa}

\begin{khi}\label{th6} 
If $\bar y = x_2^{2^s}f_2(y)$ for  $y \in (P_{k-1})_{m-2^s}$ with $\nu_1(y) = 0$ and $i=2$, then $[x] \in [\mathcal P(n)]$.
\end{khi}
It suffices to prove this case for $s = 0$. 
If $\ell(I) <k-3$, then by Case \ref{th1}, $[x] \in [\mathcal P(n)] $. Since $x_2f_2(y) = f_1(x_1y)$, if $I=I_2$, then by Case \ref{th03}, $[x] \in [\mathcal P(n)].$ 

Suppose $\ell(I) = k-3$. Then, $I = (3,\ldots, \hat u,\ldots , k)$ with $3 \leqslant  u\leqslant k$. 

If $u = 3$, then using Lemma \ref{bdcbs}(i) with $y_0 = f_2(y)$, we get
\begin{align*}x \equiv  \phi_{(1;I_3)}(X^{2^{d}-1})(x_1f_1(y))^{2^{d}} &+ \phi_{(3;I_3)}(X^{2^{d}-1})(f_2(x_2y))^{2^{d}}\\ 
&+ \sum_{v=4}^k\phi_{(4;I_4)}(X^{2^{d}-1})(f_2(x_{v-1}y))^{2^{d}}.\end{align*}

If $u > 3$, then using Lemma \ref{bdcbs}(i) with $y_0 = f_2(y)$, we get
\begin{align*}x \equiv  \phi_{(1;I)}(X^{2^{d}-1})(x_1f_1(y))^{2^{d}} 
&+ \sum_{3\leqslant v \leqslant k}\phi_{(3;I\cup v\setminus 3)}(X^{2^{d}-1})(f_2(x_{v-1}y))^{2^{d}}.
\end{align*}

Since $\nu_1(f_2(x_{v-1}y)) = \nu_2(f_2(x_{v-1}y)) = 0$, for $3\leqslant v \leqslant k$,
combining the above equalities, Cases \ref{th03}, \ref{th1}, \ref{th4} and \ref{th5}, we obtain  $[x] \in [\mathcal P(n)].$ 

\begin{exa}\label{vdu4b} Let $k=4, n, m, B_3(n), E$ be as in Example  \ref{vdu3}.  Let $I=(3), y_0 =x_3^2= x_2^2f_3(x_3^0)  \in P_4$. This case is entry with $u=4, t=3=u-1$. Then, $F =\phi_{(2;3)}(X^{7})(x_2y_0)^{8}\equiv x + E + y,$
where  $x =\phi_{(1;3)}(X^{7})(x_1y_0)^{8}$, $y =\phi_{(3;4)}(X^{7})(x_4y_0)^{8}$. Since $x_4y_0 = x_3^2x_4 \equiv x_3x_4^2$, we have  $y\equiv \phi_{(1;I_1)}(X^{7}(x_2x_3^2)^{8}) = x_1^7x_2^7x_3^9x_4^{22}$. By Case \ref{th1}, $[x] \in [\mathcal P_4(n)]$. Hence, $[F] = [x] + [E] + [y] \in [\mathcal P_4(n)]$.
By a computation analogous to the previous one, we obtain
\begin{align*} 
F&\equiv x_1x_2^{7}x_3^{7}x_4^{30} + x_1^{3}x_2^{5}x_3^{7}x_4^{30} + x_1^{3}x_2^{5}x_3^{30}x_4^{7} + x_1^{3}x_2^{7}x_3^{13}x_4^{22} + x_1^{3}x_2^{13}x_3^{22}x_4^{7}\\ &\quad + x_1^{7}x_2x_3^{7}x_4^{30} + x_1^{7}x_2x_3^{30}x_4^{7} + x_1^{7}x_2^{7}x_3^{7}x_4^{24} + x_1^{7}x_2^{7}x_3^{9}x_4^{22}. 
\end{align*} 
Here, the monomials in the right hand sides of the last relation are admissible.
\end{exa} 
\begin{khi}\label{th7} 
If $\nu_1(\bar y) =0$ and $i=3$, then $[x] \in [\mathcal P(n)]$.
\end{khi}
We have $\bar y = x_2^{a}x_3^bf_3(y)$ for suitable $y \in (P_{k-1})_{m-a-b}$ with $\nu_1(y) = \nu_2(y) = 0$, $a = \nu_2(\bar y), b = \nu_3(\bar y)$. Using Lemmas \ref{bdmu} and \ref{bdcb3}, we can assume that $a = 2^s-1$. We prove this case  by induction on $b$. If $b= 0$, then by Case \ref{th01}, this case is true. Suppose that $b>0$ and this case is true for $b-1$. 

If $I \ne I_3$, then using Lemma \ref{bdcbs}(ii) with $y_0 =x_2^{a}x_3^{b-1}f_3(y)$, we get
\begin{align*}x \equiv \phi_{(1;I\cup 3)}(X^{2^{d}-1})(x_1f_1(x_1^{a}x_2^{b-1}&y))^{2^{d}} + \phi_{(2;I\cup 3)}(X^{2^{d}-1})(x_2^{2^s}f_2(x_2^{b-1}y))^{2^{d}}\\
 &+ \sum_{v=4}^k\phi_{(3;I\cup v)}(X^{2^{d}-1})(x_2^{a}x_3^{b-1}f_3(x_{v-1}y))^{2^{d}}.
\end{align*}
From this equality,  Case \ref{th1}  
and  the inductive hypothesis we obtain $[x] \in [\mathcal P(n)].$ 

If $I = I_3$, then  using Lemma \ref{bdcbs}(i) with $y_0 =x_2^{a}x_3^{b-1}f_3(y)$, and Case \ref{th1},  we have
\begin{align*}x &\equiv \phi_{(1;I_3)}(X^{2^{d}-1})(x_1f_1(x_1^{a}x_2^{b-1}y))^{2^{d}}\\
&\quad + \phi_{(2;I_3)}(X^{2^{d}-1})(x_2^{2^s}f_2(x_2^{b-1}y))^{2^{d}} + Y_4y_0^{2^{d}}\equiv Y_4y_0^{2^{d}}\ \text{mod}(\mathcal P(n)).
\end{align*}

Using Lemmas \ref{hq4} and \ref{bdcb3}, we get
$$Y_4y_0^{2^{d}} \equiv \sum_{(j;J)}\phi_{(j;J)}(X^{2^{d}-1})(x_jx_2^{a}x_3^{b-1}f_3(y))^{2^{d}},$$
where the last sum runs over some $(j;J)$ with $ 1 \leqslant j  < 4$, $J \subset I_3$ and $J \ne I_3$.
Since $J \ne I_3$, $\phi_{(3;J)}(X^{2^{d}-1})(x_2^{a}x_3^{b}f_3(y))^{2^{d}}\in [\mathcal P(n)]$. By Cases \ref{th1} and \ref{th6},
\begin{align*}
\phi_{(1;J)}(X^{2^{d}-1})(x_1x_2^{a}x_3^{b-1}f_3(y))^{2^{d}} = \phi_{(1;J)}(X^{2^{d}-1})(x_1f_1(x_1^{a}x_2^{b-1}y))^{2^{d}}\in [\mathcal P(n)],\\
\phi_{(2;J)}(X^{2^{d}-1})(x_2x_2^{a}x_3^{b-1}f_3(y))^{2^{d}} = \phi_{(2;J)}(X^{2^{d}-1})(x_2^{2^s}f_2(x_2^{b-1}y))^{2^{d}}\in [\mathcal P(n)].
\end{align*}
 This case is proved.
\begin{exa}\label{vdu5} Let $k=4, n, m, B_3(n), F$ be as in Example \ref{vdu4b}.  
 Let $I=I_3=(4), y_0 =x_2x_4, \bar y = x_3y_0$. Since $I= I_3$,
$G =\phi_{(3;4)}(X^{7})(x_3y_0)^{8}\equiv Y + Z + Y_4y_0^8 ,
$
where  $Y =\phi_{(1;4)}(X^{7})(x_1y_0)^{8}$, $Z =\phi_{(2;4)}(X^{7})(x_2^2x_4)^{8}$. By Case \ref{th1}, $[Y], [Z] \in [\mathcal P_4(n)]$. Since $Y_4y_0^8 = \phi_{(4;\emptyset)}(X^{7})(x_2x_4^2)^{8}$, we have
$Y_4y_0^8 \equiv y + z +w$, where 
$$y = \phi_{(1;\emptyset)}(X^{7})(x_1x_2x_4)^{8},\ z = \phi_{(2;\emptyset)}(X^{7})(x_2^2x_4)^{8},\ w = \phi_{(3;\emptyset)}(X^{7})(x_2x_3x_4)^{8}.$$
By Case \ref{th1}, $[y], [z] \in [\mathcal P_4(n)]$. We have
\begin{align*}w \equiv \phi_{(1;3)}(X^{7})(x_1x_2x_4)^{8} + \phi_{(2;3)}(X^{7})(x_2^2x_4)^{8} + \phi_{(3;4)}(X^{7})(x_2x_4^2)^{8}.
\end{align*}
Using Cases \ref{th01} and \ref{th1}, one gets $[w] \in [\mathcal P_4(n)]$. Hence, $[x]= [Y] + [Z] + [y] + [z] + [w] \in [\mathcal P_4(n)]$.
By a computation analogous to the previous one, we obtain
\begin{align*} 
G&= x_1^7x_2^{15}x_3^9x_4^{14}\equiv x_1^{3}x_2^{15}x_3^{13}x_4^{14} + x_1^{7}x_2^{7}x_3^{9}x_4^{22} + x_1^{3}x_2^{7}x_3^{13}x_4^{22}\\ &\quad  + x_1^{3}x_2^{15}x_3^{5}x_4^{22}+ x_1^{7}x_2^{15}x_3x_4^{22} + x_1^{7}x_2^{7}x_3x_4^{30} + x_1^{3}x_2^{7}x_3^{5}x_4^{30}. 
\end{align*} 
All monomials in the right hand sides of the last relation are admissible.

Now, let $I=\emptyset, y_1 =x_2x_3, \bar y = x_3y_1$. Then,
$H =\phi_{(3;\emptyset)}(X^{7})(x_3y_1)^{8}\equiv a +F + G,
$
where  $a =\phi_{(1;3)}(X^{7})(x_1y_1)^{8}$. By Case \ref{th1}, $[a] \in [\mathcal P_4(n)]$. Hence, $[H]= [a] + [F] + [G] \in [\mathcal P_4(n)]$.
By a computation analogous to the previous one, we obtain
\begin{align*} 
H &= x_1x_2^{7}x_3^{7}x_4^{30} + x_1x_2^{15}x_3^{22}x_4^{7} + x_1^{3}x_2^{5}x_3^{7}x_4^{30} + x_1^{3}x_2^{5}x_3^{30}x_4^{7} \\ &\quad + x_1^{3}x_2^{7}x_3^{5}x_4^{30} + x_1^{3}x_2^{15}x_3^{5}x_4^{22} + x_1^{7}x_2x_3^{7}x_4^{30} + x_1^{7}x_2x_3^{30}x_4^{7}\\ &\quad  + x_1^{7}x_2^{7}x_3x_4^{30} + x_1^{7}x_2^{7}x_3^{7}x_4^{24} + x_1^{7}x_2^{15}x_3x_4^{22}. 
\end{align*} 
All monomials in the right hand sides of the last relation are admissible.
\end{exa}

\begin{khi}\label{th8} 
If $\nu_1(\bar y) = 0$ and $i=2$, then $[x] \in [\mathcal P(n)]$. 
\end{khi}
We have $\bar y = x_2^{a}f_2(y)$ for suitable $y \in (P_{k-1})_{m-a}$ with $\nu_1(y)  =0$ and $a=\nu_1(\bar y)$. We prove $[x] \in [\mathcal P(n)]$ by induction on $a$. If $a = 0$, then by Case \ref{th01}, $[x] \in [\mathcal P(n)]$. Suppose that $ a>0$ and this case holds for $a-1$.

 Since $x_2^{a}f_2(y) = f_1(x_1^{a}y)$, if $I = I_2$, then by Case \ref{th03}, $[x] \in [\mathcal P(n)]$. 
If $\ell(I) < k-3$, then applying Lemma \ref{bdcbs}(ii) with $y_0 =x_2^{a-1}f_2(y)$, we have
$$x \equiv  \phi_{(1;I\cup 2)}(X^{2^{d}-1})(x_1f_1(x_1^{a-1}y))^{2^{d}} +\sum_{v=3}^k \phi_{(2;I\cup v)}(X^{2^{d}-1})(x_2^{a-1}f_2(x_{v-1}y))^{2^{d}}.$$
Using Case \ref{th1} and the inductive hypothesis, we get $[x] \in [\mathcal P(n)]$.

Suppose that $\ell(I)=k-3$. Then $I = (3, \ldots,\hat u, \ldots , k)$ with $3 \leqslant u \leqslant k$. 

If $u=3$, then $I= I_3$. Applying Lemma \ref{bdcbs}(i) with $y_0 =x_2^{a-1}f_2(y)$, we obtain
\begin{align*} x \equiv  \phi_{(1;I_3)}(X^{2^{d}-1})(x_1f_1(x_1^{a-1}y))^{2^{d}} &+\phi_{(3;I_3)}(X^{2^{d}-1})(x_2^{a-1}f_2(x_2y))^{2^{d}}\\
& + \sum_{v=4}^k \phi_{(4;I_4)}(X^{2^{d}-1})(x_{v}x_2^{a-1}f_2(y))^{2^{d}}.
\end{align*}

Applying Lemma \ref{hq4} and Lemma \ref{bdcb3}, one gets
\begin{align*} \sum_{v=4}^k \phi_{(4;I_4)}(X^{2^{d}-1})(x_{v}x_2^{a-1}f_2(y))^{2^{d}} &= Y_4y_0^{2^d} \\ 
&\equiv  \sum_{(j;J)} \phi_{(j;J)}(X^{2^{d}-1})(x_jx_2^{a-1}f_2(y))^{2^{d}},\end{align*}
where the last sum runs over some $(j;J)$ with $1 \leqslant j < 4$, $J \subset I_3$ and $J \ne I_3$. 

Since $\ell(J) < \ell(I_3) = k-3$, from the above equalities and Cases  \ref{th1}, \ref{th7}, we have $[x] \in [\mathcal P(n)]$. 

If $u > 3$, applying Lemma \ref{bdcbs}(i) with $y_0 =x_2^{a-1}f_2(y)$, we get
$$ x \equiv  \phi_{(1;I)}(X^{2^{d}-1})(x_1f_1(x_1^{a-1}y))^{2^{d}} 
 + \sum_{v=3}^k \phi_{(3;I\cup v)}(X^{2^{d}-1})(x_2^{a-1}f_2(x_{v-1}y))^{2^{d}}.
$$
From the last equality, Cases \ref{th1} and \ref{th7}, we have $[x] \in [\mathcal P(n)]$. 

\begin{exa}\label{vdu6} Let $k=4, n, m, B_3(n), H$ be as in Example \ref{vdu5}.  
 Let $I=I_3=(3), y_0 =x_2^2, \bar y = x_2y_0= x_2^3$. Then,
$K :=\phi_{(2;3)}(X^{7})(x_2y_0)^{8}\equiv a + H + b ,
$
where  $a =\phi_{(1;3)}(X^{7})(x_1y_0)^{8}$, $b =\phi_{(3;4)}(X^{7})(x_4y_0)^{8}$. By Cases \ref{th01} and \ref{th1}, $[a], [b] \in [\mathcal P_4(n)]$.  Hence, $[K]= [a] + [H] + [b] \in [\mathcal P_4(n)]$.
By a simple computation, we have
\begin{align*} 
K&= x_1x_2^{7}x_3^{7}x_4^{30} + x_1^{3}x_2^{5}x_3^{7}x_4^{30} + x_1^{3}x_2^{5}x_3^{30}x_4^{7} + x_1^{3}x_2^{7}x_3^{5}x_4^{30} + x_1^{3}x_2^{29}x_3^{6}x_4^{7}\\ &\quad + x_1^{7}x_2x_3^{7}x_4^{30} + x_1^{7}x_2x_3^{30}x_4^{7} + x_1^{7}x_2^{7}x_3x_4^{30} + x_1^{7}x_2^{7}x_3^{7}x_4^{24}. 
\end{align*} 
All monomials in the right hand sides of the last relation are admissible.

\end{exa}

\begin{khi}\label{th9}
If $\bar y = x_1^{2^s}f_1(y)$ with $y \in (P_{k-1})_{m-2^s}$ and $i=1$, then $[x] \in [\mathcal P(n)]$.
\end{khi}
By Lemmas \ref{bdmu} and \ref{bdcb3}, we need only to prove this case for $s=0$. Note that $r=\ell (I) < d= k-1$. If $r < k -2$, then by Case \ref{th1}, $[x] \in [\mathcal P(n)]$. If $r =k-2$, then $I = (2,\ldots,\hat u, \ldots ,k)$ with $2 \leqslant u \leqslant k$. 

If $u = 2$, then $I = I_2$. Applying  Lemma \ref{bdcbs}(i) with $y_0 =f_1(y)$,  one gets   
$$ x \equiv  \phi_{(2;I_2)}(X^{2^{d}-1})(f_1(x_1y))^{2^{d}} 
 + \sum_{v=3}^k \phi_{(3;I_3)}(X^{2^{d}-1})(f_1(x_{v-1}y))^{2^{d}}.
$$
By Case  \ref{th03}, $\phi_{(2;I_2)}(X^{2^{d}-1})(f_1(x_1y))^{2^{d}} \in [\mathcal P(n)]$. Since $\nu_1(f_1(x_1y))^{2^{d}} =0$, by Case \ref{th7}, $\phi_{(3;I_3)}(X^{2^{d}-1})(f_1(x_{v-1}y))^{2^{d}} \in [\mathcal P(n)]$. Hence, $x \in [\mathcal P(n)]$.

If $u > 2$, then applying  Lemma \ref{bdcbs}(i) with $y_0 =f_1(y)$,  we obtain
$$ x \equiv \sum_{2 \leqslant v \leqslant k} \phi_{(2;I\setminus 2)}(X^{2^{d}-1})(f_1(x_{v-1}y))^{2^{d}}.
$$
Since $\nu_1(f_1(x_{v-1}y)) =0$, this equality and Case  \ref{th8} imply $[x] \in [\mathcal P(n)]$. 

\begin{exa}\label{vdu7} Let $k=4, n, m, B_3(n), K$ be as in Example \ref{vdu6} and let $H$ be as in Example \ref{vdu4b}.  
 Let $I=I_3=(2,3), y_0 =x_2^2, \bar y = x_1y_0= x_1x_2^2$. Then,
$L :=\phi_{(1;(2,3))}(X^{7})(x_1y_0)^{8}\equiv K + F + a,
$
where  $a =\phi_{(2;I_2)}(X^{7})(x_4y_0)^{8}\equiv \phi_{(1;I_1)}(X^{7}(x_1x_3^2)^{8})$.  Hence, $[L]= [K] + [F] + [a] \in [\mathcal P_4(n)]$.
By a simple computation, we have
\begin{align*} 
L&=x_1^{3}x_2^{7}x_3^{5}x_4^{30} + x_1^{3}x_2^{7}x_3^{13}x_4^{22} + x_1^{3}x_2^{13}x_3^{22}x_4^{7} + x_1^{3}x_2^{29}x_3^{6}x_4^{7}\\ &\quad + x_1^{7}x_2^{7}x_3x_4^{30} + x_1^{7}x_2^{7}x_3^{9}x_4^{22} + x_1^{7}x_2^{11}x_3^{5}x_4^{22}. 
\end{align*} 
All monomials in the right hand sides of the last relation are admissible.
\end{exa}

\begin{khi}\label{th10} If $i=2$, then $[x] \in [\mathcal P(n)]$. 
\end{khi}

We have $\bar y = x_1^{a}x_2^bf_2(y)$ for suitable $y \in (P_{k-1})_{m-a-b}$ with $\nu_1(y) = 0$, and $a=\nu_1(\bar y), b=\nu_2(\bar y)$. We prove this case by induction on $b$. By using Lemmas \ref{bdmu} and \ref{bdcb3}, we can assume that $a = 2^s - 1$. 

If $b = 0$, then by Case \ref{th01}, $[x] \in [\mathcal P(n)]$.  Suppose  that $b >0$ and this case is true for $b-1$. 

If $I \ne I_2$, then applying Lemma \ref{bdcbs}(ii) with $y_0 = x_1^ax_2^{b-1}f_2(y)$, we get
\begin{align*} x &\equiv  \phi_{(1;I\cup 2)}(X^{2^{d}-1})(x_1^{2^s}f_1(x_1^{b-1}y))^{2^{d}} \\
&\quad  + \sum_{3 \leqslant v \leqslant k} \phi_{(2;I\cup v)}(X^{2^{d}-1})(x_1^ax_2^{b-1}f_2(x_{v-1}y))^{2^{d}}.
\end{align*}
This equality, Case \ref{th9} and the inductive hypothesis imply $[x] \in [\mathcal P(n)]$. 

If $I = I_2$, then applying Lemma \ref{bdcbs}(i) with $y_0 = x_1^ax_2^{b-1}f_2(y)$ and using Case \ref{th1}, we get
$$ x \equiv  \phi_{(1;I_2)}(X^{2^{d}-1})(x_1^{2^s}f_1(x_1^{b-1}y))^{2^{d}} + Y_3y_0^{2^d} \equiv Y_3y_0^{2^d} \ \text{mod}(\mathcal P(n)).$$

By Lemmas \ref{hq4} and \ref{bdcb3}, we have 
\begin{align*} 
Y_3y_0^{2^d} \equiv \sum_{(j;J)} \phi_{(j;J)}(X^{2^{d}-1})(x_{j}x_1^ax_2^{b-1}f_2(y))^{2^{d}}, \end{align*}
where the last sum runs over some $(j;J)$ with $j = 1,2$, $J\subset I_2$ and $J \ne I_2$.
Since $J \ne I_2$, $\phi_{(2;J)}(X^{2^{d}-1})(x_1^ax_2^{b}f_2(y))^{2^{d}} \in [\mathcal P(n)]$. By Case  \ref{th9}, 
$$\phi_{(1;J)}(X^{2^{d}-1})(x_1x_1^ax_2^{b-1}f_2(y))^{2^{d}} =  \phi_{(1;J)}(X^{2^{d}-1})(x_1^{2^s}f_1(x_1^{b-1}y))^{2^{d}} \in [\mathcal P(n)].$$
This case is proved.
\begin{exa}\label{vdu8} Let $k=4, n, m, B_3(n)$ be as in Example \ref{vdu1} and let $H$ be as in Example \ref{vdu4b}.  
 Let $I=I_3=\emptyset, y_0 =x_1x_3, \bar y = x_2y_0= x_1x_2x_3$. Then,
$M :=\phi_{(2;\emptyset)}(X^{7})(x_1y_0)^{8}\equiv a + b + c,
$
where  $a =\phi_{(1;2)}(X^{7})(x_1y_0)^{8}$, $b=\phi_{(2;3)}(X^{7})(x_3y_0)^{8}$, $c = \phi_{(2;4)}(X^{7})(x_4y_0)^{8}$. By Cases \ref{th01} and \ref{th1}, $[a], [b], [c] \in [\mathcal P_4(n)]$.   Hence, $[M] \in [\mathcal P_4(n)]$.
By a simple computation, we have
\begin{align*} 
M&=x_1^{}x_2^{14}x_3^{23}x_4^{7} + x_1^{3}x_2^{5}x_3^{15}x_4^{22} + x_1^{3}x_2^{5}x_3^{30}x_4^{7} + x_1^{15}x_2^{}x_3^{22}x_4^{7} + x_1^{15}x_2^{1}x_3^{15}x_4^{14}. 
\end{align*} 
All monomials in the right hand sides of the last relation are admissible.

\end{exa}

\begin{khi}\label{th11} 
If  $i=1$, then $[x] \in [\mathcal P(n)]$.
\end{khi}

We have $\bar y = x_1^{a}f_1(y)$ for suitable $y \in (P_{k-1})_{m- a}$ and $a = \nu_1(\bar y)$. We prove this case by induction on $a$. 

If $a=0$, then by Case \ref{th01}, $[x] \in [\mathcal P(n)]$. Suppose that $a > 0$ and this case holds for $a-1$. 

Note that $r=\ell(I) \leqslant d -1 = k-2$. If $r< k-2$, then applying Lemma \ref{bdcbs}(ii) with $y_0 = x_1^{a-1}f_1(y)$, we get
$$x\equiv \sum_{v=2}^k \phi_{(1;I\cup v)}(X^{2^{d}-1})(x_{1}^{a-1}f_1(x_{v-1}y))^{2^{d}}. $$
Hence,  by the inductive hypothesis, we obtain $[x] \in [\mathcal P(n)]$. 

Suppose that $r = k - 2$. Then, $ I = (2,\ldots,\hat u, \ldots ,k)$ with $2\leqslant u \leqslant k$. If $u = 2$, then applying Lemma \ref{bdcbs}(i) with $y_0 = x_1^{a-1}f_1(y)$,  Lemma \ref{bdcb3} and Case \ref{th10}, we get
$$ x \equiv \phi_{(2;I_2)}(X^{2^{d}-1})(x_{1}^{a-1}f_1(x_{1}y))^{2^{d}}+ Y_3y_0^{2^d}\equiv Y_3y_0^{2^d}\ \text{ mod}(\mathcal P_k(n)).$$

By Lemmas \ref{hq4} and \ref{bdcb3}, we have
\begin{align*}
Y_3y_0^{2^d}\equiv \sum_{(j;J)} \phi_{(j;J)}(X^{2^{d}-1})(x_jx_{1}^{a-1}f_1(y))^{2^{d}}, 
\end{align*}
where the last sum runs over some $(j;J)$ with $j = 1, 2$, $J \subset I_2$ and $J \ne I_2$.
 By Case \ref{th10},  $\phi_{(2;J)}(X^{2^{d}-1})(x_2x_{1}^{a-1}f_1(y))^{2^{d}}\in [\mathcal P(n)]$. Since $J \ne I_2$, we have $\phi_{(1;J)}(X^{2^{d}-1})(x_{1}^{a}f_1(y))^{2^{d}}  \in [\mathcal P(n)]$. Hence, $x \in [\mathcal P(n)]$.

 If $u > 2$, then applying Lemma \ref{bdcbs}(i) with $y_0 = x_1^{a-1}f_1(y)$, we get
\begin{align*} x \equiv 
\sum_{2 \leqslant v \leqslant k} \phi_{(2;I\cup v\setminus 2)}(X^{2^{d}-1})(x_{1}^{a-1}f_1(x_{v-1}y))^{2^{d}}.\end{align*}

From the last equality and Case \ref{th10}, we obtain $[x] \in [\mathcal P(n)]$. 
\begin{exa}\label{vdu9} Let $k=4, n, m, B_3(n), L$ be as in Example \ref{vdu7}.  
 Let $I=I_3=(2,3), y_0 =x_1^2, \bar y = x_1y_0= x_1^3$. Then,
$N :=\phi_{(1;(2,3))}(X^{7})(x_1y_0)^{8}\equiv L + a + b,
$
where  $a =\phi_{(2;3)}(X^{7})(x_3y_0)^{8}\equiv \phi_{(1;I_1)}(X^{7}(x_1x_3^2)^{8})$ and $b =\phi_{(2;I_2)}(X^{7})(x_4y_0)^{8}\equiv \phi_{(1;I_1)}(X^{7}(x_1x_4^2)^{8})$.  Hence, $[N]= [L] + [a] + [b] \in [\mathcal P_4(n)]$.
By a simple computation, we have
\begin{align*} 
N&=x_1^{3}x_2^{7}x_3^{5}x_4^{30} + x_1^{3}x_2^{7}x_3^{13}x_4^{22} + x_1^{3}x_2^{13}x_3^{22}x_4^{7} + x_1^{3}x_2^{29}x_3^{6}x_4^{7}\\ &\quad + x_1^{7}x_2^{7}x_3x_4^{30} + x_1^{7}x_2^{7}x_3^{9}x_4^{22} + x_1^{7}x_2^{11}x_3^{5}x_4^{22} + x_1^{15}x_2^{}x_3^{22}x_4^{7} + x_1^{15}x_2^{3}x_3^{5}x_4^{22}. 
\end{align*} 
All monomials in the right hand sides of the last relation are admissible.
\end{exa}

\begin{khi}\label{th12}
$[x]\in [\mathcal P(n)]$ for all $\bar y \in (P_{k-1})_m$.
\end{khi}
 We have $\bar y = x_i^af_i(y)$ with $a=\nu_i(\bar y)$ and $y = p_{(i;\emptyset)}(\bar y/x_i^a)\in (P_{k-1})_{m-a}$. 
 We prove this case by double induction on $(i,a)$. If $a = 0$, then by Case \ref{th01}, $[x] \in [\mathcal P(n)]$ for all $i$. Suppose $a >0$ and this case holds for $a-1$ and all $i$.

If $i = 1, 2$, then by Cases \ref{th10} and \ref{th11}, $[x] \in [\mathcal P(n)]$ for all $a$. 
Suppose  $i > 2$ and assume this case already  proved in the subcases $1, 2,\ldots  , i-1$. Then,  $r + 1 \leqslant k-i + 1 < k-1=d$. Applying Lemma \ref{bdcbs}(ii) with $y_0 = x_i^{a-1}f_i(y)$, we obtain
 $$ x \equiv \sum_{1 \leqslant j < i}\phi_{(j;I\cup i)}(X^{2^{d}-1})y_j^{2^{d}}+ \sum_{i < j \leqslant k}\phi_{(i;I\cup j)}(X^{2^{d}-1})(x_i^{a-1}f_i(x_{j-1}y)^{2^d}.$$
Using the inductive hypothesis, we have $\phi_{(j;I\cup i)}(X^{2^{d}-1})y_j^{2^d} \in [\mathcal P(n)]$ for $j<i$, and 
$\phi_{(i;I\cup j)}(X^{2^{d}-1})(x_i^{a-1}f_i(x_{j-1}y)^{2^d}\in [\mathcal P(n)] $ for $j>i$. Hence, $[x]\in [\mathcal P(n)]$.

\medskip
Now we prove that the set $[B_{k}(n)]$ is linearly independent in $QP_k$. Suppose there is a linear relation
\begin{equation}\label{cmmdc1}
\mathcal S =\sum_{((i;I),z) \in \mathcal N_k\times B_{k-1}(n)}\gamma_{(i;I),z}\phi_{(i;I)}(z) \equiv 0,
\end{equation}
where $\gamma_{(i;I),z} \in \mathbb F_2$. 

If $d \geqslant k$, then  by induction on $\ell(I)$, we can show that $\gamma_{(i;I),z} =0$, for all $(i;I) \in \mathcal N_k$ and $z \in B_{k-1}(n)$ (see \cite{su} for the case $d > k$). 

Suppose that $d = k - 1$.  By Lemma \ref{bdc1}, the homomorphism $p_{j} = p_{(j;\emptyset)}$ sends the relation (\ref{cmmdc1}) to  
$\sum_{z \in B_{k-1}(n)}\gamma_{(j;\emptyset),z}z \equiv 0.$
This relation implies $\gamma_{(j;\emptyset),z} =0$ for any $1 \leqslant j \leqslant k$ and $z \in B_{k-1}(n)$. 

Suppose $0< \ell(J) < k-3$ and $\gamma_{(i;I),z}=0$ for all $(i;I) \in \mathcal N_k$ with $\ell(I)<\ell(J)$, $1\leqslant i\leqslant k$ and $z \in B_{k-1}(n)$. Then, using Lemma \ref{bdc1} and the relation (\ref{ctth}), we see that the homomorphism $p_{(j,J)}$ sends the relation (\ref{cmmdc1}) to $\sum_{z \in B_{k-1}(n)}\gamma_{(j;J),z}z \equiv 0.$ Hence,  we get $\gamma_{(j;J),z}=0$ for all $z \in B_{k-1}(n)$.

Now, let $(j;J)\in \mathcal N_k$ with $\ell(J) = k-3$. If $J \ne I_3$, then using Lemma \ref{bdc1}, we have $p_{(j;J)}(\phi_{(i;I)}(z)) \equiv 0$ for all $z \in B_{k-1}(n)$ and $(i;I) \in \mathcal N_k$ with $(i;I) \ne (j;J)$. So, we get
$$p_{(j;J)}(\mathcal S) \equiv \sum_{z \in B_{k-1}(n)}\gamma_{(j;J),z}z \equiv 0.$$
Hence,  $\gamma_{(j;J),z} = 0$, for all $z \in B_{k-1}(n)$. 

According to  Lemma \ref{bdc1}, 
$p_{(j;I_3)}(\phi_{(1;I_1)}(z)) \equiv 0$ for $z \in \mathcal C$ and $p_{(j;I_3)}(\phi_{(1;I_1)}(z)) \in \langle \mathcal E\rangle$ for $z \in \mathcal D \cup  \mathcal E$. Hence,  we obtain
$$p_{(j;I_3)}(\mathcal S) \equiv \sum_{z \in \mathcal C \cup \mathcal D}\gamma_{(j;I_3),z}z \equiv 0 \ \text{ mod }\langle \mathcal E\rangle.$$
So, we get $\gamma_{(j;I_3),z} = 0$ for all $z \in \mathcal C \cup  \mathcal D$.

Now, let $(j;J)\in \mathcal N_k$ with $\ell(J) = k-2$. Suppose that $I_3 \not\subset J$. Then,  using Lemma \ref{bdc1}, we have $p_{(j;J)}(\phi_{(1;I_1)}(z)) \equiv 0$ for all $z \in \mathcal B$. Hence,  we get
$$p_{(j;J)}(\mathcal S) \equiv \sum_{z \in \mathcal B}\gamma_{(j;J),z}z \equiv 0 .$$
From this, we obtain $\gamma_{(j;J),z} = 0$ for all $z \in \mathcal B$. 

Suppose that $I_3 \subset J$. Then, either $J = I_2, j = 1,2$ or $J = I_3\cup 2, j = 1$. According to  Lemma \ref{bdc1}, $p_{(j;I_2)}(\phi_{(1;I_1)}(z)) \in \langle \mathcal D \cup \mathcal E\rangle$ for all $z \in \mathcal B$, $p_{(j;I_3\cup 2)}(\phi_{(1;I_1)}(z)) \equiv 0$ for $z \in \mathcal C \cup \mathcal D$ and $p_{(1;I_3\cup 2)}(\phi_{(1;I_1)}(z)) \in \langle \mathcal E\rangle$ for $z \in \ \mathcal E$. Hence,  we obtain
\begin{align*} p_{(j;I_2)}(\mathcal S) &\equiv \sum_{z \in \mathcal C }\gamma_{(j;I_2),z}z \equiv 0 \ \text{ mod }\langle \mathcal D\cup \mathcal E\rangle,\\
p_{(1;I_3\cup 2)}(\mathcal S) &\equiv \sum_{z \in \mathcal C\cup  \mathcal D}\gamma_{(1;I_3\cup 2),z}z \equiv 0 \ \text{ mod }\langle \mathcal E\rangle.
\end{align*}
So, $\gamma_{(j;I_2),z} = 0$ for $z \in  \mathcal C$ and $\gamma_{(1;I_3\cup 2),z} = 0$ for $z \in  \mathcal C \cup \mathcal D$. Since $\gamma_{(i;I),z} = 0$, for all $z\in\mathcal C $ and $I \ne I_1$, applying Lemma \ref{bdc1}, we have 
 $$p_{(1;I_1)}(\mathcal S) \equiv \sum_{z \in \mathcal C }\gamma_{(1;I_1),z}z \equiv 0 \ \text{ mod }\langle \mathcal D\cup \mathcal E\rangle.$$
Hence,  $\gamma_{(1;I_1),z} = 0$ for all $z \in  \mathcal C$. So, the relation  (\ref{cmmdc1}) becomes
 \begin{align}\label{ctbs}\mathcal S \notag&= \sum_{1 \leqslant i \leqslant 3, z \in \mathcal E}\gamma_{(i;I_3),z}\phi_{(i;I_3)}(z) + \sum_{z \in  \mathcal E}\gamma_{(1;I_3\cup 2),z}\phi_{(1;I_3\cup 2)}(z)\\ &\quad+ \sum_{1 \leqslant i \leqslant 2, z \in \mathcal D \cup \mathcal E}\gamma_{(i;I_2),z}\phi_{(i;I_2)}(z) + \sum_{ z \in \mathcal D \cup \mathcal E}\gamma_{(1;I_1),z}\phi_{(1;I_1)}(z)\equiv 0.
\end{align}
Using the relation (\ref{ctbs}) and Lemma \ref{bdc1}, 
 $$p_{(i;I_2)}(\mathcal S) \equiv \sum_{z \in \mathcal D}(\gamma_{(i;I_2),z}+\gamma_{(1;I_1),z})z \equiv 0 \ \text{ mod }\langle \mathcal E\rangle, \ i = 1, 2.$$
This relation implies $\gamma_{(1;I_2),z} = \gamma_{(2;I_2),z} = \gamma_{(1;I_1),z}$ for all $z \in \mathcal D$. On the other hand, using the relation (\ref{ctbs}) and Lemma \ref{bdc1}, one gets
$$p_{(1;I_1)}(\mathcal S) \equiv \sum_{z \in \mathcal D}(\gamma_{(1;I_2),z}+ \gamma_{(2;I_2),z} +\gamma_{(1;I_1),z})z \equiv 0 \ \text{ mod }\langle \mathcal E\rangle.$$
So, $\gamma_{(1;I_2),z}+ \gamma_{(2;I_2),z} +\gamma_{(1;I_1),z} = 0$. Hence, $\gamma_{(1;I_2),z} = \gamma_{(2;I_2),z} = \gamma_{(1;I_1),z} = 0$, for all $z \in \mathcal D$. Now, the relation (\ref{ctbs}) becomes
 \begin{align}\label{ctbs1}\notag\mathcal S &= \sum_{1 \leqslant i \leqslant 3, z \in \mathcal E}\gamma_{(i;I_3),z}\phi_{(i;I_3)}(z) + \sum_{z \in  \mathcal E}\gamma_{(1;I_3\cup 2),z}\phi_{(1;I_3\cup 2)}(z)\\ 
&\quad+ \sum_{1 \leqslant i \leqslant 2, z \in  \mathcal E}\gamma_{(i;I_2),z}\phi_{(i;I_2)}(z) + \sum_{ z \in \mathcal E}\gamma_{(1;I_1),z}\phi_{(1;I_1)}(z)\equiv 0.\end{align}
Using the relation (\ref{ctbs1}) and Lemma \ref{bdc1}, one gets
 \begin{align*}
p_{(i;I_3)}(\mathcal S) &\equiv \sum_{z \in \mathcal E}(\gamma_{(i;I_3),z}+\gamma_{(1;I_1),z})z \equiv 0,  \ i = 1, 2, 3,\\
p_{(1;I_3\cup 2)}(\mathcal S) &\equiv \sum_{z \in \mathcal E}(\gamma_{(1;I_3),z}+ \gamma_{(2;I_3),z}+ \gamma_{(1;I_3\cup 2),z}+\gamma_{(1;I_1),z})z \equiv 0,\\
p_{(j;I_2)}(\mathcal S) &\equiv \sum_{z \in \mathcal E}(\gamma_{(j;I_3),z} + \gamma_{(3;I_3),z} + \gamma_{(j;I_2),z} + \gamma_{(1;I_1),z})z \equiv 0,\ j= 1, 2,\\
p_{(1;I_1)}(\mathcal S) &\equiv \sum_{z \in \mathcal E}(\gamma_{(1;I_3),z} + \gamma_{(2;I_3),z} + \gamma_{(3;I_3),z}\\
&\qquad + \gamma_{(1;I_2),z}+\gamma_{(2;I_2),z} + \gamma_{(1;I_3\cup 2),z}+ \gamma_{(1;I_1),z})z \equiv 0.
\end{align*}

From the above relations, we get 
$$\gamma_{(i;I_3),z} = \gamma_{(j;I_2),z} = \gamma_{(1;I_3\cup 2),z}= \gamma_{(1;I_1),z} = 0$$ 
for all $z \in \mathcal E, \ i = 1, 2, 3, \ j = 1,2$. 
The proposition is completely proved. \hfill $\square$

\medskip
Let $n =\sum_{1 \leqslant i \leqslant k-1}(2^{d_i}-1)$ 
with $d_i$ positive integers such that $d_1 > d_2 > \ldots >d_{k-2} \geqslant d_{k-1}=d>0,$ and let $m = \sum_{1 \leqslant i \leqslant k-2}(2^{d_i-d}-1)$. Set $p = \min\{k,d\}$ and $\mathcal N_{k,p} = \{(i;I) \in \mathcal N_k: \ell(I)<p\}$. Then, we have $|\mathcal N_{k,p}| = \sum_{1\leqslant u \leqslant  p}\binom ku$. From the proof of Proposition \ref{mdc1}, we see that the set 
$\left[\bigcup_{(i;I)\in \mathcal N_{k,p}}\phi_{(i;I)} (B_{k-1}(n))\right]$
is linearly independent in $QP_k.$
So, one gets the following.
\begin{corls}[Mothebe \cite{mo}] 
${\dim (QP_k)_n \geqslant \sum_{u=1}^p\binom ku\dim (QP_{k-1})_m.}$
\end{corls}

\subsection{Proof of Lemmas \ref{hq0} and \ref{bdcbs}}\label{sub32}\

\medskip
We need the following lemma. 

\begin{lems}\label{bdad} Let $i, j, d, a, b$ be positive integers such that $i, j\leqslant k$, $i \ne j$, and $a+b = 2^d-1$. Then
 $$x:= X_i^aX_j^b \simeq_2 X_i^{2^d-2}X_j.$$
\end{lems}
\begin{proof} We prove  the lemma by induction on $b$. If $b=1$, then 
$$ x=X_i^aX_j^b = X_i^{2^d-2}X_j.$$ So, the lemma holds. Suppose that $b > 1$ and the lemma holds for $b-1$. Since $i\ne j$, 
 $X_i^aX_j^b = x_i^bx_j^aX_{\{i,j\}}^{2^d-1}$. 
 
 If $\alpha_0(b) = 0$, then $\alpha_0(a) = \alpha_0(2^d-1-b) =1$. Using the Cartan formula and the inductive hypothesis, we have
\begin{align*}x &\simeq_0  Sq^1(x_i^{b-1}x_j^{a}X_{\{i,j\}}^{2^d-1}) + x_i^{b-1}x_j^{a+1}X_{\{i,j\}}^{2^d-1}
\simeq_1 X_i^{a+1}X_j^{b-1} \simeq_2 X_i^{2^d-2}X_j .
\end{align*}  

By an argument analogous to the previous one, we see that if $\alpha_0(b) = 1, \alpha_1(b) = 0$, then 
\begin{align*} 
x &\simeq_0  Sq^1(x_i^{b-2}x_j^{a+1}X_{\{i,j\}}^{2^d-1}) + Sq^2(x_i^{b-2}x_j^{a}X_{\{i,j\}}^{2^d-1})  +x_i^{b-1}x_j^{a+1}X_{\{i,j\}}^{2^d-1}\\
& \simeq_2 x_i^{b-1}x_j^{a+1}X_{\{i,j\}}^{2^d-1} =  X_i^{a+1}X_j^{b-1} \simeq_2 X_i^{2^d-2}X_j .
\end{align*}  

If $\alpha_0(b) = \alpha_{1}(b) = 1$, then  
\begin{align*}
x &\simeq_0 Sq^1(x_i^{b}x_j^{a-1}X_{\{i,j\}}^{2^d-1}) + Sq^2(x_i^{b-1}x_j^{a-1}X_{\{i,j\}}^{2^d-1}) + x_i^{b-1}x_j^{a+1}X_{\{i,j\}}^{2^d-1}\\
& \simeq_2 x_i^{b-1}x_j^{a+1}X_{\{i,j\}}^{2^d-1} =  X_i^{a+1}X_j^{b-1} \simeq_2 X_i^{2^d-2}X_j.
\end{align*}

The lemma is proved. 
\end{proof}

\begin{lems}\label{bdbss2} Let $(i;I) \in \mathcal N_k$ and let $d, h, u$ be   integers such that $\ell(I)= r < h \leqslant d$, and $1 \leqslant i< u \leqslant k$. Then, we have
$$Y:=\phi_{(i;I)}(X^{2^{h}-1})X_u^{2^{d }-2^{h}} \simeq_{r+2}\phi_{(i;I\cup u)}(X^{2^d-1}).$$
\end{lems}
\begin{proof} We prove the lemma by induction on $r$. If $r=0$, then using Lemma \ref{bdad}, we have
$Y = X_i^{2^h-1}X_u^{2^{d }-2^{h}}\simeq_2 X_i^{2^d-2}X_u=\phi_{(i;u)}(X^{2^d-1}).$ So, the lemma holds. 

Suppose that $r > 0$, $I= (i_1,i_2,\ldots, i_r)$ and the lemma is true for $r-1$. A direct computation, using Lemma \ref{bdad} and Proposition \ref{mdcb4}(ii), shows
\begin{align*}Y &= \phi_{(i_1;I\setminus i_1)}(X^{2^{r}-1})(X_i^{2^{h-r}-1}X_u^{2^{d-r }-2^{h-r}})^{2^{r}}\\ &\simeq_{r+2} \phi_{(i_1;I\setminus i_1)}(X^{2^{r}-1})(X_uX_i^{2^{d-r}-2})^{2^{r}}\\ &= \Big(\phi_{(i_1;I\setminus i_1)}(X^{2^{r}-1})X_u^{2^{r}}\Big)X_i^{2^{d}-2^{r+1}}:= Z.
\end{align*}
If $u < i_1$, then $Z = \phi_{(i;I\cup u)}(X^{2^d-1})$. Suppose $u \geqslant i_1$.

If $u \notin I$, then $u>i_1$. By the inductive hypothesis and Proposition \ref{mdcb4}(i), we have
\begin{align*}  Z \simeq_{r+1} \phi_{(i_1;I\cup u\setminus i_1)}(X^{2^{r+1}-1})X_i^{2^{d}-2^{r+1}} = \phi_{(i;I\cup u)}(X^{2^d-1}).
\end{align*}
If $u \in I$ and $r =1$, then $u=i_1$. Using Lemma \ref{bdad}, we have
\begin{align*}
 Z = X_i^{2^{d}-2^{r+1}}X_{i_1}^{2^{r+1}-1}\simeq_2 X_i^{2^d-2}X_{i_1}= \phi_{(i;I\cup u)}(X^{2^d-1}).
\end{align*}
Suppose that $u \in I$ and $r>1$. If $u=i_1$, then 
$Z = \phi_{(i_1;I\setminus i_1)}(X^{2^{r+1}-1})X_i^{2^{d}-2^{r+1}}.$
If $u >i_1$, then using the inductive hypothesis and Proposition \ref{mdcb4}(i), we obtain
\begin{align*} Z \simeq_{r+1} \phi_{(i_1;I\cup u\setminus i_1)}(X^{2^{r+1}-1})X_i^{2^{d}-2^{r+1}} = \phi_{(i_1;I\setminus i_1)}(X^{2^{r+1}-1})X_i^{2^{d}-2^{r+1}}.
\end{align*}
Now, applying Lemma \ref{bdad} and Proposition \ref{mdcb4}(ii), one gets
\begin{align*}
&\phi_{(i_1;I\setminus i_1)}(X^{2^{r+1}-1})X_i^{2^{d}-2^{r+1}}= \phi_{(i_2;I\setminus \{i_1,i_2\})}(X^{2^{r-1}-1})(X_{i_1}^{3}X_i^{2^{d-r+1}-4})^{2^{r-1}} \\
&\simeq_{r+1} \phi_{(i_2;I\setminus \{i_1,i_2\})}(X^{2^{r-1}-1})(X_{i_1}X_i^{2^{d-r+1}-2})^{2^{r-1}} 
=\phi_{(i;I\cup u)}(X^{2^{d}-1}).
\end{align*}
The above equalities imply $Y\simeq_{r+2} \phi_{(i;I\cup u)}(X^{2^{d}-1})$. The lemma is proved.
\end{proof}

\begin{proof}[Proof of Lemma \ref{hq0}] We prove the lemma by induction on $d$. For $d=1$, the lemma holds with $i=j_0$ and $I= \emptyset$.

Let  $d = 2$. If $j_0 = j_1=i$, then $x = \phi_{(i,\emptyset)}(X^3)$. If $j = j_0 > j_1=i$, then $x = X_i^2X_j = \phi_{(i,j)}(X^3)$. If $i = j_0 < j_1=j$, then 
$$x = X_iX_j^2\simeq_0 Sq^1(X_\emptyset X_{\{i,j\}}^2) + X_i^2X_j \simeq_1 X_i^2X_j = \phi_{(i,j)}(X^3).$$
So, the lemma holds for $d = 2$. 

Suppose $d > 2$ and the lemma holds for $d-1$. By the inductive hypothesis, there is $(h;H) \in \mathcal N_k$ such that 
$$\prod_{0 \leqslant t <d-1}X_{j_t}^{2^t} \simeq_{d-2} \phi_{(h;H)}(X^{2^{d-1}-1}),$$
where $h = \min\{j_0,j_1,\ldots , j_{d-2}\}$. If $j_{d-1} = h$, then the lemma holds with $(i;I) = (h;H)$. Assume that $j_{d-1} \ne h$. 
If $H =\emptyset $, then $j_0 = j_1=\ldots = j_{d-2}=h$. Using Proposition \ref{mdcb4}(i) and Lemma \ref{bdad}, we obtain
$$x \simeq_{d-1} X_{h}^{2^{d-1}-1} X_{j_{d-1}}^{2^{d-1}} \simeq_2   X_{i}^{2^{d}-1} X_{i_1}=\phi_{(i;i_1)}(X^{2^d-1}),$$ 
where $i = \min\{h, j_{d-1}\} = \min\{j_0,j_1,\ldots , j_{d-1}\}$, $i_1 = \max\{h, j_{d-1}\}$. So, the lemma is true. 

Suppose $H = (h_1, \ldots, h_s)$ with $\{h_1,\ldots, h_s\} =\{j_0,\ldots , j_{d-2}\}\setminus \{h\}$ and $0 < s < d-1$. Then, we have 
$\phi_{(h;H)}(X^{2^{d-1}-1}) = \phi_{(h_1;H\setminus h_1)}(X^{2^{s}-1})X_{h}^{2^{d-1}-2^{s}}.$
If $h < j_{d-1}$ and $s <d-2$, then by Lemma \ref{bdbss2}, we have 
\begin{align*}x \simeq_{d-2} \phi_{(h;H)}(X^{2^{d-1}-1})X_{j_{d-1}}^{2^{d-1}}\simeq_{s+2} \phi_{(h;H\cup j_{d-1})}(X^{2^{d}-1}).
\end{align*}
Since $s+2 \leqslant d-1$, the lemma is true with $i=h$ and $I = H\cup j_{d-1}$.

Suppose that $h < j_{d-1}$ and $s = d-2$. From the proof of Lemma \ref{bdad}, we have  $X_{h}X_{j_{d-1}}^{2}\simeq_{1}X_{h}^{2}X_{j_{d-1}} $. Hence, using Proposition \ref{mdcb4}, one gets
\begin{align*}x &\simeq_{d-2} \phi_{(h_1;H\setminus h_1)}(X^{2^{d-2}-1})(X_{h}X_{j_{d-1}}^{2})^{2^{d-2}}\\ 
&\simeq_{d-1} \phi_{(h_1;H\setminus h_1)}(X^{2^{d-2}-1})(X_{h}^{2}X_{j_{d-1}})^{2^{d-2}}\\ 
&= \phi_{(h_1;H\setminus h_1)}(X^{2^{d-2}-1})X_{j_{d-1}}^{2^{d-2}}X_{h}^{2^{d-1}} .
\end{align*}

By the inductive hypothesis, we have
\begin{align*} \phi_{(h_1;H\setminus h_1)}(X^{2^{d-2}-1})X_{j_{d-1}}^{2^{d-2}}\simeq_{d-2} \phi_{(j;J)}(X^{2^{d-1}-1}),
\end{align*}
where $j = \min\{h_1, \ldots, h_s,j_{d-1}\} = \min(\{j_0,\ldots, j_{d-1}\}\setminus \{h\})$, $J = (H\cup j_{d-1})\setminus j$.

If $j_{d-1} \notin H$, then from the above equalities and Proposition \ref{mdcb4}(i), we get 
$$ x\simeq_{d-1} \phi_{(j;J)}(X^{2^{d-1}-1})X_{h}^{2^{d-1}}= \phi_{(h;J\cup j)}(X^{2^{d}-1}).$$ 
The lemma holds with $i=h$, $I = J\cup j = H\cup j_{d-1}$. 

If $j_{d-1} \in H$, then $\ell(J) = d-3$. By Lemma \ref{bdad}, $X_j^3X_h^4 \simeq_{2} X_jX_h^6$. Using Proposition \ref{mdcb4}, we obtain
\begin{align*}x &\simeq_{d-1} \phi_{(j;J)}(X^{2^{d-1}-1})X_{h}^{2^{d-1}}= \phi_{(j_1;J\setminus j_1)}(X^{2^{d-3}-1}) (X_j^3X_h^4)^{2^{d-3}}\\ & \simeq_{d-1}\phi_{(j_1;J\setminus j_1)}(X^{2^{d-3}-1}) (X_jX_h^6)^{2^{d-3}} = \phi_{(h;J\cup j)}(X^{2^{d}-1}) .
\end{align*}
Here $j_1 = \min J$. The lemma holds with $i=h$ and $I = J\cup j =H= H\cup j_{d-1}$.

If $h > j_{d-1}$ and $s = d-2$, then using  Proposition \ref{mdcb4}(i), we have
\begin{align*}x \simeq_{d-1} \phi_{(h;H)}(X^{2^{d-1}-1})X_{j_{d-1}}^{2^{d-1}}
= \phi_{(j_{d-1};H\cup h)}(X^{2^{d}-1}) .
\end{align*}

Suppose that $h > j_{d-1}$ and $s <d-2$. By Lemma \ref{bdad}, 
$X_{h}^{2^{d-s-1}-1}X_{j_{d-1}}^{2^{d-s-1}}\simeq_{2} X_{h}X_{j_{d-1}}^{2^{d-s}-2}.$ 
Since $s + 2<d$,  using Proposition \ref{mdcb4}, we obtain
\begin{align*}x &\simeq_{d-1} \phi_{(h_1;H\setminus h_1)}(X^{2^{s}-1})(X_{h}^{2^{d-s-1}-1}X_{j_{d-1}}^{2^{d-s-1}})^{2^s}\\
&\simeq_{d-1} \phi_{(h_1;H\setminus h_1)}(X^{2^{s}-1})(X_{h}X_{j_{d-1}}^{2^{d-s}-2})^{2^s}
= \phi_{(j_{d-1};H\cup h)}(X^{2^{d}-1}).
\end{align*}
The lemma holds with $i = j_{d-1}$, $I= H\cup h$.
\end{proof}

\begin{proof}[Proof of Lemma \ref{bdcbs}]  
Applying the Cartan formula, we have
$$Sq^1(X_\emptyset^{2^c-1}y_0^{2^c}) = \sum _{1\leqslant j \leqslant k} X_j^{2^c-1}y_j^{2^c},$$
where $c$ is a positive integer.  From this, we obtain
\begin{equation}\label{ctpbs}X_{i}^{2^{c}-1}y_i^{2^c} \equiv \sum _{1\leqslant j < i} X_j^{2^c-1}y_j^{2^c} + \sum _{i < j \leqslant k} X_j^{2^c-1}y_j^{2^c}.\end{equation}

If $r=0$, then $t_j = j$ and $I^{(j)} = \emptyset$ for $j>i$. Then, the first part of the lemma follows from the relation (\ref{ctpbs}) with $c =d$.

If $d>r>0$, then $\phi_{(i;I)}(X^{2^{d}-1})y_i^{2^d} = \phi_{(i_1;I\setminus i_1)}(X^{2^{r}-1})(X_i^{2^{c}-1}y_i^{2^c})^{2^{r}}$, with $c = d - r>0$ and $i_1 = \min I$. Hence,  using Lemma \ref{bdcb3} and the relation (\ref{ctpbs}), we get
\begin{align*}\phi_{(i;I)}(X^{2^{d}-1})y_i^{2^d} &\equiv \sum _{1\leqslant j < i}\phi_{(i_1;I\setminus i_1)}(X^{2^{r}-1})( X_j^{2^c-1}y_j^{2^c})^{2^{r}}\\ 
&\quad+ \sum _{i < j \leqslant k} \phi_{(i_1;I\setminus i_1)}(X^{2^{r}-1})(X_j^{2^c-1}y_j^{2^c})^{2^{r}}. 
\end{align*}
A simple computation, using Lemmas \ref{bdbss2} and \ref{bdcb3}, shows
\begin{align*} &\phi_{(i_1;I\setminus i_1)}(X^{2^{r}-1})( X_j^{2^c-1}y_j^{2^c})^{2^{r}} = \phi_{(j;I)}(X^{2^{d}-1})y_j^{2^d}, \text{ for $j < i$},\\
&\phi_{(i_1;I\setminus i_1)}(X^{2^{r}-1})( X_j^{2^c-1}y_j^{2^c})^{2^{r}}\equiv \phi_{(t_j;I^{(j)})}(X^{2^{d}-1})y_j^{2^d}, \text{ for $j > i$}.
\end{align*}
Hence, the first part of the lemma follows.

If $d>r+1$, then $\phi_{(i;I)}(X^{2^{d}-1})y_i^{2^d} = \phi_{(i;I)}(X^{2^{r+1}-1})(X_i^{2^{c}-1}y_i^{2^c})^{2^{r+1}}$, with $c = d - r - 1>0$. Hence,  using the relation (\ref{ctpbs}) and Lemma \ref{bdcb3}, we get
\begin{align*}\phi_{(i;I)}(X^{2^{d}-1})y_i^{2^d} &\equiv \sum _{1\leqslant j < i}\phi_{(i;I)}(X^{2^{r+1}-1})( X_j^{2^c-1}y_j^{2^c})^{2^{r+1}}\\ 
&\quad + \sum _{i < j \leqslant k} \phi_{(i;I)}(X^{2^{r+1}-1})(X_j^{2^c-1}y_j^{2^c})^{2^{r+1}}. 
\end{align*}
Applying Lemmas \ref{bdbss2} and \ref{bdcb3}, we have
\begin{align*} &\phi_{(i;I)}(X^{2^{r+1}-1})( X_j^{2^c-1}y_j^{2^c})^{2^{r+1}} = \phi_{(j;I\cup i)}(X^{2^{d}-1})y_j^{2^d}, \text{ for $j < i$},\\
&\phi_{(i;I)}(X^{2^{r+1}-1})( X_j^{2^c-1}y_j^{2^c})^{2^{r+1}}\equiv \phi_{(i;I\cup j)}(X^{2^{d}-1})y_j^{2^d}, \text{ for $j > i$}.
\end{align*}
So, the second part of the lemma is proved.
\end{proof}

\subsection{Proof of Lemma \ref{hq4}}\label{sub33}\

\medskip

In this subsection, we denote $\omega_{(t,k,d)} = \omega((x_{t+1}\ldots x_k)^{2^d-1}x_t^{2^d})$ with $d \geqslant k-t +1$,  $P_{(t,k)} = \mathbb F_2[x_t,\ldots ,x_k]\subset P_{(1,k)} = P_k$, and $n_t =\deg(\omega_{(t,k,d)}) = (k-t)(2^d-1) +2^d.$ 
\begin{lems}\label{bdbs3} Let $x$ be a monomial of degree $n_t$ in $P_{(t,k)}$. If $x \notin P_{(t,k)}^-(\omega_{(t,k,d)})$, then $\omega(x) = \omega_{(t,k,d)}$.
\end{lems}
The proof of the lemma is similar to the one of Lemma \ref{bdbs1} (see the proof of Lemma 2.14 in \cite{su} for $t=1$).

Denote by $R_{(t,k,d)}$ the subspace of $P_{(t,k)}$ spanned by all $Sq^{2^i}(y)$ with $0\leqslant i \leqslant k-t$, where $y$ is a monomial of degree $n_t-2^i$ in $P_{(t,k)}$ such that $\omega_{j}(y) >0$ for some $j>d$.

For $J= (j_1,j_2, \ldots, j_{i})$, $t\leqslant j_1 < j_2 < \ldots <j_{i}\leqslant k$, $0\leqslant i \leqslant k-t+1$, we set 
$$y_{(J,t,k,d)} =(x_tx_{t+1}\ldots x_k)^{2^d-1}/ (x_{j_1}^{2^{i-1}}x_{j_2}^{2^{i-2}}\ldots x_{j_{i}}^{2^0}).$$ 
Here, by convention, $x_{j_1}^{2^{i-1}}x_{j_2}^{2^{i-2}}\ldots x_{j_{i}}^{2^0} = 1$ for $i=0$. Denote 
$$Y_{(t,k,d)} = \sum_{u=t}^kx_t^{2^{k-t}-1}y_{(I_t,t+1,k,d)}x_u^{2^d}.$$

\begin{lems}\label{bdbss} For $d \geqslant k\geqslant 2$, $\omega = \omega_{(1,k,d)}$, and $J^* = (1,2,\ldots, k-1)$,
$$\phi_{(1;I_1)}(X^{2^d-1})x_k^{2^d} \simeq_{(0,\omega)} Y_{(1,k-1,d)}x_k^{2^d-1} + Sq^{2^{k-1}}(y_{(J^*,1,k,d)})\ \text{\rm mod}(R_{(1,k,d)}).$$
\end{lems}

\begin{proof} We observe that $\phi_{(1;I_1)}(X^{2^d-1})x_k^{2^d} = \big(\phi_{(1;I_1)}(X^{2^k-1})x_k^{2^k}\big)(X_1x_k)^{2^d-2^k}$ and $\phi_{(1;I_1)}(X^{2^k-1})x_k^{2^k} = \Big(\prod_{j=1}^kX_j^{2^{k-j}}\Big)x_k^{2^k}$. By a direct computation using the Cartan formula, we get
\begin{align*}\phi_{(1;I_1)}(X^{2^k-1})x_k^{2^k} &\simeq_{(0,\omega_{(1,k,k)})} \sum_{u=1}^{k-1}\Big(\prod_{j=1}^{k-1}X_j^{2^{k-j-1}}\Big)X_u^{2^{k-1}}x_u^{2^k} + \sum_{i=0}^{k-1}Sq^{2^i}(W_i),
\end{align*}
where $W_i = (\prod_{j=1}^{ k-i-1}X_j^{2^{k-j}})(\prod_{j=k-i }^{k-1}X_j^{2^{k-j-1}})X_\emptyset^{2^i}x_k^{2^k-2^{i+1}}$. Using Lemma \ref{hq0}, we have $\Big(\prod_{j=1}^{k-1}X_j^{2^{k-j-1}}\Big)X_u^{2^{k-1}}\simeq_{k-1} \phi_{(1;I_1^*)}(X^{2^k-1})$, for $1< u <k$ and $I_1^* = (2,3, \ldots, k-1)$. Hence,
\begin{align*}\Big(\prod_{j=1}^{k-1}X_j^{2^{k-j-1}}\Big)X_u^{2^{k-1}} \simeq_{(0,\omega_{(1,k,k)})}
 \phi_{(1;I_1^*)}(X^{2^k-1}) + \sum_{j=0}^{k-2}Sq^{2^j}(g_j), 
\end{align*}
where $g_j$ are suitable polynomials in $P_k$.
 
A simple computation shows that 
$\Big(\prod_{j=1}^{k-1}X_j^{2^{k-j-1}}\Big)X_1^{2^{k-1}} = \phi_{(1;I_1^*)}(X^{2^k-1})$; 
if $g \in P_k^-(\omega_{(1,k,k)})$, then $(X_1x_k)^{2^d-2^k}g  \in P_k^-(\omega)$; $W_{k-1} = y_{(J^*,1,k,k)}$;
\begin{align*}
&Sq^{2^i}(W_i)(X_1x_k)^{2^d-2^k} = Sq^{2^i}(W_i(X_1x_k)^{2^d-2^k})  \in R_{(1,k,d)}, 0\leqslant i \leqslant k-2;\\
&Sq^{2^j}(g_j)x_u^{2^k}(X_1x_k)^{2^d-2^k} = Sq^{2^j}(g_jx_u^{2^k}(X_1x_k)^{2^d-2^k}) \in R_{(1,k,d)},  0\leqslant j \leqslant k-2.
\end{align*}

From the above equalities, we obtain
\begin{align*}\phi_{(1;I_1)}(X^{2^d-1})x_k^{2^d} \simeq_{(0,\omega)} &\sum_{u=1}^{k-1}\phi_{(1;I_1^*)}(X^{2^k-1})x_u^{2^k}(X_1x_k)^{2^d-2^k}\\
 &+ Sq^{2^{k-1}}(y_{(J^*,1,k,k)})(X_1x_k)^{2^d-2^k}\ \text{\rm mod}(R_{(1,k,d)}), 
\end{align*}
Since $\sum_{u=1}^{k-1}\phi_{(1;I_1^*)}(X^{2^k-1})x_u^{2^k} = Y_{(1,k-1,k)}x_k^{2^k-1}$, the lemma is true for $d=k$. 

Suppose that $d>k$. Then, $Sq^{2^{k-1}}(y_{(J^*,1,k,k)}(X_1x_k)^{2^d-2^k})\in R_{(1,k,d)}$. According to Lemma \ref{bdbs3}, $\phi_{(1;I_1^*)}(X^{2^k-1})x_u^{2^k}(X_1x_k)^{2^d-2^k} \in P_k^-(\omega)$, for $1 <u<k$. 

For $u=1$, we have
\begin{align*}
&\phi_{(1;I_1^*)}(X^{2^k-1})x_1^{2^k}(X_1x_k)^{2^d-2^k} = \phi_{(2;I_2^*)}(X^{2^{k-2}-1})(X_1^{2^{d-k+2}-3}X_k^2)^{2^{k-2}}x_k^{2^{d}}.
\end{align*}  
Here $I_2^* = (3,\ldots,k-1)$. By Lemma \ref{bdad}, 
$X_1^{2^{d-k+2}-3}X_k^2\simeq_{2} X_1X_k^{2^{d-k+2}-2}.$
So, using Proposition \ref{mdcb4}, we obtain
\begin{align*}\phi_{(2;I_2^*)}(X^{2^{k-2}-1})(X_1^{2^{d-k+2}-3}X_k^2)^{2^{k-2}} &\simeq_{k} \phi_{(2;I_2^*)}(X^{2^{k-2}-1})(X_1X_k^{2^{d-k+2}-2})^{2^{k-2}}\\
&= \phi_{(1;I_1^*)}(X^{2^{k-1}-1})X_k^{2^{d}-2^{k-1}}.
\end{align*}
This equality implies
\begin{align*}\phi_{(1;I_1^*)}(X^{2^k-1})x_1^{2^k}&(X_1x_k)^{2^d-2^k} \simeq_{(0,\omega)}\\ &\Big(\phi_{(1;I_1^*)}(X^{2^{k-1}-1})X_k^{2^{d}-2^{k-1}} + \sum_{i=0}^{k-1}Sq^{2^i}(p_i)\Big)x_k^{2^d},
\end{align*}
where $p_i$ are suitable polynomials in $P_k$. It is easy to see that $Sq^{2^i}(p_i)x_k^{2^d} = Sq^{2^i}(p_ix_k^{2^d}) \in R_{(1,k,d)}$. So, combining the above equalities gives
\begin{align*}\phi_{(1;I_1)}&(X^{2^d-1})x_k^{2^d} \simeq_{(0,\omega)} \phi_{(1;I_1^*)}(X^{2^{k-1}-1})X_k^{2^{d}-2^{k-1}}x_k^{2^d} \ \text{ mod}(R_{(1,k,d)}).
\end{align*}
Using the Cartan formula, we have
\begin{align*}\phi_{(1;I_1^*)}(X^{2^{k-1}-1})X_k^{2^{d}-2^{k-1}}x_k^{2^{d}} &\simeq_{(0,\omega)} Sq^{2^{k-1}}\Big(\phi_{(1;I_1^*)}(X^{2^{k-1}-1})X_\emptyset^{2^{d}-2^{k-1}}\Big)\\
&\qquad + \sum_{u=1}^{k-1}\phi_{(1;I_1^*)}(X^{2^{k-1}-1})X_u^{2^{d}-2^{k-1}}x_u^{2^{d}}. 
\end{align*}
A simple computation shows: $\phi_{(1;I_1^*)}(X^{2^{k-1}-1})X_\emptyset^{2^{d}-2^{k-1}} = y_{(J^*,1,k,d)}$;
\begin{align*} 
\phi_{(1;I_1^*)}(X^{2^{k-1}-1})X_u^{2^{d}-2^{k-1}}&=\phi_{(1;I_1^*)}\big(X_{\{k-1,k\}}^{2^{k-1}-1}\big)X_{\{u,k\}}^{2^{d}-2^{k-1}}x_k^{2^d-1};\\
\phi_{(1;I_1^*)}\big(X_{\{k-1,k\}}^{2^{k-1}-1}\big)X_{\{1,k\}}^{2^{d}-2^{k-1}}&= \phi_{(1;I_1^*)}\big(X_{\{k-1,k\}}^{2^{d}-1}\big).
\end{align*} 
For $1<u<k$, $I_1^*\cup u = I_1^*$.  Since $\ell(I_1^*)=k-2$, applying Lemma \ref{bdbss2} for $P_{k-1}$, 
$$\phi_{(1;I_1^*)}\big(X_{\{k-1,k\}}^{2^{k-1}-1}\big)X_{\{u,k\}}^{2^{d}-2^{k-1}}\simeq_{k} \phi_{(1;I_1^*)}\big(X_{\{k-1,k\}}^{2^{d}-1}\big).$$ 
This equality implies
\begin{align*}&\phi_{(1;I_1^*)}(X^{2^{k-1}-1})X_u^{2^{d}-2^{k-1}}x_u^{2^{d}}
\simeq_{(0,\omega)} \Big(\phi_{(1;I_1^*)}\big(X_{\{k-1,k\}}^{2^{d}-1}\big)+\sum_{i=0}^{k-1}Sq^{2^i}(h_i)\Big)x_u^{2^{d}}x_k^{2^d-1},
\end{align*}
where $h_i$ are suitable polynomials in $P_{k-1}$. Using the Cartan formula and Lemma \ref{bdbs3}, we obtain
$Sq^{2^i}(h_i)x_u^{2^{d}}x_k^{2^d-1}  \in R_{(1,k,d)}+P_k^-(\omega)$.

 Combining the above equalities gives
\begin{align*}&\phi_{(1;I_1)}(X^{2^{d}-1})x_k^{2^{d}} \simeq_{(0,\omega)} \phi_{(1;I_1^*)}(X^{2^{k-1}-1})X_k^{2^{d}-2^{k-1}}x_k^{2^{d}}\ \text{ mod}(R_{(1,k,d)})\\
&\simeq_{(0,\omega)} \Big(\sum_{u=1}^{k-1}\phi_{(1;I_1^*)}\big(X_{\{k-1,k\}}^{2^{d}-1}\big)x_u^{2^{d}}\Big)x_k^{2^d-1}+ Sq^{2^{k-1}}(y_{(J^*,1,k,d)}) \ \text{ mod}(R_{(1,k,d)}) \\
&= Y_{(1,k-1,d)}x_k^{2^d-1} + Sq^{2^{k-1}}(y_{(J^*,1,k,d)}) \ \text{ mod}(R_{(1,k,d)}).
\end{align*}
The lemma is proved.
\end{proof}

 Set $\mathcal M_{(t,k,i)} = \{J= (j_1,j_2, \ldots, j_{i}) : t\leqslant j_1 < j_2 < \ldots <j_{i}\leqslant k\}, 0\leqslant i \leqslant k-t+1$.

\begin{lems}\label{bdbss0} For $d \geqslant k-t+1$, 
$$Y_{(t,k,d)}\simeq_{(0,\omega_{(t,k,d)})} \sum_{0\leqslant i\leqslant k-t}\sum_{J \in \mathcal M_{(t,k,i)}} Sq^{2^i}(y_{(J,t,k,d)}) \ \text{\rm mod}(R_{(t,k,d)}).$$
\end{lems}
\begin{proof} Note that $Y_{(t,k,d)} = Y_1(x_t,\ldots , x_k) \in P_{(t,k)}$. So, we need only to prove the lemma for $Y_1 = Y_{(1,k,d)}$ and $\omega = \omega_{(1,k,d)} = \omega(X_1^{2^d-1}x_1^{2^d})$. The proof proceeds by induction on $k$. 

It is easiy to see that $Y_{(1,1,d)} = x_1^{2^d} = Sq^1(x_1^{2^d-1})$, $y_{(\emptyset,1,1,d)} = x_1^{2^d-1}$, $\mathcal M_{(1,1,0)} = \{\emptyset\}$ and $R_{(1,1,d)} = 0$. Hence, the lemma is true for $k=1$. 

Suppose that $k > 1$ and the lemma is true for $k-1$. We have 
$$Y_{(1,k,d)} = X_kY_{(1,k-1,d-1)}^2x_k^{2^d-2} + \phi_{(1;I_1)}(X^{2^d-1})x_k^{2^d}.$$
It is easy to see that if $h \in P_{k-1}^-(\omega_{(1,k-1,d-1)})$, then $X_kx_k^{2^d-2}h^2 \in P_{k}^-(\omega)$. So, using the inductive hypothesis, we get
\begin{align*}
X_kY_{(1,k-1,d-1)}^2&x_k^{2^d-2} \simeq_{(0,\omega)} \sum_{j=0}^{k-2} X_k(Sq^{2^j}(h_j))^2x_k^{2^d-2}\\ 
&+\sum_{0\leqslant i\leqslant k-2}\sum_{J \in \mathcal M_{(1,k-1,i)}} X_k\big(Sq^{2^{i}}(y_{(J,1,k-1,d-1)})\big)^2x_k^{2^d-2},
\end{align*}
where $h_j$ are suitable polynomials in $P_{k-1}$. 

By a direct computation using Proposition \ref{mdcb1}, Lemma \ref{bdbs3} and the Cartan formula, we obtain $X_k(Sq^{2^j}(h_j))^2x_k^{2^d-2}\in R_{(1,k,d)} + P_k^-(\omega),$ and
\begin{align*}
 X_k\big(Sq^{2^{i}}&(y_{(J,1,k-1,d-1)})\big)^2x_k^{2^d-2}\simeq_{(0,\omega)}\\ &Sq^{2^{i+1}}(X_ky_{(J,1,k-1,d-1)}^2x_k^{2^d-2}) + (X_ky_{(J,1,k-1,d-1)}^2x_k^{2^d-2})x_k^{2^{i+1}}.
\end{align*}
A simple computation shows that $X_ky_{(J,1,k-1,d-1)}^2x_k^{2^d-2} = y_{(J\cup k,1,k,d)}$. 

Observe that if $g \in P_{k-1}^-(\omega_{(1,k-1,d)})$, then $gx_k^{2^d-1} \in P_{k}^-(\omega)$. Hence, using Lemma \ref{bdbss} and the inductive hypothesis, we have
\begin{align*}
\phi_{(1;I_1)}(X^{2^d-1})x_k^{2^d} &\simeq_{(0,\omega)} \sum_{j=0}^{k-2} Sq^{2^j}(g_j)x_k^{2^d-1} + Sq^{2^{k-1}}(y_{(J^*,1,k,d)})\\ 
&\qquad+\sum_{0\leqslant i\leqslant k-2}\sum_{J \in \mathcal M_{(1,k-1,i)}} Sq^{2^{i}}(y_{(J,1,k-1,d)})x_k^{2^d-1},
\end{align*}
where $g_j$ are suitable polynomials in $P_{k-1}$. Using Lemma \ref{bdbs3} and the Cartan formula, we get $Sq^{2^j}(g_j)x_k^{2^d-1} \in R_{(1,k,d)}+P_k^-(\omega)$ and
$$Sq^{2^{i}}(y_{(J,1,k-1,d)})x_k^{2^d-1} \simeq_{(0,\omega)} Sq^{2^{i}}(y_{(J,1,k,d)}) + y_{(J,1,k,d)}x_k^{2^i}.$$
If $J =\emptyset$, then $i=0$ and $y_{(J,1,k,d)}x_k^{2^i} = y_{(J\cup k,1,k,d)}x_k^{2^{i+1}} = X_k^{2^d-1}x_k^{2^d}$. If $J \ne \emptyset$, then
\begin{align*}y_{(J,1,k,d)}x_k^{2^i}&=\phi_{(j_1;J\setminus j_1)}(X^{2^{i}-1})X_k^{2^d-2^{i}}x_k^{2^d}\\
y_{(J\cup k,1,k,d)}x_k^{2^{i+1}}&=\phi_{(j_1;J\cup k\setminus j_1)}(X^{2^{i+1}-1})X_k^{2^d-2^{i+1}}x_k^{2^{d}}.\end{align*}
Here $ j_1 = \min J$. 
Since $0\leqslant \ell(J\setminus j_1) < \ell(J\cup k\setminus j_1) = i \leqslant k-2$, using Lemma \ref{bdbss2}, we obtain
\begin{align*}y_{(J,1,k,d)}x_k^{2^i}&=\phi_{(j_1;J\setminus j_1)}(X^{2^{i}-1})X_k^{2^d-2^{i}}\simeq_{k} \phi_{(j_1;J\cup k\setminus j_1)}(X^{2^{d}-1}),\\
y_{(J\cup k,1,k,d)}x_k^{2^{i+1}}&=\phi_{(j_1;J\cup k\setminus j_1)}(X^{2^{i+1}-1})X_k^{2^d-2^{i+1}}\simeq_{k} \phi_{(j_1;J\cup k\setminus j_1)}(X^{2^{d}-1}).\end{align*}
This implies
$$y_{(J,1,k,d)}x_k^{2^i} + y_{(J\cup k,1,k,d)}x_k^{2^{i+1}} \simeq_{(0,\omega)} \sum_{j=0}^{k-1}Sq^{2^j}(z_j)x_k^{2^d},$$
where $z_j$ are suitable polynomials in $P_k$. 
By using the Cartan formula, we have $Sq^{2^j}(z_j)x_k^{2^d}= Sq^{2^j}(z_jx_k^{2^d})\in R_{(1,k,d)}$. It is easy to see that  $\mathcal M_{(1,k,0)}= \mathcal M_{(1,k-1,0)} = \{\emptyset\}$, $\mathcal M_{(1,k-1,k-1)}=\{J^*\}$ and
\begin{align*}&\mathcal M_{(1,k,i)} = \{J\cup k : J \in \mathcal M_{(1,k-1,i-1)}\}\cup \mathcal M_{(1,k-1,i)},\ 1\leqslant i \leqslant k-1.
\end{align*} 
Now, the lemma follows from the above equalities.
\end{proof}

\begin{proof}[Proof of Lemma \ref{hq4}] For $t=1$, the lemma follows from Lemma \ref{bdbss0}. 
For $t>1$, we have $Y_t = Z^{2^d-1}Y_{(t,k,d)}$ with $Z = x_1x_2\ldots x_{t-1}$.

Observe that if $h \in P_{(t,k)}^-(\omega_{(t,k,d)})$, then $Z^{2^d-1}h \in P_k^-(\omega).$

Let $Sq^{2^i}(y) \in R_{(t,k,d)}$ with $y$ a monomial in $P_{(t,k)}$, $0\leqslant i \leqslant k-t$, and $g:=Z^{2^d-1}Sq^{2^i}(y) $. According to the Cartan formula, 
\begin{align*}  g = Sq^{2^i}\big(Z^{2^d-1}y\big) + \sum _{1 \leqslant v \leqslant 2^i}Sq^v(Z^{2^d-1})Sq^{2^i - v}(y).
\end{align*} 
For $1\leqslant v < 2^i$, by using the Cartan formula, we can easily show that if a monomial $z$ appears as a term of the polynomial $Sq^v(Z^{2^d-1})Sq^{2^i - v}(y)$, then $\omega_u(z) < k-1$ for some $u\leqslant d$. So, $Sq^v(Z^{2^d-1})Sq^{2^i - v}(y) \in P_k^-(\omega)$. Hence, using the Cartan formula and Lemma \ref{bdbs3}, we get
$$g \simeq_{(i+1,\omega)} Sq^{2^i}(Z^{2^d-1})y  \simeq_{(0,\omega)} \sum_{1\leqslant j <t} Z^{2^d-1}x_j^{2^i}y \in P_k^-(\omega).$$

Let $f := Z^{2^d-1}Sq^{2^i}(y_{(J,t,k,d)})$ with $J \in \mathcal M_{(t,k,i)}$, $0\leqslant i \leqslant k-t$. By an argument analogous to the previous one, we have
\begin{align*}f \simeq_{(i+1,\omega)} \sum _{1\leqslant j < t}Z^{2^d-1}x_j^{2^i}y_{(J,t,k,d)}= \sum _{1\leqslant j < t}\phi_{(j;J)}(X^{2^d-1})x_j^{2^d}.
\end{align*}
Since $i+1 \leqslant k-t+1$, using Lemma \ref{bdbss0} and the above equalities, we obtain
$$Y_t = Z^{2^d-1}Y_{(t,k,d)} \simeq_{(k-t+1,\omega)} \sum_{1\leqslant j < t}\sum_{J \in \mathcal M_{(t,k)}} \phi_{(j;J)}(X^{2^d-1})x_j^{2^d}.$$ 
Here  $\mathcal M_{(t,k)} = \bigcup_{i=0}^{k-t} \mathcal M_{(t,k,i)}$. Obviously $J\subset I_{t-1}$ and $J \ne I_{t-1}$ for all $J \in \mathcal M_{(t,k)}$. The lemma follows.
\end{proof}

\subsection{Proof of Lemma \ref{bdc1}}\label{sub34}\

\medskip
First, we prepare some lemmas. 
\begin{lems}\label{bddl} Let $d$ be a positive integer, $\omega = \omega(X^{2^d-1})$ and $(j;J), (i;I) \in \mathcal N_k$ with $\ell(I) < d$. Then
$$ p_{(j;J)}\phi_{(i;I)}(X^{2^d-1}) \simeq_{(0,\omega)} \begin{cases}X^{2^d-1}, &\text{ if }\ (i;I)\subset (j;J),\\ 0, &\text{ if }\ (i;I)\not\subset (j;J). \end{cases}$$
\end{lems}
\begin{proof} Suppose that $(i;I)\not\subset (j;J).$ If $i \notin (j;J)$, then  from (\ref{rela22}), we see that $p_{(j;J)}(\phi_{(i;I)}(X^{2^d-1}))$ is a sum of monomials of the form
$$w = x_{i'}^{2^r-1}f_{k-1;i'}(z),$$
for suitable monomial $z$ in $P_{k-2}$. Here $i' =i$ if $j>i$ and $i' = i-1$ if $j<i$. In this case, we have $\alpha_r(2^r-1) = 0$ and $\omega_{r+1}(w) <k-1$. Hence, $w \in P_{k-1}^-(\omega)$.  Suppose that $i \in (j;J)$. Since $(i;I)\not\subset (j;J)$, there is $1 \leqslant t \leqslant r$, such that $i_t \notin (j;J)$, then  from (\ref{rela22}), we see that $p_{(j;J)}(\phi_{(i;I)}(X^{2^d-1}))$ is a sum of monomials of the form
$$w = x_{i_t-1}^{2^r-2^{r-t}-1}f_{k-1;i_t-1}(z),$$
for some monomial $z$ in $P_{k-2}$. It is easy to see that $\alpha_{r-t}(2^r-2^{r-t}-1) = 0$ and $\omega_{r-t+1}(w) < k-1$. Hence, $w \in P_{k-1}^-(\omega)$.  

Suppose that $(i;I)\subset (j;J).$ If $i = j$, then  from (\ref{rela22}), we see that the polynomial $p_{(j;J)}(\phi_{(i;I)}(X^{2^d-1}))$ is a sum of monomials of the form 
$$w = \Big(\prod_{1 \leqslant t \leqslant r}x_{i_t-1}^{2^r-2^{r-t}-1+ b_t}\Big)\Big(\prod _{j+1\in J\setminus I}x_j^{2^d-1+c_j}\Big)\Big(\prod _{j+1\notin J}x_j^{2^d-1}\Big),$$
where $b_1+ b_2 +\ldots + b_r + \sum_{j+1\in J\setminus I}c_j = 2^r-1.$ If $c_j >0$, then $\alpha_{u_j}(2^d-1+c_j) = 0$ with $u_j$ the smallest index such that $\alpha_{u_j}(c_j) = 1$. Hence, $w \in P_{k-1}^-(\omega)$.   
If $b_t=0$ for suitable $1 \leqslant t \leqslant r$, then $\alpha_{r-t}(2^r-2^{r-t}-1) = 0$ and $\omega_{r-t+1}(w) < k-1$. Hence, $w \in P_{k-1}^-(\omega)$.   Suppose that $b_t > 0$ for any $t$. Let $v_t$ be the smallest index such that $\alpha_{v_t}(b_t) = 1$. If $v_t \ne r-t$, then $\alpha_{v_t}(2^r-2^{r-t}-1+ b_t) = 0$ and  $w \in P_{k-1}^-(\omega)$.   So, $u_t = r-t$ and $b_t = 2^{r-t} + b_t'$ with $b_t' \geqslant 0$. If $b'_t >0$, then $\alpha_{v_t'}(2^r-2^{r-t}-1+ b_t) = \alpha_{v_t'}(2^r-1+ b_t') = 0$  with $v_t'$ the smallest index such that $\alpha_{v_t'}(b_t') = 1$. Hence, $w \in P_{k-1}^-(\omega)$.
If $b_t' = 0$ for $1 \leqslant t \leqslant r$, then $w = X^{2^d-1}$. 

 If $i \in J$, then  from (\ref{rela22}), we see that the polynomial $p_{(j;J)}(\phi_{(i;I)}(X^{2^d-1}))$ is a sum of monomials of the form 
$$w = x_{i-1}^{2^r-1+b_0}\Big(\prod_{1 \leqslant t \leqslant r}x_{i_t-1}^{2^r-2^{r-t}-1+ b_t}\Big)\Big(\prod _{j+1\in J\setminus (i;I)}x_j^{2^d-1+c_j}\Big)\Big(\prod _{j+1\notin J}x_j^{2^d-1}\Big),$$
where $b_0 + b_1+ b_2 +\ldots + b_r + \sum_{j+1\in J\setminus (i;I)}c_j = 2^d-1.$ By a same argument as above, we see that $w \in P_{k-1}^-(\omega)$ if  either $c_j >0$ or $b_t \ne 2^{r-t}$ for some $j, t$ with $t >0$. Suppose $c_j = 0$ and $b_t = 2^{r-t}$ with all $j$ and $t>0$. Then, $2^d-1 = b_0 + b_1+ b_2 +\ldots + b_r + \sum_{j+1\in J\setminus (i;I)}c_j = b_0 + 2^r-1$ and $w = X^{2^d-1}$.
The lemma is proved.
\end{proof}

The following is easily proved by a direct computation.

\begin{lems}\label{bdtn} The following diagram is commutative: 
\begin{displaymath}
\xymatrix{P_{k-1} \ar[rr]^{f_{k,i}} \ar[d]^{p_{(i;I_{i}^{*})}}&& P_k \ar[d]^{p_{(i+1;I_{i+1})}}\\
P_{k-2} \ar[rr]^{f_{k-1,i}} && P_{k-1}.}
\end{displaymath}
Here $I_{i}^{*} = (i+1,\ldots, k-1)$ for $1\leqslant i < k$.
\end{lems}
\begin{proof}[Proof of Lemma \ref{bdc1}] \

i) Suppose that either $d \geqslant k$ or $d = k -1$ and $I\ne I_1$, then $\phi_{(i;I)}(z) = \phi_{(i;I)}(X^{2^d-1})f_i(\bar z)^{2^d}$. Hence,  the first part of the lemma follows from  Lemmas \ref{bddl} and \ref{bdcb3}. 

ii) Let $z \in \mathcal C$ and $d=k-1$. According to the relation (\ref{ctth}), $\phi_{(1;I_1)}(z) = \phi_{(2;I_2)}(X^{2^d-1})f_{k,1}(\bar z)^{2^d}$.  Hence,  from  Lemmas \ref{bddl},  \ref{bdtn} and \ref{bdcb3}, we have
\begin{align*}p_{(i;I)}(\phi_{(1;I_1)}(z)) 
&\equiv p_{(i;I)}(\phi_{(2;I_2)}(X^{2^d-1}))p_{(i;I)}(f_{k,1}(\bar z)^{2^d})\\
&\equiv \begin{cases}z &\text{if }  (i;I) = (1;I_1),\\
X^{2^d-1}f_{k-1,1}(p_{(1;I_{1}^*)}(\bar z^{2^d}))\in \langle \mathcal D \cup \mathcal E \rangle, &\text{if }  (i;I) = (2;I_2),\\
0,  &\text{otherwise}.
\end{cases}
\end{align*}

iii) Let $z \in \mathcal D$ and $d=k-1$. Since $\nu_1(z) = 2^{k-1}-1$, we have $\nu_1(\bar z) =0.$ Hence, using the relation (\ref{ctth}), Lemmas \ref{bddl},  \ref{bdtn} and \ref{bdcb3}, we get
\begin{align*}p_{(i;I)}(\phi_{(1;I_1)}(z)) 
&\equiv p_{(i;I)}(\phi_{(3;I_3)}(X^{2^d-1}))p_{(i;I)}(f_{k,2}(\bar z)^{2^d})\\
&\equiv \begin{cases}z &\text{if }  I_2\subset I,\\
X^{2^d-1}f_{k-1,2}(p_{(2;I_{2}^*)}(\bar z)^{2^d})\in \langle \mathcal E \rangle, &\text{if }  (i;I) = (3;I_3),\\
0,  &\text{otherwise}.
\end{cases}
\end{align*}

iv)  Let $z \in \mathcal E$ and $d=k-1$. Using the relation (\ref{ctth}),  
Lemmas \ref{bddl} and \ref{bdcb3}, we have $p_{(i;I)}(\phi_{(1;I_1)}(z))  \equiv 0$, if $(4;I_4) \not\subset (i;I)$. Suppose $(4;I_4) \subset (i;I)$. If $I_3= (4;I_4) \not\subset I$, then $(i;I) =(4;I_4)$. Using Lemmas \ref{bddl}, \ref{bdtn} and  \ref{bdcb3}, one gets 
\begin{align*}p_{(4;I_4)}(\phi_{(1;I_1)}(z)) &= p_{(4;I_4)}(\phi_{(4;I_4)}(X^{2^d-1}))p_{(4;I_4)}(f_{k,3}(\bar z)^{2^d})\\
 &\equiv X^{2^d-1}f_{k-1,3}(p_{(3;I_{3}^*)}(\bar z^{2^d})).
\end{align*}

Suppose that the monomial $y$ is a term of $f_{k-1,3}(p_{(3;I_{3}^*)}(\bar z^{2^d}))$. Since $\nu_1(\bar z) = \nu_2(\bar z) = 0$, we have $\omega_1(y) < k-3$. According to Theorem \ref{dlsig}, $X^{2^d-1}y \equiv 0$. Hence, $X^{2^d-1}f_{k-1,3}(p_{(3;I_{3}^*)}(\bar z^{2^d})) \equiv 0$. Now, using Lemmas \ref{bddl}, \ref{bdtn} and \ref{bdcb3}, we get
\begin{align*}p_{(i;I)}(\phi_{(1;I_1)}(z)) 
&\equiv p_{(i;I)}(\phi_{(4;I_4)}(X^{2^d-1}))p_{(i;I)}(f_3(\bar z)^{2^d})
\equiv \begin{cases}z &\text{if }  I_3\subset I,\\
0,  &\text{otherwise }.
\end{cases}
\end{align*}
The lemma is completely proved.
\end{proof}

\section{The cases $k \leqslant 3$}\label{s3a}

In this section and the next sections, we denote by $B_{k}(n)$ the set of all admissible monomials of degree $n$  in $P_k$. For a subset $B \subset P_k$, we denote $B^0 = B\cap P_k^0$, $B^+ = B\cap P_k^+$. For an $\omega$-vector $\omega =(\omega_1,\omega_2,\ldots,\omega_m)$ of degree $n$, we set $B_k(\omega) = B_{k}(n)\cap P_k(\omega)$. 

 If  there is $i_0=0, i_1, i_2, \ldots , i_r > 0$ such that $i_1 + i_2 + \ldots + i_r  = m$ and $\omega_{i_1+\ldots +i_{s-1} + t} = a_s, 1 \leqslant t \leqslant i_s, 1 \leqslant s \leqslant  r$, then we denote 
$\omega = (a_1^{(i_1)},a _2^{(i_2)},\ldots , a_r^{(i_r)})$. If $i_u = 1$, then we denote $a_u^{(1)} = a_u$.  

Using Lemma \ref{b53}(i)  in Subsection \ref{s6} and Theorem \ref{dlcb1}, we easily obtain the following.
\begin{prop}\label{md41} For any $s \geqslant 1$, 
$$B_k(1^{(s)}) = \big\{x_{i_1}x_{i_2}^2\ldots x_{i_{m-1}}^{2^{m-2}}x_{i_m}^{2^s-2^{m-1}};\ 1 \leqslant i_1< \ldots <i_m \leqslant k, 1\leqslant m \leqslant \min\{s,k\}\big\}.
$$
\end{prop}

It is well known that if $n \ne 2^u-1$ then $B_1(n) = \emptyset$. If $n = 2^u-1$ for $u \geqslant 0$, then $B_1(n) = B_1(1^{(u)}) = \{x^{2^u-1}\}.$ 
 It is easy to see that $\Phi(B_1(0)) = \{1\} = B_2(0)$, \  $\Phi(B_1(1)) = \{x_1, x_2\}= B_2(1)$. According to a result in Peterson \cite{pe}, for $u > 1$, we have
$$B_2(2^u-1) = \Phi(B_1(2^u-1)) = 
\{x_1^{2^u-1}, x_2^{2^u-1}, x_1x_2^{2^u-2}\}, 
$$

By Theorem \ref{dlmd1}, $B_2(n) = \emptyset$ if $n \ne 2^{t+u} + 2^t - 2$ for all nonnegative integers $t, u$. We define the $\mathbb F_2$-linear map $\psi: (P_k)_m  \to (P_k)_{2m+k}$ by $\psi(y) = X_\emptyset y^2$ for any monomial $y \in (P_k)_m$. From Theorem \ref{dlmd2}, we have 

\begin{thm}[Peterson \cite{pe}] If $n = 2^{t+u}+2^t-2$, with $t, u$ positive integers, then 
\begin{align*}B_2(n) &= \psi^t(\Phi(B_1(2^u-1)))\\ 
&= \begin{cases} \{(x_1x_2)^{2^t-1}\}, & u= 0,\\ 
\{x_1^{2^{t+1}-1}x_2^{2^t-1},\ x_1^{2^{t}-1}x_2^{2^{t+1}-1}\},\ &u =1,\\ 
\{x_1^{2^{t+u}-1}x_2^{2^t-1}, x_1^{2^{t}-1}x_2^{2^{t + u}-1}, x_1^{2^{t+1}-1}x_2^{2^{t+u} - 2^t -1}\}, & u > 1.
\end{cases}\end{align*}
 \end{thm} 

By Theorems \ref{dlmd1} and \ref{dlmd2}, for $k = 3$, we need only to consider the cases of degree $n = 2^s -2, \ n = 2^s -1$ and $n  = 2^{s+t} + 2^s - 2$ with $s, t$ positive integers. By a direct computation using Theorem \ref{dl1} we have 

\begin{thm}[Kameko \cite{ka}] \label{mdkmk}\

 {\rm i)} If $n = 2^{s} - 2$, then
$B_3(2^s-2) = \Phi(B_2(2^s-2)).$

{\rm ii)} If $n = 2^{s} - 1$, then
$B_3(2^s-1) = B_3(1^{(s)})\cup\psi(\Phi(B_2(2^{s-1}-2))).$

{\rm iii)} If $n = 2^{s+t} + 2^s - 2$, then
$$B_3(n) = \begin{cases}\Phi(B_2(8))\cup\{x_1^3x_2^4x_3\}, &\text{ if } s=1, t=2,\\ \Phi(B_2(2^{s+t} + 2^s -2)), &\text{otherwise.}
 \end{cases}$$
 \end{thm}

\section{Proof of Theorem \ref{dl3}} \label{s4}
\setcounter{equation}{0}
\renewcommand{\theequation}{\arabic{section}.\arabic{subsection}.\arabic{thms}.\arabic{equation}}

For $1 \leqslant i \leqslant k$, define $\varphi_i: QP_k \to QP_k$, the homomorphism induced by the $\mathcal{A}$-homomorphism  $\overline{\varphi}_i:P_k \to P_k$, which is determined by  $\overline{\varphi}_1(x_1) = x_1+x_2$,  $\overline{\varphi}_1(x_j) = x_j$ for $j > 1,$ and $\overline{\varphi}_i(x_i) = x_{i-1}, \overline{\varphi}_i(x_{i-1}) = x_i$, $\overline{\varphi}_i(x_j) = x_j$ for $j \ne i, i-1,\ 1 < i \leqslant k$.  Note that the general linear group $GL_k$ is generated by $\overline{\varphi}_i,\ 0 < i \leqslant k$ and the symmetric group $\Sigma_k$ is generated by $\overline{\varphi}_i,\ 1 < i \leqslant k$.

Let $B$ be a finite subset of $P_k$ consisting of some monomials of degree $n$. To prove the set $[B]$ is linearly independent in $QP_k$, we order the set $B$ by the order as in Definition \ref{defn3} and denote the elements of $B$ by $d_i = d_{n,i}, 0< i \leqslant b =|B|$ in such a way that $d_{n,i} < d_{n,j} $ if and only if $i < j$. Suppose there is a linear relation 
$$\mathcal S = \sum_{1 \leqslant j \leqslant b}\gamma_jd_{n,j} \equiv 0,$$
with $\gamma_j \in \mathbb F_2$. For $(i;I) \in \mathcal N_k$, we explicitly compute $p_{(i;I)}(\mathcal S)$ in terms of a minimal set of $\mathcal A$-generators in $P_{k-1}$.  Computing from some relations $p_{(i;I)}(\mathcal S) \equiv 0$ with $(i;I) \in \mathcal N_k$ and $\overline{\varphi}_i(\mathcal S) \equiv 0$, we will obtain $\gamma_j =0$ for all $j$.

\medskip
\subsection{The case of degree $n = 2^{s+1}-3$}\label{sub1}\
\setcounter{equation}{0}

\medskip
In this subsection we prove the following.

\begin{props}\label{dlc3} For any $s \geqslant 1$,  $\Phi(B_3(2^{s+1}-3))$ is a minimal set of generators for $\mathcal A$-module $P_4$ in degree $2^{s+1}-3$. 
\end{props}

We can see that $\Phi(B_3(2^{s+1}-3))$ is the set of all the admissible monomials of degree $n = 2^{s+1}-3$, $|\Phi(B_3(1))| =4$, $|\Phi(B_3(5))| = 15$,  $|\Phi(B_3(13))| = 35$ and $|\Phi(B_3(2^{s+1}-3))| = 45$, for $s \geqslant 4$. We need the following lemmas for the proof of the proposition.

\begin{lems}\label{3.1} If $x$ is an admissible monomial of degree $2^{s+1}-3$ in $P_4$, then $\omega(x) = (3^{(s-1)},1)$.
\end{lems}

\begin{proof} It is easy to see that the lemma holds for $s=1$. Suppose $s\geqslant 2$. Obviously, $z=x_1^{2^s-1}x_2^{2^{s-1}-1}x_3^{2^{s-1}-1}$ is the minimal spike of degree $2^{s+1}-3$ in $P_4$ and $\omega(z) = (3^{(s-1)},1)$. Since  $2^{s+1}-3$ is odd, we get either $\omega_1(x)=1$ or $\omega_1(x)=3$. If $\omega_1(x)=1$, then $\omega(x) < \omega(z)$. By Theorem \ref{dlsig}, $x$ is hit. This contradicts the fact that $x$ is admissible. Hence,  we have $\omega_1(x) =3$. Using Proposition \ref{mdcb3} and Theorem \ref{dlsig}, we obtain $\omega_i(x) = 3, \ i=1,2,\ldots , s-1$. Since $\deg x = 2^{s+2}-3$, from this it implies $\omega_{s}(x)=1$ and $\omega_i(x) = 0$ for $i>s$. The lemma is proved.   
\end{proof}
 By a direct computation, we have the following.
\begin{lems}\label{3.2} The following monomials are strictly inadmissible: 
$$X_1x_1^2, X_iX_j^2, \ 1 \leqslant i < j \leqslant 4.$$
\end{lems}

\begin{proof}[Proof of Proposition \ref{dlc3}] We have $n= 2^{s+1}-3 = 2^s+2^{s-1} + 2^{s-1} - 3$. Hence,  the proposition follows from Theorem \ref{dl1} for $s \geqslant 4$. According to Theorem \ref{mdkmk},  
$$B_3(n) = \{v_{1}= X^{2^{s-1}-1}x_3^{2^{s-1}}, v_{2} = X^{2^{s-1}-1}x_2^{2^{s-1}}, v_{3} = X^{2^{s-1}-1}x_1^{2^{s-1}}\},$$
where $X = x_1x_2x_3.$

It is easy to see that $\Phi (B_3(1)) = \{x_1, x_2, x_3, x_4\}$. Hence,  the proposition holds for $s = 1$. For $s = 2$, using Lemma \ref{3.2}, we see that 
$$\Phi^+(B_3(5)) = \{x_1x_2x_3x_4^2,\ x_1x_2x_3^2x_4, \ x_1x_2^2x_3x_4\}$$
is a minimal set of generators for $(P_4^+)_5$. Using Theorem \ref{mdkmk}, it is easy to see that, for $s = 3$, $\Phi^+(B_3(13))$ is the set of  23 monomials:

\medskip
\centerline{\begin{tabular}{lllll}
$d_{1} =  x_1x_2^{2}x_3^{3}x_4^{7}$,& $d_{2} =  x_1x_2^{2}x_3^{7}x_4^{3}$,& $d_{3} =  x_1x_2^{3}x_3^{2}x_4^{7}$,& $d_{4} =  x_1x_2^{3}x_3^{3}x_4^{6}$,\cr 
$d_{5} =  x_1x_2^{3}x_3^{6}x_4^{3}$,& $d_{6} =  x_1x_2^{3}x_3^{7}x_4^{2}$,& $d_{7} =  x_1x_2^{6}x_3^{3}x_4^{3}$,& $d_{8} =  x_1x_2^{7}x_3^{2}x_4^{3}$,\cr 
$d_{9} =  x_1x_2^{7}x_3^{3}x_4^{2}$,& $d_{10} =  x_1^{3}x_2x_3^{2}x_4^{7}$,& $d_{11} =  x_1^{3}x_2x_3^{3}x_4^{6}$,& $d_{12} =  x_1^{3}x_2x_3^{6}x_4^{3}$,\cr 
$d_{13} =  x_1^{3}x_2x_3^{7}x_4^{2}$,& $d_{14} =  x_1^{3}x_2^{3}x_3x_4^{6}$,& $d_{15} =  x_1^{3}x_2^{3}x_3^{3}x_4^{4}$,& $d_{16} =  x_1^{3}x_2^{3}x_3^{4}x_4^{3}$,\cr 
$d_{17} =  x_1^{3}x_2^{3}x_3^{5}x_4^{2}$,& $d_{18} =  x_1^{3}x_2^{5}x_3^{2}x_4^{3}$,& $d_{19} =  x_1^{3}x_2^{5}x_3^{3}x_4^{2}$,& $d_{20} =  x_1^{3}x_2^{7}x_3x_4^{2}$,\cr 
$d_{21} =  x_1^{7}x_2x_3^{2}x_4^{3}$,& $d_{22} =  x_1^{7}x_2x_3^{3}x_4^{2}$,& $d_{23} =  x_1^{7}x_2^{3}x_3x_4^{2}$.&
\end{tabular}}

\medskip
By a direct computation using Lemmas \ref{3.1}, \ref{3.2} and Theorem \ref{dlsig}, if $x$ is an admissible monomial of degree 13 in $P_4^+$, then $x \in \Phi^+(B_3(13))$. Hence, $(QP_4^+)_{13}$ is spanned by $[\Phi^+(B_3(13))]$.

 Now we prove that the set $[\Phi^+(B_3(13))]$ is linearly independent. 
Suppose there is a linear relation
\begin{equation} \mathcal S = \sum_{j=1}^{23} \gamma_{j}d_{j} \equiv 0,\label{ct411} 
\end{equation}
where $\gamma_j \in \mathbb F_2, 1\leqslant j \leqslant 23$. 

Consider the homomorphisms $p_{(1;i)}: P_4 \to P_3, i = 2,3,4$. By a direct computation from (\ref{ct411}), we have
\begin{align*}
p_{(1;2)}(\mathcal S ) &\equiv \gamma_{1}v_{1} +  \gamma_{2}v_{2} +   \gamma_{7}v_{3}\equiv 0,\\
p_{(1;3)}(\mathcal S ) &\equiv \gamma_{3}v_{1} + (\gamma_{5}+\gamma_{16})v_{2} +  \gamma_{8}v_{3}\equiv 0,\\
p_{(1;4)}(\mathcal S )&\equiv (\gamma_{4}+\gamma_{15})v_{1} + \gamma_{6}v_{2} +  \gamma_{9}v_{3}\equiv 0 . 
\end{align*}
From the above equalities it implies 
\begin{equation}\begin{cases} 
\gamma_j = 0, \ j = 1,2,3,6,7, 8, 9, \\
\gamma_{5}= \gamma_{16}, \ \gamma_{4}= \gamma_{16}. 
\end{cases} \label{ct412} 
\end{equation}

Substituting (\ref{ct412}) into the relation (\ref{ct411}), we have
\begin{equation} \mathcal S = \gamma_4d_{4} + \gamma_5d_{5} + \sum_{10\leqslant j\leqslant 23} \gamma_{j}d_{j} \equiv 0.\label{ct413}\end{equation}
Applying the homomorphisms $p_{(2;3)}, p_{(2;4)}, p_{(3;4)} : P_4 \to P_3$ to (\ref{ct413}), we get
\begin{align*} 
 p_{(2;3)}(\mathcal S) &\equiv \gamma_{10}v_{1} + (\gamma_{12}+\gamma_{16}+\gamma_{18})v_{2} +   \gamma_{21}v_{3} \equiv 0,\\
 p_{(2;4)}(\mathcal S) &\equiv (\gamma_{11}+\gamma_{15}+\gamma_{19})v_{1} +  \gamma_{13}v_{2} +  \gamma_{22}v_{3} \equiv 0,\\
 p_{(3;4)}(\mathcal S) &\equiv (\gamma_{14}+\gamma_{15}+\gamma_{16}+\gamma_{17})v_{1} +  \gamma_{20}v_{2} +  \gamma_{23}v_{3} \equiv 0.  
\end{align*}

Hence,  we get 
\begin{equation}\begin{cases} 
\gamma_j = 0, \ j = 10, 13, 20, 21, 22, 23, \\
\gamma_{12}+\gamma_{16}+\gamma_{18} = 
\gamma_{11}+\gamma_{15}+\gamma_{19}=0,\\
\gamma_{14}+\gamma_{15}+\gamma_{16}+\gamma_{17}=0.
\end{cases} \label{ct414} 
\end{equation}

Substituting (\ref{ct414}) into the relation (\ref{ct413}) we get
\begin{equation} \mathcal S = \gamma_4d_{4} + \gamma_5d_{5} + \gamma_{11}d_{11}  + \gamma_{12}d_{12} + \sum_{14\leqslant j\leqslant 19} \gamma_{j}d_{j} \equiv 0.\label{ct415}\end{equation}

The homomorphisms $p_{(1;(2,3))}, p_{(1;(2,4))}, p_{(1;(3,4))} : P_4 \to P_3$, send (\ref{ct415}) respectively to
\begin{align*} 
p_{(1;(2,3))}( \mathcal S)&\equiv (\gamma_{5}+\gamma_{12}+\gamma_{16})v_{2} + \gamma_{18}v_{3}\equiv 0,\\
p_{(1;(2,4))}(\mathcal S)&\equiv (\gamma_{4}+\gamma_{11}+\gamma_{15})v_{1} + \gamma_{19}v_{3} \equiv 0\\
p_{(1;(3,4))}( \mathcal S) &\equiv (\gamma_{4}+\gamma_{14}+\gamma_{15})v_{1} +  (\gamma_{5}+\gamma_{16}+\gamma_{17})v_{2} \equiv 0. 
 \end{align*} 
 From this we obtain 
\begin{equation}\begin{cases} 
\gamma_{18} = \gamma_{19} = \gamma_{5}+\gamma_{12}+\gamma_{16} = 0,\\
\gamma_{4}+\gamma_{11}+\gamma_{15} = \gamma_{4}+\gamma_{14}+\gamma_{15} = \gamma_{5}+\gamma_{16}+\gamma_{17} = 0.
\end{cases} \label{ct416} 
\end{equation}

Combining (\ref{ct412}), (\ref{ct414}) and (\ref{ct416}), we obtain $\gamma_j = 0, \ j = 1,2,\ldots , 23.$
 The proposition is proved.    
\end{proof}

\subsection{The case of degree $n = 2^{s+1}-2$}\label{s5}\
\setcounter{equation}{0}

\medskip
It is well-known that Kameko's homomorphism 
$$ \widetilde{Sq}^0_* :(QP_k)_{2m+k} \to (QP_k)_{m}$$  
is  an epimorphism. Hence,  we have
$$ (QP_k)_{2m+k}\cong (QP_k)_{m}\oplus (QP_k^0)_{2m+k}\oplus (\text{Ker}\widetilde{Sq}^0_*\cap (QP_k^+)_{2m+k}),$$
 and $(QP_k)_{m} \cong \langle[\psi (B_{k}(m))]\rangle \subset (QP_k)_{2m+k}$. 

For $k=4$, from Theorem \ref{mdkmk}, it is easy to see that 
$$\Phi(B_3(2)) = \Phi^0(B_3(2)) =\{x_ix_j \ : \ 1 \leqslant i < j \leqslant 4\}.$$
For $m = 2^s-3, \ s \geqslant 2$, we have
\begin{align*}&|\Phi^0(B_3(6))| = 18,\ |\Phi^0(B_3(2^{s+1}-2))| = 22,\  \text{for } s \geqslant 3, \\
& |\psi (B_{4}(1))| = 4,\ \text{Ker}\widetilde{Sq}^0_*\cap [B_4^+(6)] = \{[x_1x_2^2x_3x_4^2], [x_1x_2x_3^2x_4^2]\}.
\end{align*}
Hence, $\dim(QP_4)_{2}= 6$, $\dim(QP_4)_{6}= 24$.

\medskip 
The main result of this subsection is:

\begin{props}\label{dlc4}\
 For any $s \geqslant 3$, $(QP_4^+)_{2^{s+1}-2}\cap \text{\rm Ker}\widetilde{Sq}^0_*$  is  an $\mathbb F_2$-vector space of dimension $13$ with a basis consisting of all the classes represented by the following admissible monomials:

\medskip
 \centerline{\begin{tabular}{lll}
$d_{1} = x_1x_2x_3^{2^{s} - 2}x_4^{2^{s} - 2}$,&
$d_{2} = x_1x_2^{2}x_3^{2^{s} - 4}x_4^{2^{s} -1}$,&
$d_{3} = x_1x_2^{2}x_3^{2^{s} - 3}x_4^{2^{s} - 2}$,\cr
$d_{4} = x_1x_2^{2}x_3^{2^{s} -1}x_4^{2^{s} - 4}$,&
$d_{5} = x_1x_2^{3}x_3^{2^{s} - 4}x_4^{2^{s} - 2}$,&
$d_{6} = x_1x_2^{3}x_3^{2^{s} - 2}x_4^{2^{s} - 4}$,\cr
$d_{7} = x_1x_2^{2^{s} - 2}x_3x_4^{2^{s} - 2}$,&
$d_{8} = x_1x_2^{2^{s} -1}x_3^{2}x_4^{2^{s} - 4}$,&
$d_{9} = x_1^{3}x_2x_3^{2^{s} - 4}x_4^{2^{s} - 2}$,\cr
$d_{10} = x_1^{3}x_2x_3^{2^{s} - 2}x_4^{2^{s} - 4}$,&
$d_{12} = x_1^{3}x_2^{2^{s} - 3}x_3^{2}x_4^{2^{s} - 4}$,&
$d_{13} = x_1^{2^{s} -1}x_2x_3^{2}x_4^{2^{s} - 4}$.
\end{tabular}}

 {\begin{tabular}{ll}
\ \! $d_{11} = x_1^{3}x_2^{3}x_3^{4}x_4^{4}$, for $s = 3,$&
$d_{11} = x_1^{3}x_2^{5}x_3^{2^{s} - 6}x_4^{2^{s} - 4}$, for $s > 3,$\cr
\end{tabular}}
\end{props}

The proof of this proposition is based on some lemmas.

\begin{lems}\label{4.1} If $x$ is an admissible monomial of degree $2^{s+1}-2$ in $P_4$ and $[x] \in \text{\rm Ker}\widetilde{Sq}^0_*$, then $\omega(x) = (2^{(s)})$.
\end{lems}

\begin{proof}  We prove the lemma by induction on $s$. Obviously, the lemma holds for $s = 1$. Observe  that $z = (x_1x_2)^{2^s-1}$ is the minimal spike of degree $2^{s+1}-2$ in $P_4$ and $\omega(z) = (2^{(s)})$. Since $2^{s+1}-2$ is even, using Theorem \ref{dlsig} and the fact that $[x] \in \text{Ker}\widetilde{Sq}^0_*$, we obtain $\omega_1(x)=2$. Hence, $x =x_ix_jy^2$, where $y$ is a monomial of degree $2^s-2$ and $1\leqslant i < j \leqslant 4$. Since $x$ is admissible, by Theorem \ref{dlcb1}, $y$ is also admissible. Now, the lemma follows from the inductive hypothesis.
\end{proof}  

The following lemma is proved by a direct computation.

\begin{lems}\label{4.2} The following monomials are strictly inadmissible:

\rm{i)} $x_i^2x_jx_k^3, \ x_i^3x_j^4x_k^7,\ i < j, k\ne i, j$,
 $x_1^2x_2^2x_3x_4, x_1^2x_2x_3^2x_4, x_1^2x_2x_3x_4^2, x_1x_2^2x_3^2x_4.$

\rm{ii)} $x_1x_2^6x_3^3x_4^4, x_1^3x_2^4x_3x_4^6, x_1^3x_2^4x_3^3x_4^4.$

\rm{iii)} $x_1x_2^7x_3^{10}x_4^{12}, x_1^7x_2x_3^{10}x_4^{12}, x_1^3x_2^3x_3^{12}x_4^{12}, x_1^3x_2^5x_3^8x_4^{14}, x_1^3x_2^5x_3^{14}x_4^8, x_1^7x_2^7x_3^8x_4^8$.
\end{lems}

\begin{proof}[Proof of Proposition \ref{dlc4}] Let $x$ be an admissible monomial in $P_4$ and $[x] \in \text{\rm Ker}\widetilde{Sq}^0_*$. By Lemma \ref{4.1}, $\omega_i(x) = 2$, for $1 \leqslant i \leqslant s$. By induction on $s$, we see that if $x \ne d_i$, for $i = 1, 2, \ldots , 13,$ then there is a monomial $w$, which is given in Lemma \ref{4.2} such that $x = wy^{2^u}$ for some monomial $y$ and positive integer $u$. By Theorem \ref{dlcb1}, $x$ is inadmissible. Hence,  $\text{\rm Ker}\widetilde{Sq}^0_*\cap (QP_4^+)$ is spanned by the classes $[d_i]$ with $i = 1, 2, \ldots , 13$. 

Now, we prove that the classes $[d_i]$ with $i = 1, 2, \ldots , 13,$ are linearly independent. 
Suppose there is a linear relation
\begin{equation}\mathcal S = \sum_{1\leqslant i\leqslant 13}\gamma_id_{i} \equiv 0, \label{ct4.2.1}
\end{equation}
with $\gamma_i \in \mathbb F_2$. 
According to Theorem \ref{mdkmk}, for $s \geqslant 3$, $B_3(n)\cap (P_3^+)_{n}$ is the set consisting of 4 monomials:

\centerline{\begin{tabular}{ll}
$w_{1} = x_1x_2^{2^s-2}x_3^{2^s-1},$& $w_{2} = x_1x_2^{2^s-1}x_3^{2^s-2}$,\cr 
$w_{3} = x_1^3x_2^{2^s-3}x_3^{2^s-2},$&$w_{4} = x_1^{2^s-1}x_2x_3^{2^s-2}$.\cr  
\end{tabular}}
 
\medskip
Apply the homomorphisms $p_{(1;2)}, p_{(3;4)}: P_4 \to P_3$ to the relation (\ref{ct4.2.1}) to obtain
\begin{align*} 
p_{(1;2)}(\mathcal S) &\equiv \gamma_{2}w_{1} + \gamma_{4}w_{2} +   \gamma_{3}w_{3} +  \gamma_{7}w_{4} \equiv 0.\\
p_{(3;4)}(\mathcal S) &\equiv \gamma_{7}w_{1} + \gamma_{8}w_{2} +   \gamma_{12}w_{3} +  \gamma_{13}w_{4} \equiv 0.
\end{align*}

From these relations, we get $\gamma_i = 0, \ i = 2,3,4,7,8,12, 13$. Then, the relation (\ref{ct4.2.1}) becomes
\begin{equation}\mathcal S=\gamma_1d_{1} + \gamma_5d_{5} + \gamma_6d_{6} + \gamma_{9}d_{9} + \gamma_{10}d_{10} + \gamma_{11}d_{11}\equiv 0. \label{ct4.2.2}
\end{equation}
Apply the homomorphisms $p_{(1;4)}, p_{(2;3)}: P_4 \to P_3$ to the relation (\ref{ct4.2.2}) to get
\begin{align*} 
p_{(1;4)}(\mathcal S) &\equiv (\gamma_{1}+\gamma_{5}+\gamma_{10}+\gamma_{11})w_{1} + \gamma_{6}w_{3}  \equiv 0,\\
p_{(2;3)}(\mathcal S) &\equiv (\gamma_{1}+\gamma_{5}+\gamma_{10}+\gamma_{11})w_{2} + \gamma_{9}w_{3}   \equiv 0.
\end{align*} 
These equalities imply $\gamma_{6} = \gamma_{9} = \gamma_{1}+\gamma_{5}+\gamma_{10}+\gamma_{11} = 0.$ Hence,  we obtain
\begin{equation}\mathcal S=\gamma_1d_{1} + \gamma_5d_{5} + \gamma_{10}d_{10} + \gamma_{11}d_{11}\equiv 0. \label{ct4.2.3}
\end{equation}
For $s > 3$, applying the homomorphisms $p_{(1;3)}, p_{(2;4)} : P_4 \to P_3$ to (\ref{ct4.2.3}), we get 
\begin{align*} 
p_{(1;3)}(\mathcal S) &\equiv\gamma_{1}w_{2} + \gamma_{5}w_{3} \equiv 0,\\
p_{(2;4)}(\mathcal S) &\equiv\gamma_{1}w_{1} + \gamma_{10}w_{3} \equiv 0.
\end{align*} 
From the above equalities, we get $\gamma_i = 0, i = 1,2,\ldots , 13.$

For $s = 3$, applying the homomorphisms $p_{(1;3)}, p_{(2;4)} : P_4 \to P_3$ to (\ref{ct4.2.3}), we get 
\begin{align*} 
p_{(1;3)}(\mathcal S) &\equiv (\gamma_{1}+\gamma_{11})w_{2} + \gamma_{8}w_{3} \equiv 0,\\
p_{(2;4)}(\mathcal S) &\equiv (\gamma_{1}+\gamma_{11})w_{1} + \gamma_{10}w_{3} \equiv 0.
\end{align*} 
From the above equalities, we get $\gamma_i = 0, i = 2,\ldots , 10,12,13$ and $\gamma_{1} = \gamma_{11}$.  So, the relation (\ref{ct4.2.3}) becomes
$$\gamma_{1}(d_{1} + d_{11}) \equiv 0.$$
Now, we prove that $[d_{1} + d_{11}] \ne 0$. Suppose the contrary, that  the polynomial $d_{1} + d_{11}= x_1x_2x_3^6x_4^6 + x_1^3x_2^3x_3^4x_4^4$ is hit. Then, by the unstable property of the action of $\mathcal A$ on the polynomial algebra, we have
$$x_1x_2x_3^6x_4^6 + x_1^3x_2^3x_3^4x_4^4 = Sq^1(A) + Sq^2(B) + Sq^4(C),$$
for some polynomials $A \in (P_4^+)_{13}, B \in (P_4^+)_{12}, C \in (P_4^+)_{10}$. 
Let $(Sq^2)^3$ act on the both sides of the above equality. Since $(Sq^2)^3Sq^1 = 0$ and $(Sq^2)^3Sq^2 = 0$, we get
$$(Sq^2)^3(x_1x_2x_3^6x_4^6 + x_1^3x_2^3x_3^4x_4^4) = (Sq^2)^3Sq^4(C).$$
On the other hand, by a direct computation, it is not difficult to check that 
$$(Sq^2)^3(x_1x_2x_3^6x_4^6 + x_1^3x_2^3x_3^4x_4^4) \ne (Sq^2)^3Sq^4(C),$$
for all $C \in (P_4^+)_{10}$. This is a contradiction. Hence, $[d_{1} + d_{11}] \ne 0 $ and $\gamma_{1} = \gamma_{11} = 0$.
The proposition is proved.     
\end{proof}

\subsection{The case of degree $n = 2^{s+1}-1$}\label{s6}\
\setcounter{equation}{0}

\medskip
First, we determine the $\omega$-vector of an admissible monomial of degree $2^{s+1}-1$ in $P_4$. 

\begin{lems}\label{b51} If $x$ is an admissible monomial of degree $2^{s+1}-1$ in $P_4$,  then either $\omega(x)= (1^{(s+1)})$ or $\omega(x) = (3,2^{(s-1)})$ or  $\omega(x) = (1,3)$ for $s = 2$.
\end{lems}

\begin{proof} Obviously, the lemma holds for $s=1$. Suppose $s\geqslant 2$. By a direct computation we see that if $w$ is a monomial in $P_4$ such that $\omega(w) =(1,3,2)$ or $\omega(w) = (1,1,3)$, then $w$ is strictly inadmissible. 

Since $2^{s+1}-1$ is odd, we have either $\omega_1(x) =1$ or $\omega_1(x) =3$. 
If $\omega_1(x) =1$, then $x= x_iy^2$, where $y$ is a monomial of degree $2^s-1$. Hence,  either $\omega_1(y) = 1$ or $\omega_1(y) = 3$. So, the lemma holds for $s = 2$. 
Suppose that $s \geqslant 3$. If  $\omega_1(y) = 3$, then $y = X_iy_1^2$, where $y_1$ is a monomial of degree $2^{s-1}-2$. Since $y_1$ is admissible, using Proposition \ref{mdcb3}, one gets $\omega_1(y_1) = 2$. Hence, $x$ is inadmissible. If $\omega_1(y) = 1$, then $y = x_jy_1^2$, where $y_1$ is an admissible monomial of degree $2^{s -1}-1$. By the inductive hypothesis $\omega(y_1) = (1^{(s-1)})$. So, we get $\omega(x) = (1^{(s+1)})$.

Suppose that $\omega_1(x) = 3$. Then, $x = X_{i}y^2$, where  $y$ is an admissible monomial of degree $2^s-2$. Since $x$ and $y$ are admissible, by Lemma \ref{4.2} and Proposition \ref{mdcb3}, $\omega(y) = (2^{(s-1)})$.  The lemma is proved. 
\end{proof}

For $s = 1$, we have $(QP_4)_3 = (QP_4^0)_3$. Hence, $B_4(3) = \Phi^0(B_3(3))$. Using Proposition \ref{md41} and Theorem \ref{mdkmk}, we have
\begin{align*} |&\Phi^0(B_3(3))| = 14, \  |\Phi^0(B_3(7))|\ = 26, \ |\Phi^0(B_3(15))| =  38,\\ 
&|\Phi^0(B_3(2^{s+1}-1))| = 42, \text{ for } s \geqslant 4.
\end{align*}

For $s = 2$, $B_4(7) = B_4(1^{(3)})\cup B_4(1,3)\cup B_4(3,2)$. By a direct computation, we have
$B_4(1,3) = \{ x_1X_1^2\}, \ B_4(3,2) =\Phi(B_3(7))$.

Recall that
$$B_3(2^{s+1}-1)  = B_3(1^{(s+1)})\cup \psi(\Phi(B_2(2^s-2))),$$
where $B_2(2^s-2) = \{x_1^{2^{s-1}-1}x_2^{2^{s-1}-1}\}$. Hence, $B_3(3,2^{(s-1)}) = \psi(\Phi(B_2(2^s-2)))$.

\begin{props}\label{dlc5} 
For any $s \geqslant 3$, $B_4(3,2^{(s-1)}) = \Phi(B_3(3,2^{(s-1)}))\cup A(s)$, where $A(s)$ is determined as follows:
\begin{align*}
&A(3) = \{x_1^{3}x_2^{4}x_3x_4^{7},x_1^{3}x_2^{4}x_3^{7}x_4,
x_1^{3}x_2^{7}x_3^{4}x_4,x_1^{7}x_2^{3}x_3^{4}x_4,
x_1^{3}x_2^{4}x_3^{3}x_4^{5}\},\\
&A(4) = \{x_1^{3}x_2^{4}x_3^{11}x_4^{13},
x_1^{3}x_2^{7}x_3^{8}x_4^{13},
x_1^{7}x_2^{3}x_3^{8}x_4^{13},
x_1^{7}x_2^{7}x_3^{8}x_4^{9},
x_1^{7}x_2^{7}x_3^{9}x_4^{8},\},\\
&A(s) = \{x_1^3x_2^4x_3^{2^s-5}x_4^{2^s-3}\}, s \geqslant 5.
\end{align*}
\end{props}

Combining Lemma \ref{b51} and Propositions \ref{md41}, \ref{dlc5}, we have 
$$B_4(2^{s+1}-1) = B_4(1^{(s+1)})\cup \Phi(B_3(3,2^{(s-1)}))\cup A(s).$$

The following can easily be proved by a direct computation.

\begin{lems}\label{b53} The following monomials are strictly inadmissible:

\medskip
\rm{i)} $x_i^2x_j, x_i^3x_j^4,\ 1 \leqslant i < j \leqslant 4.$

\rm{ii)} $X_2x_1^2x_2^2, \ X_1x_1^2x_i^2, \ i=  2, 3, 4.$

\rm{iii)}  $x_i^3x_j^{12}x_kx_\ell^{15},\ x_i^3x_j^{4}x_k^9x_\ell^{15}, x_i^3x_j^{5}x_k^8x_\ell^{15}, \ i < j < k, \ \ell \ne i,\ j,\ k.$

\rm{iv)} $x_1^7x_2^{11}x_3^{12}x_4$,\ $ x_1^3x_2^{12}x_3^3x_4^{13}$,\ $X_jx_1^2x_2^4x_3^8x_4^8x_j^6,\ x_1^7x_2^{11}x_3^4x_4^8x_j,$\ 

\quad\ $ x_1^3x_2^{3}x_3^{12}x_4^{8}x_i^4x_j, \  x_1^3x_2^{3}x_3^{24}x_4^{29}x_i^4, \ i=1,2, \ j=3, 4.$
\end{lems}

\begin{proof}[Proof of Proposition \ref{dlc5}] By a direct computation using Lemmas \ref{b51},  \ref{b53} and Theorem \ref{dlcb1} we see that if $x$ is a monomial of degree $2^{s+1} - 1$ in $P_4$ and $x \notin  \Phi(B_3(3,2^{(s-1)}))\cup A(s)$, then there is a monomial $w$ which is given in Lemma \ref{b53} such that $ x = wy^{2^u}$ for some monomial $y$ and integer $u > 1$. Hence,  $x$ is inadmissible.

Now we prove that the set $[\Phi^+(B_3(3,2^{(s-1)}))\cup A(s)]$ is linearly independent in $QP_4^+$.  For $s = 3$, we have $|\Phi^+(B_3(3,2,2))\cup A(3)| = 36$. Suppose there is a linear relation
\begin{equation}\mathcal S = \sum_{1 \leqslant i \leqslant 36}\gamma_id_i\equiv 0,\label{ct431}\end{equation}
with $\gamma_i \in \mathbb F_2$ and $d_i = d_{15,i}$.

By a simple computation, we see that $B_3(3;2,2) = \psi(\Phi(B_2(6)))$ is the set consisting of 6 monomials:
$$v_1 = x_1x_2^7x_3^7, \ v_2 = x_1^3x_2^5x_3^7, v_3 = x_1^3x_2^7x_3^5, v_4 = x_1^7x_2x_3^7, v_5 = x_1^7x_2^3x_3^5, v_6 = x_1^7x_2^7x_3.$$
By a direct computation, we have
\begin{align*}
p_{(1;2)}(\mathcal S ) &\equiv  \gamma_{3}v_{2} +  \gamma_{4}v_{3} +   (\gamma_{9}+\gamma_{22})v_{4} +  (\gamma_{10}+\gamma_{23})v_{5} +  (\gamma_{11}+\gamma_{24})v_{6}\equiv 0,\\
p_{(1;3)}(\mathcal S ) &\equiv  (\gamma_{1}+\gamma_{16})v_{1} +  \gamma_{5}v_{2} +  (\gamma_{7}+\gamma_{20})v_{3} +  \gamma_{13}v_{5} +  (\gamma_{15}+\gamma_{30})v_{6}\equiv 0,\\
p_{(1;4)}(\mathcal S ) &\equiv   (\gamma_{2}+\gamma_{19})v_{1} +  (\gamma_{6}+\gamma_{21}+\gamma_{27})v_{2} +  \gamma_{8}v_{3} +  (\gamma_{12}+\gamma_{29})v_{4} +  \gamma_{14}v_{5}\equiv 0,\\
p_{(2;3)}(\mathcal S ) &\equiv  (\gamma_{1}+\gamma_{3}+\gamma_{5}+\gamma_{9})v_{1} + (\gamma_{16}+\gamma_{22})v_{2}\\
&\quad +  (\gamma_{18}+\gamma_{20}+\gamma_{23}+\gamma_{26})v_{3} +  \gamma_{32}v_{5} +  (\gamma_{34}+\gamma_{36})v_{6}\equiv 0,\\
p_{(2;4)}(\mathcal S ) &\equiv  (\gamma_{2}+\gamma_{4}+\gamma_{8}+\gamma_{11})v_{1}\\
&\quad + (\gamma_{17}+\gamma_{21})v_{2} +  (\gamma_{19}+\gamma_{24}v_{3} +  \gamma_{31}+\gamma_{35})v_{4} +  \gamma_{33}v_{5} \equiv 0,\\
p_{(3;4)}(\mathcal S ) &\equiv   (\gamma_{12}+\gamma_{13}+\gamma_{14}+\gamma_{15})v_{1} + (\gamma_{25}+\gamma_{26}+\gamma_{27}+\gamma_{28})v_{2}\\
&\quad + (\gamma_{29}+\gamma_{30})v_{3} +  (\gamma_{31}+\gamma_{32}+\gamma_{33}+\gamma_{34})v_{4} +  (\gamma_{35}+\gamma_{36})v_{5} \equiv 0. 
 \end{align*}
From these equalities, we obtain 
\begin{align}\label{ct62}\begin{cases}
\gamma_{j}= 0 , \ j = 3,4,5,8,13,14,32,33,\\ 
\gamma_{1} = \gamma_{9} = \gamma_{16} = \gamma_{22},\  
\gamma_{2} = \gamma_{11} = \gamma_{19} = \gamma_{24},\ \gamma_{7} = \gamma_{20},\\ 
\gamma_{1} = \gamma_{9} = \gamma_{16} = \gamma_{22},\ \ 
 \gamma_{10} = \gamma_{23}, \
\gamma_{17} = \gamma_{21},\\ 
\gamma_{12} = \gamma_{15} = \gamma_{29} = \gamma_{30},\ \ 
\gamma_{31} = \gamma_{34} = \gamma_{35} = \gamma_{36},\\
\gamma_{6}+\gamma_{21}+\gamma_{27} = \gamma_{7}+\gamma_{10}+\gamma_{18}+\gamma_{26}= \gamma_{25}+\gamma_{26}+\gamma_{27}+\gamma_{28} = 0. 
\end{cases}
 \end{align}

By a direct computation using (\ref{ct62}) and Theorem \ref{dlsig},  we get
 \begin{align*}
p_{(1;(2,3))}(\mathcal S ) &\equiv \gamma_{18}w_{3} + \gamma_{26}w_{5} +   \gamma_{28}w_{6}\equiv 0,\\
p_{(1;(2,4))}(\mathcal S ) &\equiv (\gamma_{6}+\gamma_{10}+\gamma_{27})w_{2} +  \gamma_{25}w_{4} +  \gamma_{27}w_{5}\equiv 0,\\
p_{(1;(3,4))}(\mathcal S ) &\equiv  (\gamma_{17}+\gamma_{18})w_{1}\\
&\quad + (\gamma_{6}+\gamma_{7}+\gamma_{17}+\gamma_{25}+\gamma_{26}+\gamma_{27})w_{2} +  (\gamma_{17}+\gamma_{28})w_{3}\equiv 0.
 \end{align*}
Combining the above equalities and (\ref{ct62}), one gets $\gamma_j = 0$ for $j \ne$ 1, 2 , 9, 11, 12, 15, 16, 19, 22, 24 , 29, 30, 31 and $\gamma_{1} = \gamma_{9} = \gamma_{16} = \gamma_{22}$, $\gamma_{2} = \gamma_{11} = \gamma_{19} = \gamma_{24}$, $\gamma_{12} = \gamma_{15} = \gamma_{29} = \gamma_{30}$, $\gamma_{31} = \gamma_{34} = \gamma_{35} = \gamma_{36}$. Hence, the relation (\ref{ct431}) becomes
\begin{equation}\gamma_1\theta_1 + \gamma_2\theta_2 + \gamma_{12}\theta_3 + \gamma_{31}\theta_{4}\equiv 0,\label{ct432}
\end{equation}
where
\begin{align*} 
&\theta_{1} =  d_{1}+ d_{9}+ d_{16}+ d_{22}, \ \ 
\theta_{2} =  d_{2}+ d_{11}+ d_{19}+ d_{24}, \\ 
&\theta_{3} =  d_{12}+ d_{15}+ d_{29}+ d_{30}, \ \ 
\theta_{4} =  d_{31}+ d_{34}+ d_{35}+ d_{36}.
\end{align*}

Now, we prove that $\gamma_{1} = \gamma_{2} = \gamma_{12} = \gamma_{31} = 0.$

The proof is divided into 4 steps.

{\it Step 1.}  Under the homomorphism $\varphi _1$, the image of (\ref{ct432}) is
\begin{equation} \gamma_{1}\theta_{1}+ \gamma_{2}\theta_{2}  +\gamma_{12}\theta_{3} + \gamma_{31}(\theta_{4} + \theta_{3})\equiv 0.\label{ct433}
\end{equation}
Combining (\ref{ct432}) and (\ref{ct433}), we get
\begin{equation} \gamma_{31}\theta_{3} \equiv 0.\ \label{ct434}
\end{equation}
If the polynomial $\theta_3$ is hit, then we have
$$\theta_3 = Sq^1(A) +Sq^2(B) + Sq^4(C),$$
for some polynomials $A \in (P_4^+)_{14}, B \in (P_4^+)_{13}, C \in (P_4^+)_{11}$. Let $(Sq^2)^3$ act on the both sides of this equality. We get
$$(Sq^2)^3(\theta_3) = (Sq^2)^3Sq^4(C),$$
By a direct calculation, we see that the monomial $x=x_1^8x_2^7x_3^4x_4^2$ is a term of $(Sq^2)^3(\theta_3)$. If this monomial is a term of $(Sq^2)^3Sq^4(y)$ for a monomial $y \in (P_4^+)_{11}$, then $y = x_2^7f_2(z)$ with $z \in P_3$ and $\deg z = 4$. Using the Cartan formula, we see that $x$ is a term of $x_2^7(Sq^2)^3Sq^4(z) = x_2^7(Sq^2)^3(z^2) = 0$. Hence, 
$$(Sq^2)^3(\theta_3) \ne (Sq^2)^3Sq^4(C),$$
for all $C \in (P_4^+)_{11}$ and we have a contradiction.
 So, $[\theta_3] \ne 0$ and  $\gamma_{31} = 0.$ 

{\it Step 2.} Since $\gamma_{31} = 0$, the homomorphism $\varphi_2$ sends (\ref{ct432}) to
\begin{equation} \gamma_{1}\theta_{1}+ \gamma_{2}\theta_{2}  + \gamma_{12}\theta_{4} \equiv 0.\label{ct435}
\end{equation}
Using the relation (\ref{ct435}) and by the same argument as given in Step 1, we get $\gamma_{12} = 0$.

{\it Step 3.} Since $\gamma_{31}  = \gamma_{12}=0$, the homomorphism $\varphi_3$ sends (\ref{ct432}) to
\begin{equation} 
\gamma_{1}[\theta_{1}]+ \gamma_{2}[\theta_{3}] = 0.\label{ct436}
\end{equation}
 Using the relation (\ref{ct436}) and by the same argument as given in Step 2, we obtain $\gamma_{3}=0$. 

{\it Step 4.} Since $\gamma_{31}  = \gamma_{12}= \gamma_{2} = 0$, the homomorphism $\varphi_4$ sends (\ref{ct432}) to
$$\gamma_{1}\theta_{2} =0.$$
 Using this relation and by the same argument as given in Step 3, we obtain   $\gamma_{1} = 0.$  

For $s \geqslant 4$,  $B_3(3,2^{(s-1)})) = \psi(\Phi(B_2(2^{s-1}-2)))$ is the set consisting of 7 monomials:
\begin{align*}&v_{1} = x_1x_2^{2^{s}-1}x_3^{2^{s}-1},\ v_{2} = x_1^{3}x_2^{2^{s}-3}x_3^{2^{s}-1},\ v_{3} = x_1^{3}x_2^{2^{s}-1}x_3^{2^{s}-3},\ v_{4} = x_1^{7}x_2^{2^{s}-5}x_3^{2^{s}-3},\\
& v_{5} = x_1^{2^{s}-1}x_2x_3^{2^{s}-1},\ v_{6} = x_1^{2^{s}-1}x_2^{3}x_3^{2^{s}-3},\ v_{7} = x_1^{2^{s}-1}x_2^{2^{s}-1}x_3.\end{align*}

Let $s=4$. Then, we have $|\Phi^+(B_3(3,2,2,2))\cup A(4)| = 46$. Suppose there is a linear relation
\begin{equation}\mathcal S = \sum_{1\leqslant j \leqslant 46}\gamma_jd_{j} = 0, \label{ct437}
\end{equation}
with $\gamma _j \in \mathbb F_2$ and $d_i = d_{31,i}$. 

By a direct computation using Theorem \ref{dlsig}, we have 
\begin{align*}
p_{(1;2)}(\mathcal S )  &\equiv \gamma_{3}w_{2} + \gamma_{4}w_{3} +   (\gamma_{9}+\gamma_{25})w_{4} +  \gamma_{12}w_{5} +  \gamma_{13}w_{6} +  \gamma_{14}w_{7}  \equiv 0,\\
p_{(1;3)}(\mathcal S )  &\equiv  (\gamma_{1}+\gamma_{19})w_{1} + \gamma_{5}w_{2} +  (\gamma_{7}+\gamma_{23}+\gamma_{37}+\gamma_{39})w_{3}\\
&\quad +  (\gamma_{10}+\gamma_{28})w_{4} +  \gamma_{16}w_{6} +  \gamma_{18}w_{7} \equiv 0,\\
p_{(1;4)}(\mathcal S )  &\equiv  (\gamma_{2}+\gamma_{22})w_{1} +  (\gamma_{6}+\gamma_{24}+\gamma_{27}+\gamma_{29}+\gamma_{32}+\gamma_{40})w_{2} \\
&\quad +  \gamma_{8}w_{3}+  \gamma_{11}w_{4} +  (\gamma_{15}+\gamma_{34})w_{5} +  \gamma_{17}w_{6}  \equiv 0.
\end{align*}
From these equalities, we get
\begin{align}\label{ct63}\begin{cases}
\gamma_j = 0, \ j = 3, 4, 5, 8, 11, 12, 13, 14, 16, 17, 18,\\ 
\gamma_{9}=\gamma_{25},\ \gamma_{1}=\gamma_{19},\  \gamma_{7}+\gamma_{23}+\gamma_{37}+\gamma_{39}=0,\  \gamma_{10}=\gamma_{28},\\ \gamma_{2}=\gamma_{22},\ \gamma_{6}+\gamma_{24}+\gamma_{27}+\gamma_{29}+\gamma_{32}+\gamma_{40} = 0,\ \gamma_{15}+\gamma_{34} = 0.       
\end{cases}
 \end{align}

Using the relations (\ref{ct63}) and Theorem \ref{dlsig}, we obtain
\begin{align*}
p_{(2;3)}(\mathcal S )  &\equiv \gamma_{1}w_{1} +  \gamma_{1}w_{2} +   (\gamma_{9}+\gamma_{10}+\gamma_{21}+\gamma_{23}+\gamma_{26}+\gamma_{31}+\gamma_{39})w_{3}\\
&\quad +  (\gamma_{35}+\gamma_{37})w_{4} +  \gamma_{43}w_{6}  \equiv 0,\\
p_{(2;4)}(\mathcal S )  &\equiv   \gamma_{2}w_{1} + \gamma_{45}w_{7}   + (\gamma_{20}+\gamma_{24}+\gamma_{38}+\gamma_{40})w_{2} +  \gamma_{2}w_{3} +  \gamma_{36}w_{4}\\
&\quad +  (\gamma_{42}+\gamma_{46})w_{5} +  \gamma_{44}w_{6}  \equiv 0,\\
p_{(3;4)}(\mathcal S )  &\equiv  \gamma_{15}w_{1} +  (\gamma_{30}+\gamma_{31}+\gamma_{32}+\gamma_{33})w_{2}\\
&\quad +  \gamma_{15}w_{3} +  \gamma_{41}w_{4} +  (\gamma_{42}+\gamma_{43}+\gamma_{44}+\gamma_{45})w_{5} +  \gamma_{42}w_{6}  \equiv 0.
\end{align*}
From these equalities, we get
\begin{align}\label{ct64}\begin{cases}
\gamma_j = 0, \ j = 1, 2, 15, 36, 41, 42, 43, 44, 45, 46,\\ 
\gamma_{10}+\gamma_{21}+\gamma_{23}+\gamma_{26}+\gamma_{31}+\gamma_{39} = 0,\\
\gamma_{35}=\gamma_{37},\  \gamma_{20}+\gamma_{24}+\gamma_{38}+\gamma_{40}= 0,\\ \gamma_{30}+\gamma_{31}+\gamma_{32}+\gamma_{33}  =0.      \end{cases}
 \end{align}
By a direct computation using (\ref{ct63}), (\ref{ct64}) and  Theorem \ref{dlsig}, we have 
\begin{align*}
p_{(1;(2,3))}(\mathcal S )  &\equiv  (\gamma_{7}+\gamma_{21}+\gamma_{23}+\gamma_{39})w_{3} + \gamma_{26}w_{4} +   \gamma_{31}w_{6} +  \gamma_{33}w_{7}  \equiv 0,\\
p_{(1;(2,4))}(\mathcal S )  &\equiv  (\gamma_{6}+\gamma_{9}+\gamma_{20}+\gamma_{24}+\gamma_{27}+\gamma_{29}+\gamma_{32}+\gamma_{38}+\gamma_{40})w_{2}\\
&\quad +  \gamma_{27}w_{4} +  \gamma_{30}w_{5} +  \gamma_{32}w_{6}  \equiv 0,\\
p_{(1;(3,4))}(\mathcal S )  &\equiv  (\gamma_{6}+\gamma_{10}+\gamma_{23}+\gamma_{24}+\gamma_{26}+\gamma_{27}+\gamma_{29}+\gamma_{30}+\gamma_{31}+\gamma_{32})w_{2}\\ &\quad  +  (\gamma_{7}+\gamma_{23}+\gamma_{24}+\gamma_{33}+\gamma_{35}+\gamma_{38}+\gamma_{39}+\gamma_{40})w_{3}\\
&\quad  + (\gamma_{20}+\gamma_{21}+\gamma_{35})w_{1} +  \gamma_{29}w_{4}   \equiv 0,\\
p_{(2;(3,4))}(\mathcal S )  &\equiv (\gamma_{10}+\gamma_{20}+\gamma_{23}+\gamma_{24}+\gamma_{29}+\gamma_{30}+\gamma_{35}+\gamma_{38}+\gamma_{39}+\gamma_{40})w_{2}\\
&\quad  +  (\gamma_{9}+\gamma_{10}+\gamma_{21}+\gamma_{23}+\gamma_{24}+\gamma_{26}+\gamma_{27}+\gamma_{29}+\gamma_{31}+\gamma_{32})w_{3}\\  &\quad + (\gamma_{6}+\gamma_{7}+\gamma_{9}+\gamma_{10})w_{1} +    \gamma_{38}w_{4}   \equiv 0.
\end{align*}

Combining the above equalities, (\ref{ct63}) and (\ref{ct64}), we get
\begin{align}\label{ct65}\begin{cases}
\gamma_j = 0, \ j \ne 7, 10, 21, 23, 24, 28, 35, 37, 39, 40,\\
\gamma_{7}=\gamma_{10}=\gamma_{28},\  \gamma_{21}=\gamma_{35}=\gamma_{37},\\
\gamma_{7}+\gamma_{21}+\gamma_{23}+\gamma_{39} = 0.
\end{cases}
 \end{align}
 Hence,  we obtain 
\begin{equation}
\gamma_{7}\theta_{1} + \gamma_{21}\theta_{2} + \gamma_{39}\theta_{3}  +\gamma_{24}\theta_{4} \equiv 0,\label{ct438}
\end{equation}
where
\begin{align*} 
&\theta_{1}= d_{7}+ d_{10}+ d_{23}+ d_{28},\\
&\theta_{2} = d_{21}+ d_{23}+ d_{35}+ d_{37},\\
&\theta_{3} = d_{23} + d_{39},\ \
\theta_{4} = d_{24} + d_{40}.
\end{align*}

Now, we prove $\gamma_{7} = \gamma_{21} = \gamma_{24} = \gamma_{39} = 0.$ The proof is divided into 4 steps.

{\it Step 1.} The homomorphism $\varphi_1$ sends (\ref{ct438}) to
\begin{equation}
 \gamma_{7}\theta_{1} + \gamma_{21}(\theta_{2}+ \theta_{1}) + \gamma_{24}\theta_{3}  +\gamma_{39}\theta_{4} \equiv 0.\label{ct439}
\end{equation}
Combining (\ref{ct438}) and (\ref{ct439}) gives
\begin{equation}
 \gamma_{25}\theta_{1} \equiv 0.\label{ct4310}
\end{equation}

By an analogous argument as given in the proof of the proposition for the case $s = 3$,  $[\theta_1] \ne 0$. So, we get $\gamma_{21} = 0.$

{\it Step 2.} Applying the homomorphism $\varphi_2$ to (\ref{ct437}), we obtain
\begin{equation}
 \gamma_{7}\theta_{2} + \gamma_{24}\theta_{3} +\gamma_{39}\theta_{4} =0.\label{ct4311}
\end{equation}
Using (\ref{ct4311})  and by a same argument as given in Step 1, we get $\gamma_{7} = 0$.

{\it Step 3.} Under the homomorphism $\varphi_3$, the image of (\ref{ct437}) is
\begin{equation}
  \gamma_{24}[\theta_{2}]  +\gamma_{39}[\theta_{4}] = 0.\label{ct4313}
\end{equation}
Using (\ref{ct4313})  and by a same argument as given in Step 3, we obtain $\gamma_{24}=0$. 

{\it Step 4.} Since $\gamma_{7} =\gamma_{22}= \gamma_{24}=0$, the homomorphism $\varphi_3$ sends (\ref{ct437}) to
$$\gamma_{39}[\theta_{3}] = 0.$$
From this equality and by a same argument as given in Step 3, we get $ \gamma_{39}=0.$

\medskip
For $s \geqslant 5$, $|\Phi^+(B_3(3,2^{(s-1)}))\cup A(s)| = 43$. Suppose that there is a linear relation
\begin{equation}\mathcal S = \sum_{1\leqslant j \leqslant 43}\gamma_jd_{j} \equiv 0, \label{ct4312}\end{equation}
with $\gamma _j \in \mathbb F_2$.

Using the relations $p_{(j;J)}(\mathcal S) \equiv 0,$  for $(j;J) \in \mathcal N_4$ and the admissible monomials $v_i, i= 1, 2,\ldots, 7,$ we obtain $\gamma_j = 0$ for any $j$.
The proposition is proved.  
\end{proof}

\subsection{The case of degree $2^{s+t+1} + 2^{s+1}-3$}\label{s7} \
\setcounter{equation}{0}

\medskip
First of all, we determine the $\omega$-vector of an admissible monomial of degree $n = 2^{s+t+1}+2^{s+1} -3$ for any positive integers $s, t$.

\begin{lems}\label{6.1.1} Let  $x$ be a  monomial of degree $2^{s+t+1}+2^{s+1}-3$ in $P_4$ with $s,t$ positive integers.  If $x$ is admissible, then either $\omega (x) = (3^{(s)},1^{(t+1)})$ or $\omega (x) = (3^{(s+1)},2^{(t-1)})$.
\end{lems}

\begin{proof} Observe  that the monomial $z=x_1^{2^{s+t+1}-1}x_2^{2^s-1}x_3^{2^s-1}$ is the minimal spike of degree $2^{s+t+1} + 2^{s+1} - 3$ in $P_4$ and 
$ \omega (z) = (3^{(s)},1^{(t+1)}) $. 
Since $x$ is admissible and $2^{s+t+1}+2^{s+1}-3$ is odd, using Theorem \ref{dlsig}, we obtain $\omega_1(x)=3$.  Using Theorem \ref{dlsig} and Proposition \ref{mdcb3}, we get $\omega_i(x) = 3$ for $i=1,2,\ldots , s$. 

Let $x' = \prod_{i\geqslant 1}X_{I_{i+s-1}(x)}^{2^{i-1}}$. Then, $\omega_i(x') = \omega_{i+s}(x), i\geqslant 1$ and $\deg (x') = 2^{t+1}-1$.
Since $x$ is admissible, using Theorem \ref{dlcb1}, we see that $x'$ is also admissible. By Lemmas \ref{b51}, either $\omega(x') = (1^{(t+1)})$ or $\omega(x') = (3,2^{(t-1)})$ or $\omega(x') = (1,3)$ for $t=2$.  By a direct computation we see that if  $\omega(x') = (1,3)$, then $x$ is inadmissible. So, the lemma is proved. 
\end{proof} 

Using Theorem \ref{dl1}, we easily obtain the following. 

\begin{props}\label{dlc6} For any positive integers $s, t$ with $ s \geqslant 3$, $\Phi(B_3(n))$ is a minimal set of generators for $\mathcal A$-module $P_4$ in degree $n = 2^{s+t+1}+2^{s+1}-3$. 
\end{props} 

Hence,  it suffices to consider the subcases $s = 1$ and $s = 2$.

\medskip\noindent
{5.4.1. \bf The subcase $s=1$.}\label{s61}\

\medskip

For  $s = 1$, $n = 2^{t+2} + 1=  (2^{t+2} - 1) + (2-1) + (2-1)$.   Hence,  $\mu(2^{t+2} + 1) = 3$ and  Kameko's homomorphism 
$$ \widetilde{Sq}^0_* :(QP_3)_{2^{t+2} + 1} \to (QP_3)_{2^{t+1}-1}$$ 
is an isomorphism. So, we get
$$B_3(n) = \psi(B_3(2^{t+1}-1)) = \psi(B_3(1^{(t+1)}))\cup \psi(B_3(3,2^{(t-1)})) .$$

\begin{props}  For any positive integer $t$, $C_4(n) = \Phi(B_3(n))\cup B(t)$ is the set of all the admissible monomials for $\mathcal A$-module $P_4$ in degree $n = 2^{t+2} + 1$, where the set $B(t)$ is determined as follows:
\begin{align*}B(1) &= \{x_1^3x_2^4x_3x_4\}, \quad  B(2) =\{x_1^3x_2^5x_3^8x_4\},\\
B(3) &= \{x_1^{3}x_2^{7}x_3^{11}x_4^{12}, x_1^{7}x_2^{3}x_3^{11}x_4^{12}, x_1^{7}x_2^{11}x_3^{3}x_4^{12}, x_1^{7}x_2^{7}x_3^{8}x_4^{11}, x_1^{7}x_2^{7}x_3^{11}x_4^{8}\},\\
B(t) &= \{ x_1^{3}x_2^{7}x_3^{2^{t+1}-5}x_4^{2^{t+1}-4}, x_1^{7}x_2^{3}x_3^{2^{t+1}-5}x_4^{2^{t+1}-4},\\
&\hskip 2cm x_1^{7}x_2^{2^{t+1}-5}x_3^{3}x_4^{2^{t+1}-4}, x_1^{7}x_2^{7}x_3^{2^{t+1}-8}x_4^{2^{t+1}-5}\}, \text{ for $t > 3$}. 
\end{align*}
\end{props}

The following lemma is proved by a direct computation.

\begin{lems}\label{b610}  The following monomials are strictly inadmissible: 

\medskip
{\rm i)} $X_2x_1^{2}x_2^{12},\  X_3^3x_3^4x_i^4,\ i = 1, 2, \ X_jx_1^2x_2^4x_j^8, \   X_2^3x_2^4x_j^4, \ j = 3, 4.$

{\rm ii)} $X_3x_1^2x_2^4x_3^{24}, \ X_3x_1^2x_2^4x_j^{8}x_4^{16}, j= 3,4.$

{\rm iii)} $X_3X_2^2x_1^4x_2^8x_4^{12}, \ X_4X_2^2x_1^4x_2^8x_3^{12},\ X_4X_3^2x_1^4x_2^{12}x_3^8,\ X_4X_3^2x_1^{12}x_2^4x_3^8$.

{\rm iv)} $X_j^3x_i^{4}x_j^{8}x_m^{12},\ 1\leqslant i < j \leqslant 4,\ m \ne i, j$.

{\rm v)} $X_jX_2^2x_1^4x_3^4x_2^8x_4^{8}, j = 3, 4,\ X_j^32x_1^4x_3^4x_2^8x_4^{8}, \ j = 2, 4$.

{\rm vi)} $X_3^3x_1^4x_2^4x_3^{24}x_4^{24}, \  X_3^3x_1^4x_2^4x_i^8x_4^8x_3^{16}x_4^{16}, \ X_4X_2^2x_1^4x_2^4x_i^8x_4^8x_3^{16}x_4^{16},\ i = 1, 2,$ \ 

\quad \  $X_j^3x_1^{12}x_2^{12}x_3^{16}x_4^{16},\ j= 3, 4, \  X_4X_3^2x_1^{12}x_2^{12}x_3^{16}x_4^{16}.$
\end{lems}

\begin{proof}[Proof of Proposition \ref{dlc6}] Let $x$ be an admissible monomial of degree $n = 2^{t+2}+1$. According to Lemma \ref{6.1.1}, $x = X_iy^2$ with $y$ a monomial of degree $2^{t+1}-1$. Since $x$ is admissible, by Theorem \ref{dlsig}, $y$ is admissible. By a direct computation, we see that if $y \in B_4(2^{t+1}-1)$ and $X_iy^2\notin C_4(n)$, then there is a monomial $w$ which is given in one of Lemmas \ref{3.2}, \ref{b53}, \ref{b610} such that $X_iy^2 = wz^{2^u}$ with some positive integer $u$ and monomial $z$. By Theorem \ref{dlcb1}, $x$ is inadmissible.

 Now, we prove the set $[C_4(n)]$ is linearly independent in $QP_4$. 

For $t=1$, we have $|C_4^+(9)| = 18$. Suppose there is a linear relation
\begin{equation}
\mathcal S = \sum_{i=1}^{18}\gamma_id_i \equiv 0,\label{ct441}
\end{equation}
with $\gamma_i \in \mathbb F_2$. A direct computation from the relations $p_{(r;j)}(\mathcal S) \equiv 0,$  for $1 \leqslant r < j \leqslant 4$, we obtain $\gamma_i = 0$ for $i \ne 1, 4, 9, 10, 11, 12$ and $\gamma_1 = \gamma_2 = \gamma_3 = \gamma_{10} = \gamma_{11} = \gamma_{12}$. Hence, the relation (\ref{ct441}) becomes $\gamma_1\theta \equiv 0$ where $\theta = d_1 + d_4 + d_9 + d_{10} + d_{11} + d_{12}$. 

We prove $\gamma_{1} = 0$.  Suppose $\theta$ is hit. Then we get
$$\theta = Sq^1(A) + Sq^2 (B) +Sq^4(C),$$
for some polynomials $A\in (P_4^+)_8, B \in (P_4^+)_7, C \in (P_4^+)_5$. Let $(Sq^2)^3$ act on the both sides of this equality. It is easy to check that $(Sq^2)^3Sq^4(C)$ $ = 0$ for all $C\in (P_4^+)_5$. Since $(Sq^2)^3$ annihilates $Sq^1$ and $Sq^2$, the right hand side is sent to zero. On the other hand,  a direct computation shows
$$ (Sq^2)^3(\theta) = (1,2,4,8) + \text{symmetries} \ne 0.$$ 
Hence,  we have a contradiction. So, we obtain $\gamma_{1} = 0.$

For $t = 2$, $|C_4^+(17)| = 47$. Suppose there is a linear relation

\begin{equation}
\mathcal S = \sum_{i=1}^{47}\gamma_id_i \equiv 0,\label{ct442}
\end{equation}
with $\gamma_i \in \mathbb F_2$ and $d_i = d_{17,i}$. A direct computation from the relations $p_{(j;J)}(\mathcal S) \equiv 0,$  for $(j;J) \in \mathcal N_4$, we obtain $\gamma_i = 0$ for $i \ne 1, 4, 8, 9, 10, 11, 17, 18$ and $\gamma_1 = \gamma_2 = \gamma_8 = \gamma_{9} = \gamma_{10} = \gamma_{11}  = \gamma_{17} = \gamma_{18}$. Hence, the relation (\ref{ct442}) becomes $\gamma_1\theta \equiv 0$ where $\theta = d_1 + d_4 + d_8 + d_{11} + d_{13} + d_{16} + d_{17} + d_{18}$. 

By a same argument as given in the proof of the proposition for $t=1$, we see that  $[\theta] \ne 0$. Hence,   $\gamma_{1} = 0.$

We have $|C_4^+(33)| = 84$ for $t=3$, and $|C_4^+(2^{t+2}+1)| = 94$ for $t >3$. Suppose there is a linear relation
\begin{equation}
\mathcal S = \sum_{i=1}^{84}\gamma_id_i \equiv 0,\label{ct443}
\end{equation}
with $\gamma_i \in \mathbb F_2$ and $d_i = d_{33,i}$. A direct computation from the relations $p_{(j;J)}(\mathcal S) \equiv 0,$  for $(j;J) \in \mathcal N_4$, we obtain $\gamma_i = 0$ for all $i \notin E$ with $E= \{$1, 3, 8, 9, 13, 14, 17, 24, 25, 42, 43, 59, 60, 65, 66, 67$\}$ and $\gamma_{i} =\gamma_{1}$ for all $i \in E$. Hence, the relation (\ref{ct443}) become $\gamma_1\theta \equiv 0$ with $\theta = \sum_{i \in E}d_i$.

By a same argument as given in the proof of the proposition for $t=1$, we see that $[\theta] \ne 0$. Therefore  $\gamma_{1} = 0.$

Now, we prove the set $[C_4^+(n)]$ is linearly independent for $t> 3$. Suppose there is a linear relation
\begin{equation}
\mathcal S = \sum_{i=1}^{94}\gamma_id_i \equiv 0, \label{ct444}
\end{equation}
with $\gamma_i \in \mathbb F_2$ and $d_i = d_{n,i}$. By a direct computation from the relations $p_{(j;J)}(\mathcal S) \equiv 0,$  for $(j;J) \in \mathcal N_4$, we obtain $\gamma_i = 0$ for all $i$.
\end{proof}

\medskip\noindent
{5.4.2. \bf The subcase $s = 2$.}\label{s62}\ \medskip
\setcounter{equation}{0}

For $s=2$, we have $n = 2^{t+3} + 5$.  According to Theorem \ref{dlmd2},  the iterated Kameko homomorphism 
$$ (\widetilde{Sq}^0_*)^2 :(QP_3)_{2^{t+3} + 5} \to (QP_3)_{2^{t+1}-1}$$ 
is an isomorphism. So, we get
$$B_3(n) = \psi^2(B_3(2^{t+1}-1)) = \psi^2(B_3(1^{(t+1)}))\cup \psi^2(\Phi(B_3(3,2^{(t-1)}))) .$$

\begin{props}\label{mdc42}  \

{\rm i)} $C_4(21)=\Phi(B_3(21)) \cup \{x_1^{7}x_2^{9}x_3^{2}x_4^{3},\ 
x_1^{7}x_2^{9}x_3^{3}x_4^{2},\ x_1^{3}x_2^{7}x_3^{8}x_4^{3},\ x_1^{7}x_2^{3}x_3^{8}x_4^{3}\}$  is the set of all the admissible monomials for $\mathcal A$-module $P_4$ in degree $21$.

{\rm ii)} For any integer $t > 1$, $C_4(n) = \Phi(B_3(n))$ is the set of all the admissible monomials  for $\mathcal A$-module $P_4$ in degree $n = 2^{t+3} + 5$.
\end{props}

The following lemma is proved by a direct computation.

\begin{lems}\label{b611}  The following monomials are strictly inadmissible:

\medskip
{\rm i)} $ X_2^3x_3^4,\   X_i^4X_j^3, \ 1 \leqslant i < j \leqslant 4,\ X_2^3x_1^4x_2^8.$

{\rm ii)} $X_3^3x_i^4x_3^{24},\  X_3^3x_i^4x_3^8x_4^{16},\ X_4^3x_i^4x_3^8x_4^{16}, \  X_4^7x_i^8x_4^{8}, i=1 , 2.$

 {\rm iii)}  $x_1^7x_2^{11}x_3^{17}x_4^2,\ X_j^3x_2^8x_j^{16}, X_j^7x_3^8x_4^8, j = 3,4$. 

{\rm iv)} $x_1^{15}x_2^{15}x_3^{16}x_4^{23},\ 
x_1^{15}x_2^{15}x_3^{23}x_4^{16},\ 
x_1^{15}x_2^{15}x_3^{17}x_4^{22}$.\ 
\end{lems}

\begin{proof}[Proof of Proposition \ref{mdc42}] Let $x$ be an admissible monomial of degree $n = 2^{t+3}+ 5$. According to Lemma \ref{6.1.1}, $x = X_iy^2$ with $y$ a monomial of degree $2^{t+2}+1$. Since $x$ is admissible, by Theorem \ref{dlsig}, $y$ is admissible. 

By a direct computation, we see that if $y \in B_4(2^{t+2}+1)$ and $X_iy^2$ does not belong to $C_4(n)$, then there is a monomial $w$ which is given in one of Lemmas \ref{3.2}, \ref{b53}, \ref{b611} such that $X_iy^2 = wz^{2^u}$ with some positive integer $u$ and monomial $z$. 
By Theorem \ref{dlcb1}, $x$ is inadmissible. Hence,  $QP_4(n)$ is generated by the set given in the proposition.  

Now, we prove the set $[C_4^+(n)]$ is linearly independent in $QP_4$.

 For $t=1$, we have $|C_4^+(21)| = 66$. Suppose there is a linear relation
\begin{equation}
\mathcal S = \sum_{i=1}^{66}\gamma_id_i \equiv 0,\label{ct445}
\end{equation}
with $\gamma_i \in \mathbb F_2$ and $d_i = d_{21,i}$. 

 By a simple computation, we see that $B_3(21)$ is the set consisting of 7 monomials:
\begin{align*}
&v_{1} = x_1^{3}x_2^{3}x_3^{15},\ v_{2} = x_1^{3}x_2^{7}x_3^{11}, \ v_{3} = x_1^{3}x_2^{15}x_3^{3}, \ v_{4} = x_1^{7}x_2^{3}x_3^{11}, \\ 
&\ v_{5} = x_1^{7}x_2^{11}x_3^{3}, \ v_{6} = x_1^{15}x_2^{3}x_3^{3}, \ v_{7} = x_1^{7}x_2^{7}x_3^{7}. 
\end{align*}

A direct computation shows
\begin{align*}
 p_{(1;2)}(\mathcal S) &\equiv \gamma_{1}v_{1} + \gamma_{2}v_{2} +   \gamma_{3}v_{3} +  \gamma_{10}v_{4} +  \gamma_{11}v_{5} +  \gamma_{16}v_{6} +  \gamma_{57}v_{7} \equiv 0,\\
p_{(1;3)}(\mathcal S) &\equiv \gamma_{4}v_{1} +  \gamma_{6}+\gamma_{27}v_{2} +  (\gamma_{8}+\gamma_{30}+\gamma_{49})v_{3} +  \gamma_{12}v_{4}\\
&\quad +  (\gamma_{14}+\gamma_{38}v_{5} +  \gamma_{17})v_{6} +  \gamma_{58}v_{7} \equiv 0,\\
p_{(1;4)}(\mathcal S) &\equiv  (\gamma_{5}+\gamma_{26}+\gamma_{48})v_{1} +  (\gamma_{7}+\gamma_{29}v_{2} +  \gamma_{9})v_{3} +  (\gamma_{13}+\gamma_{37})v_{4}\\
&\quad  +  \gamma_{15}v_{5} +  \gamma_{18}v_{6} +  \gamma_{59}v_{7} \equiv 0,\\
p_{(2;3)}(\mathcal S) &\equiv   \gamma_{19}v_{1} + (\gamma_{21}+\gamma_{27}+\gamma_{32}+\gamma_{60})v_{2} +  (\gamma_{23}+\gamma_{30}+\gamma_{34}+\gamma_{38}+\gamma_{40})v_{3}\\
&\quad  +  \gamma_{43}v_{4} +  (\gamma_{45}+\gamma_{49}+\gamma_{51})v_{5} +  \gamma_{54}v_{6} +  \gamma_{63}v_{7} \equiv 0,\\
p_{(2;4)}(\mathcal S) &\equiv  (\gamma_{20}+\gamma_{26}+\gamma_{33}+\gamma_{37}+\gamma_{41})v_{1} + \gamma_{22}+\gamma_{29}+\gamma_{35}+\gamma_{61})v_{2}\\
&\quad  +  \gamma_{24}v_{3} + (\gamma_{44}+\gamma_{48}+\gamma_{52})v_{4} + \gamma_{46}v_{5} +  \gamma_{55}v_{6} +   \gamma_{64}v_{7} \equiv 0,\\
p_{(3;4)}(\mathcal S) &\equiv   (\gamma_{25}+\gamma_{26}+\gamma_{27}+\gamma_{28}+ \gamma_{29}+\gamma_{30}+\gamma_{31})v_{1}\\ &\quad +  (\gamma_{36}+\gamma_{37}+\gamma_{38}+\gamma_{39}+\gamma_{62})v_{2}  +  \gamma_{42}v_{3}\\
&\quad +  (\gamma_{47}+\gamma_{48}+\gamma_{49}+\gamma_{50}+\gamma_{65})v_{4} +  \gamma_{53}v_{5} +  \gamma_{56}v_{6} +  \gamma_{66}v_{7} \equiv 0.    
\end{align*}

From the above equalities, we get
$\gamma_i = 0$, for $i =$ 1, 2, 3, 4, 9, 10, 11, 12, 15, 16, 17, 18, 19, 24, 42, 43, 46, 53, 54, 55, 56, 57, 58, 59, 63, 64, 66 and 
$\gamma_{6} = \gamma_{27}.\
 \gamma_{8}+\gamma_{30}+\gamma_{49}= 0,\ 
 \gamma_{14} = \gamma_{38},\
 \gamma_{5}+\gamma_{26}+\gamma_{48}= 0,\ 
 \gamma_{7}=\gamma_{29},\
 \gamma_{13}= \gamma_{37},\
 \gamma_{6}+\gamma_{21}+\gamma_{32}+\gamma_{60}= 0,\ 
 \gamma_{14}+\gamma_{23}+\gamma_{30}+\gamma_{34}+\gamma_{40} \gamma_{45}+\gamma_{49}+\gamma_{51}= 0,\ 
 \gamma_{20}+\gamma_{26}+\gamma_{33}+\gamma_{37}+\gamma_{41}= 0,\ 
 \gamma_{7}+\gamma_{22}+\gamma_{35}+\gamma_{61} = 0,\ 
\gamma_{44}+\gamma_{48}+\gamma_{52} = 0,\ 
\gamma_{6}+\gamma_{7}+\gamma_{25}+\gamma_{26}+\gamma_{28}+\gamma_{30}+\gamma_{31} = 0,\ 
\gamma_{14}+\gamma_{36}+\gamma_{37}+\gamma_{39}+\gamma_{62} = 0,\ 
\gamma_{47}+\gamma_{48}+\gamma_{49}+\gamma_{50}+\gamma_{65} = 0.$

\medskip
With the aid of the above equalities, we have
\begin{align*}
p_{(1;(2,3))}(\mathcal S) &\equiv \gamma_{21}v_{2} +  (\gamma_{8}+\gamma_{23}+\gamma_{30}+\gamma_{45}+\gamma_{49})v_{3} +   \gamma_{32}v_{4}\\
&\quad + (\gamma_{34}+\gamma_{45}+\gamma_{49}+\gamma_{51})v_{5} +  (\gamma_{40}+\gamma_{51})v_{6} +  \gamma_{60}v_{7}\equiv 0,\\ 
p_{(1;(2,4))}(\mathcal S) &\equiv  (\gamma_{5}+\gamma_{20}+\gamma_{26}+\gamma_{44}+\gamma_{48})v_{1} +  \gamma_{22}v_{2}\\
&\quad +  (\gamma_{33}+\gamma_{44}+\gamma_{48}+\gamma_{52})v_{4}  +  \gamma_{35}v_{5} +  (\gamma_{41}+\gamma_{52})v_{6} +  \gamma_{61}v_{7} \equiv 0.       
\end{align*}
From this, we obtain $\gamma_i = 0$, for $i =$ 21, 22, 32, 35, 60, 61 and 
$\gamma_{8}+\gamma_{23}+\gamma_{30}+\gamma_{45}+\gamma_{49} = 0,\ 
\gamma_{34}+\gamma_{45}+\gamma_{49}+\gamma_{51}  = 0,\ 
\gamma_{40} = \gamma_{51},\
 \gamma_{5}+\gamma_{20}+\gamma_{26}+\gamma_{44}+\gamma_{48}= 0,\ 
 \gamma_{33}+\gamma_{44}+\gamma_{48}+\gamma_{52} = 0,\ 
\gamma_{41} = \gamma_{52}.$ By a direct computation using the above equalities, one gets
\begin{align*}
p_{(1;(3,4))}(\mathcal S) &\equiv (\gamma_{5}+\gamma_{25}+\gamma_{26}+\gamma_{47}+\gamma_{48})v_{1} +  (\gamma_{28}+\gamma_{47}+\gamma_{48}+\gamma_{49}+\gamma_{50})v_{2}\\
&\quad  +  (\gamma_{8}+\gamma_{30}+\gamma_{31}+\gamma_{49}+\gamma_{50})v_{3} +  \gamma_{36}v_{4} +  \gamma_{39}v_{5} +  \gamma_{62}v_{7}     \equiv 0,\\ 
p_{(2;(3,4))}(\mathcal S) &\equiv   (\gamma_{13}+\gamma_{20}+\gamma_{25}+\gamma_{26}+\gamma_{33}+\gamma_{36}+\gamma_{40}+\gamma_{41})v_{1} +  (\gamma_{6}+\gamma_{7}\\
&\quad +\gamma_{13}+\gamma_{14}+\gamma_{28}+\gamma_{33}+\gamma_{34}+\gamma_{36}+\gamma_{39})v_{2}  +  (\gamma_{14}+\gamma_{23}+\gamma_{30}+\gamma_{31}\\
&\quad+\gamma_{34}+\gamma_{39}+\gamma_{40}+\gamma_{41})v_{3}  +  (\gamma_{44}+\gamma_{47}+\gamma_{48}+\gamma_{51}+\gamma_{52})v_{4} \\
&\quad +  (\gamma_{45}+\gamma_{49}+\gamma_{50}+\gamma_{51}+\gamma_{52})v_{5} +  \gamma_{65}v_{7}  \equiv 0.       
\end{align*}
So, we obtain $ \gamma_{36} =  \gamma_{39} =  \gamma_{62} =  \gamma_{65} = 0,$ $\gamma_{5}+\gamma_{25}+\gamma_{26}+\gamma_{47}+\gamma_{48}  = 0,\ 
\gamma_{28}+\gamma_{47}+\gamma_{48}+\gamma_{49}+\gamma_{50}  = 0,\ 
 \gamma_{8}+\gamma_{30}+\gamma_{31}+\gamma_{49}+\gamma_{50}  = 0,\ 
\gamma_{13}+\gamma_{20}+\gamma_{25}+\gamma_{26}+\gamma_{33}+\gamma_{40}+\gamma_{41} = 0,\ 
 \gamma_{6}+\gamma_{7}+\gamma_{13}+\gamma_{14}+\gamma_{28}+\gamma_{33}+\gamma_{34}  = 0,\ 
\gamma_{14}+\gamma_{23}+\gamma_{30}+\gamma_{31}+\gamma_{34}+\gamma_{40}+\gamma_{41}  = 0,\ 
\gamma_{44}+\gamma_{47}+\gamma_{48}+\gamma_{51}+\gamma_{52}  = 0,\ 
\gamma_{45}+\gamma_{49}+\gamma_{50}+\gamma_{51}+\gamma_{52}  = 0.$

Combining the above equalities, one gets $\gamma_i = 0$ for $i \ne$ 5, 8, 13, 14, 20, 23, 25, 26, 30, 31, 37, 38, 40, 41, 44, 45, 47, 48, 49, 50, 51, $\gamma_i = \gamma_5$ for $i = 8, 13, 14, 37, 38$, $\gamma_i = \gamma_{20}$ for $i = 23, 44, 45$, $\gamma_i = \gamma_{25}$ for $i = 40, 47, 51$, $\gamma_i = \gamma_{31}$ for $i = 41, 50, 52$, 
$\gamma_{20}+\gamma_{25}+\gamma_{49} = 0,\ 
\gamma_{5}+\gamma_{20}+\gamma_{26}+\gamma_{31} = 0,\  \gamma_{20}+\gamma_{31}+\gamma_{48} = 0,\ 
\gamma_{5}+\gamma_{20}+\gamma_{25}+\gamma_{30}  = 0.$

\medskip
Substituting the above equalities into the relation (\ref{ct445}), we have
\begin{equation} \gamma_{25}[\theta_{1}] + \gamma_{31}[\theta_{2}] + \gamma_{5}[\theta_{3}] + \gamma_{20}[\theta_{4}] =0, \label{ct4451}
\end{equation}
where
\begin{align*}
\theta_{1} &=  d_{25} + d_{30} + d_{40} + d_{47} + d_{49} + d_{51}, \\
\theta_{2} &=   d_{26} + d_{31} + d_{41} + d_{48} + d_{50} + d_{52}, \\
\theta_{3}  &= d_{5} + d_{8} + d_{13} + d_{14} + d_{26} + d_{30} + d_{37} + d_{38}, \\
\theta_{4} &=   d_{20} + d_{23} + d_{26} + d_{30} + d_{44} + d_{45} + d_{48} + d_{49}. 
\end{align*}

We need to show that $\gamma_{5} = \gamma_{20} = \gamma_{25} = \gamma_{31} = 0.$ The proof is divided into 4 steps.

{\it Step 1.} The homomorphism $\varphi_1$ sends (\ref{ct4451}) to
\begin{equation} \gamma_{25}[\theta_{1}] + \gamma_{31}[\theta_{2}] + (\gamma_{5}+\gamma_{20})[\theta_{3}] + \gamma_{20}[\theta_{4}]  =0. \label{ct4452}
\end{equation}
Combining (\ref{ct4451}) and (\ref{ct4452}) gives
$$\gamma_{20}[\theta_{3}] =0.$$
We prove $[\theta_3] \ne 0$. We have $\varphi_2\varphi_3([\theta_1]) = [\theta_3].$ So, we need only to prove that $[\theta_1] \ne 0$. Suppose $[\theta_1] = 0$. Then the polynomial $\theta_1$ is hit and we have
$$\theta_1 = Sq^1(A) +Sq^2(B) + Sq^4(C) + Sq^8(D),$$
for some polynomials $A \in (P_4^+)_{20}, B\in (P_4^+)_{19}, C \in (P_4^+)_{17}, D\in (P_4^+)_{13}$.

Let $(Sq^2)^3$ act on the both sides of this  equality. Since $(Sq^2)^3Sq^1=0$ and $(Sq^2)^3Sq^2=0$, we get
$$(Sq^2)^3(\theta_3) = (Sq^2)^3Sq^4(C) + (Sq^2)^3Sq^8(D).$$

By a direct computation, we see that the monomial $x = x_1^7x_2^{12}x_3^2x_4^6$ is a term of $(Sq^2)^3(\theta_1)$. If this monomial is a term of $(Sq^2)^3Sq^8(y)$, then $y=x_1^7f_1(z)$ with $z$ a monomial of degree 6 in $P_3$ and $x$ is a term of $x_1^7(Sq^2)^3Sq^8(f_1(z)) = 0$. So, the monomial $x$ is not a term of $(Sq^2)^3Sq^8(D)$ for all $D\in (P_4^+)_{13}$.

If this monomial is a term of $(Sq^2)^3Sq^4(y)$, where the monomial $y$ is a term of $C$, then $y=x_1^7f_1(z)$ with $z$ a monomial of degree 10 in $P_3$ and $x$ is a term of $x_1^7(Sq^2)^3Sq^4(f_1(z)) = 0$. By a direct computation, we see that either $x_1^7x_2^6x_3x_4^3$ or $x_1^7x_2^5x_3^2x_4^3$ is a term of $C$.

If $x_1^7x_2^6x_3x_4^3$ is a term of $C$ then 
$$(Sq^2)^3(\theta_1+Sq^4(x_1^7x_2^6x_3x_4^3)) = (Sq^2)^3(Sq^4(C') + Sq^8(D)),$$
where $C' = C + x_1^7x_2^6x_3x_4^3$. The monomial $x'=x_1^{16}x_2^6x_3^2x_4^3$ is a term of the polynomial $ (Sq^2)^3(\theta_1+Sq^4(x_1^7x_2^6x_3x_4^3))$. If $x'$ is a term of the polynomial $(Sq^2)^3Sq^8(y')$, with $y'$ a monomial in $(P_4^+)_{13}$. Then $y' = x_1^ax_2^bx_3^cx_4^3$ with $a \geqslant 7, b\geqslant 3, c>0$. This contradicts with the fact that $\deg y' = 13$. So, $x'$ is not a term of $(Sq^2)^3Sq^8(D)$ for all $D\in (P_4^+)_{13}$. Hence,  $x'$ is a term of $(Sq^2)^3(Sq^4(C')$. By a direct computation, we see that either $x_1^7x_2^6x_3x_4^3$ or $x_1^7x_2^5x_3^2x_4^3$ is a term of $C'$. Since $x_1^7x_2^6x_3x_4^3$ is not a term of $C'$, the monomial $x_1^7x_2^5x_3^2x_4^3$ is a term of $C'$. Then we have 
$$(Sq^2)^3(\theta_1+Sq^4(x_1^7x_2^6x_3x_4^3+x_1^7x_2^5x_3^2x_4^3)) = (Sq^2)^3(Sq^4(C'') + Sq^8(D)),$$
where $C'' = C' + x_1^7x_2^5x_3^2x_4^3 = C + x_1^7x_2^6x_3x_4^3 + x_1^7x_2^5x_3^2x_4^3$. Now the monomial $x = x_1^7x_2^{12}x_3^2x_4^6$ is a term of 
$$(Sq^2)^3(\theta_1+Sq^4(x_1^7x_2^6x_3x_4^3+x_1^7x_2^5x_3^2x_4^3)).$$ 
Hence,  either $x_1^7x_2^6x_3x_4^3$ or $x_1^7x_2^5x_3^2x_4^3$ is a term of $C''$. On the other hand, the two monomials $x_1^7x_2^6x_3x_4^3$ and $x_1^7x_2^5x_3^2x_4^3$ are  not  the  terms of $C''$. We have a contradiction. So,  one gets $\gamma_{20} = 0.$

{\it Step 2.} Since $\gamma_{20} = 0$, the homomorphism $\varphi_2$ sends (\ref{ct4452}) to
\begin{equation} \gamma_{25}[\theta_{1}] + \gamma_{31}[\theta_{2}] + \gamma_{5}[\theta_{3}] =0. \label{ct4453}
\end{equation}
Using (\ref{ct4453}) and the result in Step 1, we get  $\gamma_{5} = 0$.

{\it Step 3.} The homomorphism $\varphi_3$ sends (\ref{ct4452}) to
\begin{equation}\gamma_{25}[\theta_{4}] +  \gamma_{31}[\theta_{2}] =0. \label{ct4454}
\end{equation}
Using the relation (\ref{ct4454}) and the result in Step 2, we obtain $\gamma_{25} = 0$. 

{\it Step 4.} Since $\varphi_4([\theta_2]) = [\theta_1]$, we have
$$  \gamma_{31}[\theta_{1}] =0.$$
Using this equality and by a same argument as given in Step 3, we get $\gamma_{31}=0.$ 

 For $t>1$, we have $|C_4^+(n)| = m(t)$ with $m(2) = 95, m(3) = 128$ and $m(t) = 139$ for $t \geqslant 4$. Suppose there is a linear relation
\begin{equation}
\mathcal S = \sum_{i=1}^{m(t)}\gamma_id_i \equiv 0,\label{ct446}
\end{equation}
with $\gamma_i \in \mathbb F_2$ and $d_i = d_{n,i}$. A direct computation from the relations $p_{(j;J)}(\mathcal S) \equiv 0,$  for $(j;J) \in \mathcal N_4$, we obtain $\gamma_i = 0$ for all $i$. The proposition is proved.
\end{proof}

\subsection{The case of degree $n=2^{s+t} + 2^{s}- 2$}\label{s8} \
\setcounter{equation}{0}

\medskip
For $s \geqslant 1$ and $t \geqslant 2$, the space $(QP_4)_n$ has been determined in \cite{su}. Hence,
in this subsection we need  only to compute $(QP_4)_n$ for $n= 2^{s+1} +2^s - 2$ with $s>1$.

Recall that, the homomorphism 
$$ \widetilde{Sq}^0_* :(QP_4)_{2^{s+1} +2^s - 2} \to (QP_4)_{2^{s} +2^{s-1} - 3}$$  
is  an epimorphism. Hence,  we have
$$ (QP_4)_{2m+4}\cong (QP_4)_{m}\oplus (QP_4^0)_{2m+4}\oplus (\text{Ker}\widetilde{Sq}^0_*\cap (QP_4^+)_{2m+4}),$$
where $m = 2^{s} +2^{s-1} - 3$.  So, it suffices to compute 
$ \text{Ker}\widetilde{Sq}^0_*\cap (QP_4^+)_{n}$ for $s > 1$.

\medskip
For $ s > 1$, denote by $C(s)$ the set of all the following monomials:

\medskip
\centerline{\begin{tabular}{lll}
$x_1x_2x_3^{2^{s}-2}x_4^{2^{s+1}-2}$,& 
$x_1x_2x_3^{2^{s+1}-2}x_4^{2^{s}-2}$,& 
$x_1x_2^{2^{s}-2}x_3x_4^{2^{s+1}-2}$,\cr 
$x_1x_2^{2^{s+1}-2}x_3x_4^{2^{s}-2}$,& 
$x_1x_2^{2}x_3^{2^{s}-4}x_4^{2^{s+1}-1}$,& 
$x_1x_2^{2}x_3^{2^{s+1}-1}x_4^{2^{s}-4}$,\cr 
$x_1x_2^{2^{s+1}-1}x_3^{2}x_4^{2^{s}-4}$,& 
$x_1^{2^{s+1}-1}x_2x_3^{2}x_4^{2^{s}-4}$,& 
$x_1x_2^{2}x_3^{2^{s+1}-3}x_4^{2^{s}-2}$,\cr 
$x_1x_2^{3}x_3^{2^{s+1}-4}x_4^{2^{s}-2}$,& 
$x_1^{3}x_2x_3^{2^{s+1}-4}x_4^{2^{s}-2}$.&\cr 
\end{tabular}}

\medskip
For $s > 2$, denote by $D(s)$ the set of all the following monomials:

\medskip
\centerline{\begin{tabular}{lll}
$x_1x_2^{2}x_3^{2^{s}-3}x_4^{2^{s+1}-2}$,& 
$x_1x_2^{2}x_3^{2^{s}-1}x_4^{2^{s+1}-4}$,& 
$x_1x_2^{2}x_3^{2^{s+1}-4}x_4^{2^{s}-1}$,\cr 
$x_1x_2^{2^{s}-1}x_3^{2}x_4^{2^{s+1}-4}$,& 
$x_1^{2^{s}-1}x_2x_3^{2}x_4^{2^{s+1}-4}$,& 
$x_1x_2^{3}x_3^{2^{s}-4}x_4^{2^{s+1}-2}$,\cr 
$x_1x_2^{3}x_3^{2^{s+1}-2}x_4^{2^{s}-4}$,& 
$x_1^{3}x_2x_3^{2^{s}-4}x_4^{2^{s+1}-2}$,& 
$x_1^{3}x_2x_3^{2^{s+1}-2}x_4^{2^{s}-4}$,\cr 
$x_1x_2^{3}x_3^{2^{s}-2}x_4^{2^{s+1}-4}$,& 
$x_1^{3}x_2x_3^{2^{s}-2}x_4^{2^{s+1}-4}$,& 
$x_1^{3}x_2^{2^{s+1}-3}x_3^{2}x_4^{2^{s}-4}$,\cr 
$x_1^{3}x_2^{2^{s}-3}x_3^{2}x_4^{2^{s+1}-4}$,& 
$x_1^{3}x_2^{5}x_3^{2^{s+1}-6}x_4^{2^{s}-4}$.&\cr 
\end{tabular}}

\medskip
Set $E(2) = C(2)\cup\{x_1^3x_2^4x_3x_4\}$, $E(3) = C(3)\cup D(3)\cup\{x_1^3x_2^5x_3^6x_4^8\}$ and $E(s) = C(s)\cup D(s) \cup \{x_1^{3}x_2^{5}x_3^{2^{s}-6}x_4^{2^{s+1}-4}\},$ for $s > 3$.  

\begin{props}\label{mdc52}  For any integer $s > 1$, $E(s)\cup\Phi^0(B_3(n))\cup \psi(B_4(m))$ is the set of all the admissible monomials for $\mathcal A$-module $P_4$ in degree $n = 2m +4$ with $m= 2^{s}+2^{s-1} -3$.
\end{props}

\begin{lems}\label{b451} Let $s$ be a positive integer and let $x$ be an admissible monomial of degree $n = 2^{s+ 1}+ 2^{s} - 2$ in $P_4$. If $[x] \in  \text{\rm Ker}\widetilde{Sq}^0_*$, then $\omega(x) =  (2^{(s)},1)$. 
\end{lems}
\begin{proof} We prove the lemma by induction on $s$. Since $n = 2^{s+1} + 2^s - 2$ is even, we get  either $\omega_1(x) =0$ or $\omega_1(x) =2$ or $\omega_1(x) =4$. If $\omega_1(x) =0$, then $x = Sq^1(y)$ for some monomial $y$. If $\omega_1(x) =4$, then $x = X_\emptyset y^2$ for some monomial $y$. Since $x$ is admissible, $y$ also is admissible. This implies $\text{\rm Ker}\widetilde{Sq}^0_*([x]) = [y] \ne 0$ and we have a contradiction. 
So, $\omega_1(x) =2$ and $x = x_ix_jy^2$ with $1 \leqslant i < j \leqslant 4$, and $y$ a monomial of degree $2^{s} + 2^{s-1} - 2$ in $P_4$.  Since $x$ is admissible, $y$ is also admissible. 

If $s=1$, then $\deg y =1$. Hence, the lemma holds. Now, the lemma follows from Proposition \ref{mdcb3} and the inductive hypothesis.
\end{proof}

 The following is proved by a direct computation.

\begin{lems}\label{b452} The following monomials are strictly inadmissible:

\rm{i)} $x_i^2x_jx_m,\  x_i^3x_j^4x_m^3,\ x_i^7x_j^7x_m^8,\ 1\leqslant i < j < m \leqslant 4$.

\rm{ii)} $x_1x_2^{7}x_3^{10}x_4^{4}$, $x_1^{7}x_2x_3^{10}x_4^{4}$, $x_1x_2^{6}x_3^{7}x_4^{8}$, $x_1x_2^{7}x_3^{6}x_4^{8}$, $x_1^{7}x_2x_3^{6}x_4^{8}$, $x_1^{3}x_2^{3}x_3^{4}x_4^{12}$, $x_1^{3}x_2^{3}x_3^{12}x_4^{4}$, 
$x_1^{7}x_2^{9}x_3^{2}x_4^{4}$, $x_1^{7}x_2^{8}x_3^{3}x_4^{4}$, $x_1^{3}x_2^{5}x_3^{8}x_4^{6}$. 
\end{lems}

\begin{proof}[Proof of Proposition \ref{mdc52}] Let $x$ be an admissible monomial  of degree $n=2^{s+1} + 2^s - 2$ in  $P_4$ and $[x] \in \text{\rm Ker}\widetilde{Sq}^0_*$. By Lemma \ref{b451}, $\omega_i(x) = 2$, for $1 \leqslant i \leqslant s$, $\omega_{s+1}(x) = 1$ and $\omega_i(x) = 0$ for $i > s+1$. By induction on $s$, we see that if $x \notin E(s)\cup\Phi^0(B_3(n))$ then there is a monomial $w$ which is given in one of Lemmas \ref{4.2}, \ref{b452} such that $x = wy^{2^u}$ for some monomial $y$ and positive integer $u$. By Theorem \ref{dlcb1}, $x$ is inadmissible. Hence,  $\text{\rm Ker}\widetilde{Sq}^0_*$ is spanned by the set $[E(s)\cup\Phi^0(B_3(n))]$. Now, we prove that set $[E(s)\cup\Phi^0(B_3(n))]$ is linearly independent. 

It suffices to prove that the set $[E(s)]$ is linearly independent. For $s = 2, |E(2)| = 12$. Suppose there is a linear relation
\begin{equation}
\mathcal S = \sum_{i=1}^{12}\gamma_id_i \equiv 0,\label{ct451}
\end{equation}
with $\gamma_i \in \mathbb F_2$ and $d_i = d_{10,i}$. By a direct computation from the relations $p_{(1;j)}(\mathcal S) \equiv 0,$  for $j = 1, 2, 3$, we obtain $\gamma_i = 0$ for all $i$.

For $s > 2, |E(s)| = 26$. Suppose there is a linear relation
\begin{equation}
\mathcal S = \sum_{i=1}^{26}\gamma_id_i \equiv 0,\label{ct452}
\end{equation}
with $\gamma_i \in \mathbb F_2$ and $d_i = d_{n,i}$. By a direct computation from the relations $p_{(r;j)}(\mathcal S) \equiv 0,$  for $1 \leqslant r < j \leqslant 4$, we obtain $\gamma_i = 0$ for all $i$. The proposition is proved.
\end{proof}

\subsection{The case of degree $ 2^{s + t + u} + 2^{s + t} + 2^s - 3$} \label{sub9} \
\setcounter{equation}{0}

\medskip
First, we determine the $\omega$-vector of an admissible monomial of degree $n =  2^{s+t + u} + 2^{s+t} + 2^s - 3$.

\begin{lems}\label{b81}  If $x$ is an admissible monomial of degree $2^{s+t+u}+ 2^{s+t} + 2^s - 3$ in $P_4$ then
$\omega(x) = (3^{(s)},2^{(t)},1^{(u)})$.
\end{lems}

\begin{proof} Observe that $z = x_1^{2^{s+t+u}-1}x_2^{2^{s+t}-1}x_3^{2^s-1}$ is the minimal spike of degree $2^{s+t+u} + 2^{s+t} + 2^s-3$ and 
$\omega(z) = (3^{(s)},2^{(t)},1^{(u)})$.
Since $2^{s+t+u} + 2^{s+t} + 2^s - 3$ is odd and $x$ is admissible, using Lemma \ref{bdbs1}, we get  $\omega_i(x) =3$ for $1 \leqslant i \leqslant s$. Set $x' = \prod_{1\leqslant i \leqslant s}X_{I_{i-1}(x)}^{2^{i-1}}$.  Then, $x = x'y^{2^s}$ for some monomial $y$. We have $\omega_j(y) = \omega_{j+s}(x)$ for all $j \geqslant 1$ and
$\deg y = 2^{t+u} + 2^u -2$. 

Since $x$ is admissible, using  Theorem \ref{dlcb1}, we see that $y$ is also admissible. By a direct computation we see that if $w$ is a monomial such that $\omega(w) = (3,2,3)$, then $w$ is inadmissible. Combining this fact, Lemma \ref{b51}, Proposition \ref{mdcb3} and Theorem \ref{dlcb1}, we obtain  $\omega(y) = (2^{(t)},1^{(u)})$.
The lemma is proved.
\end{proof}

Applying Theorem \ref{dl1}, we obtain the following.
\begin{props}\label{dl810} Let $s,t,u$ be positive integers. If $s \geqslant 3$, then $\Phi(B_3(n))$ is a minimal set of generators for $\mathcal A$-module $P_4$ in degree $n = 2^{s+t+u}+ 2^{s+t} + 2^s-3$.
\end{props}

So, we need only to consider the cases $s=1$ and $s=2$.

\medskip\noindent
{5.6.1. \bf The subcase $s = t = 1$.}\label{s91}\ 
\setcounter{equation}{0} 

\medskip
For  $s=1, t = 1,$ we have $n = 2^{u+2}+ 3$.  
According to Theorem \ref{mdkmk}, we have
$$B_3(n) = \begin{cases}\psi(\Phi(B_2(2^{u+1}))), & \text{ if } \ u \ne 2,\\
\psi(\Phi(B_2(8))\cup \{x_1^7x_2^9x_3^3\}, & \text{ if } \ u = 2.
\end{cases}$$
\begin{props}\label{dl811} \

{\rm i)} $C_4(11)=\Phi(B_3(11))\cup \{x_1^{3}x_2^{4}x_3x_4^{3}, x_1^{3}x_2^{4}x_3^{3}x_4\}$ is the set of all the admissible monomials for $\mathcal A$-module $P_4$ in degree $11$.

{\rm ii)} $C_4(19) = \Phi(B_3(19))\cup \{x_1^{7}x_2^{9}x_3^{2}x_4,  x_1^{3}x_2^{12}x_3x_4^{3},  x_1^{3}x_2^{12}x_3^{3}x_4,  x_1^{3}x_2^{4}x_3x_4^{11},  x_1^{3}x_2^{4}x_3^{11}x_4,$ $ x_1^{3}x_2^{7}x_3^{8}x_4, x_1^{7}x_2^{3}x_3^{8}x_4, x_1^{7}x_2^{8}x_3x_4^{3}, x_1^{7}x_2^{8}x_3^{3}x_4, x_1^{3}x_2^{4}x_3^{3}x_4^{9}, x_1^{3}x_2^{4}x_3^{9}x_4^{3}\}$ is the set of all the admissible monomials for $\mathcal A$-module $P_4$ in degree $19$.

{\rm iii)}  $C_4(n)= \Phi(B_3(n))\cup \{x_1^{3}x_2^{4}x_3x_4^{2^{u+2}-5}, x_1^{3}x_2^{4}x_3^{2^{u+2}-5}x_4, x_1^{3}x_2^{4}x_3^{3}x_4^{2^{u+2}-7}\}$ 
is the set of all the admissible monomials for $\mathcal A$-module $P_4$ in degree $n = 2^{u+2}+ 3$, with any positive integer $u \geqslant 3$.
\end{props}

By a direct computation, we can easy obtain the following lemma.  

\begin{lems}\label{b852} The following monomials are strictly inadmissible:

\rm{i)} $x_1^3x_2^4x_3^4x_4^4x_ix_j^3, \ i, j >1, \ i \ne j,\ x_1^7x_2^3x_3^4x_4^4x_j,\  x_1^3x_2^5x_3^5x_4^5x_j,\ j = 3, 4$.

\rm{ii)} $X_2x_1^2x_j^2x_2^{28},\ X_jx_1^2x_4^2x_2^4x_3^{24},\ X_2x_1^2x_j^2x_2^4x_3^{8}x_4^{16},\ X_jx_1^2x_2^4x_3^{8}x_4^{18},\ X_jx_1^2x_2^4x_3^{10}x_4^{16}$, $X_jx_1^2x_2^2x_i^4x_3^{8}x_4^{16},\  X_3x_1^2x_2^2x_i^4x_3^{24},\ X_2x_1^2x_4^2x_2^4x_3^{24},\ i = 1, 2,\ j = 3, 4.$
\end{lems}

\begin{proof}[Proof of Theorem \ref{dl811}] Let $x$ be an admissible monomial of degree $n=2^{u+2} + 3$ in  $P_4$.  By Lemma \ref{b81}, $\omega_1(x) = 3$. So, $x = X_iy^2$ with $y$ a monomial of degree $2^{u+1}$. Since $x$ is admissible, by Theorem \ref{dlcb1}, $y \in B_4(2^{u+1})$. By a direct computation, we see that if $x = X_iy^2$ with $y \in B_4(2^{u+1})$ and $x$ does not belong to the set $C_4(n)$ as given in the proposition, then there is a monomial $w$ which is given in one of Lemmas \ref{b53}, \ref{b852} such that $x = wy^{2^r}$ for some monomial $y$ and integer $r > 1$. By Theorem \ref{dlcb1}, $x$ is inadmissible. Hence,  $(QP_4)_{n}$ is spanned by the set $[C_4(n)]$.

Now, we prove that set $[E(s)\cup\Phi^0(B_3(n))]$ is linearly independent in $QP_4$.  

Set $|C_4(2^{u+2}+3)\cap P_4^+| = m(u)$, where $m(1) = 32,\ m(2) = 80,\ m(u) = 64$ for all $u > 2$.  Suppose that there is a linear relation
\begin{equation}\mathcal S =\sum_{i=1}^{m(u)}\gamma_id_i = 0, \label{ct911}
\end{equation}
with $\gamma_i \in \mathbb F_2$ and $d_i = d_{n,i}$. By a direct computation from the relations $p_{(j;J)}(\mathcal S) \equiv 0$ with $(j;J) \in \mathcal N_4,$ we obtain $\gamma_i = 0$ for all $i$ if $u \ne 2$. 

For $u = 2$, $\gamma_j =0$ for $j = $1, 3, 4, 6, 7, 8, 9, 10, 11, 12, 14, 16, 17, 18, 19, 21, 23, 26, 27, 28, 29, 30, 31, 32, 35, 36, 38, 40, 43, 45, 51, 54, 55, 60, 61, 62, 68, 71, 79, 80, and 
$\gamma_2 = \gamma_i, i = 5, 24, 25, 41, 42, 52, 53$, 
$\gamma_{13 }= \gamma_i, i = 13, 33, 20, 56, 48, 58$, 
$\gamma_{15} = \gamma_i, i = 22, 34, 49, 57, 59$,  
$\gamma_{37} = \gamma_i,  i = 67, 70, 75$, 
$\gamma_{46} = \gamma_i, i = 69, 72, 76$,
$\gamma_{65} = \gamma_i, i =  66, 73, 74, 77, 78,$
$\gamma_{46} = \gamma_{39} + \gamma_{2}$, $\gamma_{44} = \gamma_{37} + \gamma_{2}$, $\gamma_{65} = \gamma_{47} + \gamma_{13}$, 
$\gamma_{65} = \gamma_{50} + \gamma_{22}$, $\gamma_{63} = \gamma_{37} + \gamma_{13}$, $\gamma_{64} = \gamma_{46} + \gamma_{22}$.

Substituting the above equalities into the relation (\ref{ct911}), we have
\begin{equation} 
  \gamma_{37}[\theta_1] + \gamma_{46}[\theta_2]  +  \gamma_{13}[\theta_3] + \gamma_{22}[\theta_4] + \gamma_{65}[\theta_5] + \gamma_{2}[\theta_6] = 0,\label{ct912}
\end{equation}
where 
\begin{align*}
\theta_1 &= d_{37} + d_{44} + d_{63} + d_{67} + d_{70} + d_{75},\\
\theta_2 &=  d_{39} + d_{46} + d_{64} + d_{69} + d_{72} + d_{76},\\
\theta_3 &=  d_{13} + d_{20} + d_{33} + d_{47} + d_{48} + d_{56} + d_{58} + d_{63},\\
\theta_4 &=  d_{15} + d_{22} + d_{34} + d_{49} + d_{50} + d_{57} + d_{59} + d_{64},\\
\theta_5 &= d_{47} + d_{50} + d_{65} + d_{66} + d_{73} + d_{74} + d_{77} + d_{78},\\
\theta_6 &=   d_{2} + d_{5} + d_{24} + d_{25} + d_{39} + d_{41} + d_{42} + d_{44} + d_{52} + d_{53}.
\end{align*}

We need to prove $\gamma_{2} = \gamma_{13} = \gamma_{22} = \gamma_{37} =\gamma_{46} = \gamma_{65} = 0.$ The proof is divided into 4 steps.

{\it Step 1.} First we prove $\gamma_{65} = 0$ by showing the polynomial $[\theta] = [\beta_1\theta_1 + \beta_2\theta_2 + \beta_3\theta_3 + \beta_4\theta_4 + \theta_5 +  \beta_6\theta_6] \ne 0$  for all $\beta_1, \beta_2, \beta_3, \beta_4, \beta_6 \in \mathbb F_2$. Suppose the contrary that this polynomial is hit. Then we have
$$\theta = Sq^1(A) + Sq^2(B) + Sq^4(C) + Sq^8(D),$$
for some polynomials $A, B, C, D$ in $P_4^+$. Let $(Sq^2)^3$ act on the both sides of this equality. Using the relations $(Sq^2)^3Sq^1 =0, (Sq^2)^3Sq^2 =0$, we get
$$(Sq^2)^3(\theta) = (Sq^2)^3Sq^4(C) + (Sq^2)^3Sq^8(D).$$
The monomial $x_1^7x_2^{12}x_3^4x_4^2$
is a term of $(Sq^2)^3(\theta)$. If $x_1^7x_2^{12}x_3^4x_4^2 $ is a term of the polynomial $(Sq^2)^3Sq^8(y)$ with $y$ a monomial of degree 11 in $P_4$, then $y = x_1^7f_1(z)$ with $z$ a monomial of degree 4 in $P_3$. Then $x_1^7x_2^{12}x_3^4x_4^2 $ is a term of $x_1^7(Sq^2)^3Sq^8(f_1(z)) = 0.$ This is a contradiction. So,   $x_1^7x_2^{12}x_3^4x_4^2$  is not a term of $(Sq^2)^3Sq^8(D)$ for all $D$. 
Hence,  $x_1^7x_2^{12}x_3^4x_4^2$ is a term of $(Sq^2)^3Sq^4(C)$, 
 then either $x_1^7x_2^{5}x_3x_4^2$ or $x_1^7x_2^{5}x_3^2x_4$ or $x_1^7x_2^{6}x_3x_4$ is a term of $C$.

Suppose $x_1^7x_2^{5}x_3^2x_4$  
is a term of $C$. Then 
$$(Sq^2)^3(\theta + Sq^4(x_1^7x_2^{5}x_3^2x_4)) = (Sq^2)^3(Sq^4(C') + Sq^8(D)),$$
where $C' = C +x_1^7x_2^{5}x_3^2x_4$. We see that the monomial $x_1^{16}x_2^6x_3^2x_4$ is a term of $(Sq^2)^3(\theta+Sq^4(x_1^7x_2^{5}x_3^2x_4))$. This monomial is not a term of $(Sq^2)^3Sq^8(D)$ for all $D$. So, it is a term of $(Sq^2)^3Sq^4(C')$. Then either $x_1^7x_2^{5}x_3^2x_4$ or $x_1^7x_2^{6}x_3x_4$ is a term of $C$.   Since $x_1^7x_2^{5}x_3^2x_4$ is not a term of $C'$, $x_1^7x_2^{6}x_3x_4$ is a term of $C'$. Hence,  we obtain
$$(Sq^2)^3(\theta + Sq^4(x_1^7x_2^{5}x_3^2x_4 + x_1^7x_2^{6}x_3x_4))\\
 = (Sq^2)^3(Sq^4(C'') + Sq^8(D)),
$$
where $C'' = C + x_1^7x_2^{5}x_3^2x_4 + x_1^7x_2^{6}x_3x_4$. Now $x_1^7x_2^{12}x_3^4x_4^2 $ is a term of 
$$(Sq^2)^3(\theta + Sq^4(x_1^7x_2^{5}x_3^2x_4 + x_1^7x_2^{6}x_3x_4))$$
 So, either $x_1^7x_2^{5}x_3x_4^2$ or $x_1^7x_2^{5}x_3^2x_4$ or $x_1^7x_2^{6}x_3x_4$ is a term of $C''$. Since $x_1^7x_2^{5}x_3^2x_4+x_1^7x_2^{6}x_3x_4$ is  a summand of $C''$, $x_1^7x_2^{5}x_3x_4^2$ is s term of $C''$. Then $x_1^{16}x_2^6x_3^2x_4$ is a term of $(Sq^2)^3(\theta+Sq^4(x_1^7x_2^{5}x_3^2x_4 + x_1^7x_2^{5}x_3x_4^2+ x_1^7x_2^{6}x_3x_4))$. So, either $x_1^7x_2^{5}x_3x_4^2$ or $x_1^7x_2^{5}x_3^2x_4$ or $x_1^7x_2^{6}x_3x_4$ is a term of $C'' + x_1^7x_2^{5}x_3x_4^2$
 and we have a contradiction. 

By a same argument, if either $x_1^7x_2^{5}x_3x_4^2$ or $x_1^7x_2^{6}x_3x_4$ is a term of $C$ then we have also a contradiction. Hence,  $[\theta] \ne 0$ and $\gamma_{65} = 0$. 
 
{\it Step 2.} By a direct computation, we see that the homomorphism $\varphi_3$ sends (\ref{ct912}) to
$$ \gamma_{37}[\theta_1] + \gamma_{2}[\theta_3] +\gamma_{22}[\theta_4]+  \gamma_{46}[\theta_5] + \gamma_{13}[\theta_6] = 0.
$$
By Step 1, we obtain $\gamma_{46} =0$.

{\it Step 3.} The homomorphism $\varphi_2$ sends (\ref{ct912}) to
$$
 \gamma_{13}[\theta_1] + \gamma_{22}[\theta_2] + \gamma_{37}[\theta_3] + \gamma_{2}[\theta_6]  = 0.$$ By Step 2, we obtain $\gamma_{22} =0.$

{\it Step 4.} Now the homomorphism $\varphi_3$ sends (\ref{ct912}) to
$ \gamma_{37}[\theta_2] + \gamma_{13}[\theta_4]+ \gamma_{2}[\theta_6]  = 0.
$ Combining Step 2 and Step 3, we obtain $\gamma_{13} = \gamma_{37} = 0$.

Since $\varphi_2([\theta_3]) = [\theta_6]$, we get $\gamma_2 = 0.$
So, we obtain $\gamma_j = 0$ for all $j$. The proposition follows.
\end{proof}

\medskip\noindent
{5.6.2. \bf The subcase $s = 1,\ t = 2$.}\label{s92}\ 
\setcounter{equation}{0} 

\medskip
For  $s=1, t = 2,$ we have $n = 2^{u+3}+ 7 = 2m+3$ with $m = 2^{u+2}+ 2$. Combining Theorem \ref{dl1} and Theorem \ref{mdkmk}, we have $B_3(n) = \psi(\Phi(B_2(m))).$ where
$$B_2(m) = \begin{cases}\{x_1^3x_2^7, x_1^7x_2^3\}, & \text{ if } \ u=1,\\
\{x_1^3x_2^{2^{u+2}-1}, x_1^{2^{u+2}-1}x_2^3, x_1^7x_2^{2^{u+2} - 5}\}, & \text{ if } \ u>1.
\end{cases}$$

 Denote by $F(u)$  the set of all the following monomials:
\begin{align*}
&x_1^{3}x_2^{4}x_3x_4^{2^{u+3}-1},\  
x_1^{3}x_2^{4}x_3^{2^{u+3}-1}x_4,\  
x_1^{3}x_2^{2^{u+3}-1}x_3^{4}x_4,\  
x_1^{2^{u+3}-1}x_2^{3}x_3^{4}x_4, \\ &
x_1^{3}x_2^{7}x_3^{2^{u+3}-4}x_4,\  
x_1^{7}x_2^{3}x_3^{2^{u+3}-4}x_4,\  
x_1^{7}x_2^{2^{u+3}-5}x_3^{4}x_4,\  
x_1^{7}x_2^{7}x_3^{2^{u+3}-8}x_4, \\ &
x_1^{3}x_2^{4}x_3^{3}x_4^{2^{u+3}-3},\  
x_1^{3}x_2^{4}x_3^{2^{u+3}-5}x_4^{5},\  
x_1^{3}x_2^{4}x_3^{7}x_4^{2^{u+3}-7},\  
x_1^{3}x_2^{7}x_3^{4}x_4^{2^{u+3}-7}, \\ &
x_1^{7}x_2^{3}x_3^{4}x_4^{2^{u+3}-7},\  
x_1^{3}x_2^{7}x_3^{8}x_4^{2^{u+3}-11},\  
x_1^{7}x_2^{3}x_3^{8}x_4^{2^{u+3}-11}.
\end{align*}

\begin{props}\label{dl92} \

\noindent \ \ {\rm i)} $C_4(23)=\Phi(B_3(23))\cup F(1) \cup \{x_1^{7}x_2^{9}x_3^{2}x_4^{5},\ x_1^{7}x_2^{9}x_3^{3}x_4^{4}\}$ is the set of all the admissible monomials for $\mathcal A$-module $P_4$ in degree $23$.

\noindent\ \ {\rm ii)} $C_4(n) = \Phi(B_3(n))\cup F(u) \cup \{x_1^{7}x_2^{7}x_3^{8}x_4^{2^{u+3}-15}, x_1^{7}x_2^{7}x_3^{9}x_4^{2^{u+3}-16}, x_1^{3}x_2^{4}x_3^{11}x_4^{2^{u+3}-11}\}$
  is the set of all the admissible monomials for $\mathcal A$-module $P_4$ in degree $n = 2^{u+3}+ 7$ with any positive integer $u > 1$.
\end{props}

By a direct computation, we can easy obtain the following lemma.  

\begin{lems}\label{b911} The following monomials are strictly inadmissible:

\medskip
\rm{i)} $X_2x_1^2x_j^6x_2^{12}, X_jx_1^2x_2^4x_3^8x_4^6, X_2x_1^2x_i^4x_2^8x_3^2x_4^4, X_2x_1^2x_2^4x_3^8x_4^6, \ i = 1, 2, j= 3,4$.

\rm{ii)} $X_3x_1^2x_2^2x_i^{12}x_3^{20}$, $X_3x_1^2x_2^2x_i^{4}x_3^{20}x_4^4$,  $X_jx_1^2x_2^2x_i^{12}x_3^{4}x_4^{16}$, $X_jx_1^2x_2^4x_i^{14}x_3^{16},$

\quad \ $X_jx_1^6x_2^{10}x_3^{4}x_4^{16}$, $X_jx_1^6x_2^{10}x_3^{16}x_4^{4}$, $X_3x_1^6x_2^{10}x_3^{20}$, $X_2x_1^2x_2^{4}x_3^{14}x_4^{16}, i = 1, 2, j= 3,4$.
\end{lems}

\begin{proof}[Proof of Proposition \ref{dl92}] Let $x$ be an admissible monomial of degree $n=2^{u+3} + 7$ in  $P_4$.  
By Lemma \ref{b81}, $\omega_1(x) = 3$. So, $x = X_iy^2$ with $y$ a monomial of degree $2^{u+2}+2$. Since $x$ is admissible, by Theorem \ref{dlcb1}, $y \in B_4(2^{u+2}+2)$. 

By a direct computation, we see that if $x = X_iy^2$ with $y \in B_4(2^{u+2}+2)$ and $x $ does not belong to the set $C_4(n)$ as given in the proposition, then there is a monomial $w$ which is given in one of Lemmas \ref{b911}, \ref{b53} such that $x = wy^{2^r}$ for some monomial $y$ and integer $r > 1$.
 By Theorem \ref{dlcb1}, $x$ is inadmissible. Hence,  $(QP_4)_{n}$ is spanned by the set $[C_4(n)]$.

Now, we prove that set $[C_4(n)]$ is linearly independent in $QP_4$.  

For $u = 1$, we have,  $|C_4^+(23)\cap P_4^+| = 99$.   Suppose that there is a linear relation
\begin{equation}\mathcal S =\sum_{i=1}^{99}\gamma_id_i = 0, \label{ct921}
\end{equation}
with $\gamma_i \in \mathbb F_2$ and $d_i = d_{23,i}$. By a direct computation from the relations $p_{(j;J)}(\mathcal S) \equiv 0$ with $(j;J) \in \mathcal N_4,$ we obtain $\gamma_i = 0$ for all $i \in E$, with some $E \subset \mathbb N_{99}$ and  the relation (\ref{ct912}) becomes
\begin{equation}\sum_{i=1}^{15}c_i[\theta_i] = 0,\label{ct913}\\
\end{equation}
where $c_1= \gamma_1, c_2=\gamma_4, c_3=\gamma_{33}, c_4= \gamma_{94}, c_5 = \gamma_2,  c_6 = \gamma_{22}, c_7 = \gamma_{74}, c_8 = \gamma_{29}, c_9 = \gamma_{81}, c_{10} = \gamma_{68}, c_{11} = \gamma_{10}, c_{12} = \gamma_{43}, c_{13} = \gamma_{54}, c_{14} = \gamma_{70}, c_{15} = \gamma_{11}$ and
\begin{align*}
\theta_{1} &=  d_{1} + d_{17} + d_{37} + d_{49},\\ 
\theta_{2} &= d_{4} + d_{21} + d_{44} + d_{53},\\
\theta_{3} &= d_{33} + d_{36} + d_{72} + d_{73},\\ 
\theta_{4} &= d_{94} + d_{97} + d_{98} + d_{99} ,\\
\theta_{5} &= d_{2} + d_{19} + d_{40} + d_{51},\\
\theta_{6} &= d_{22} + d_{25} + d_{62} + d_{63},\\
\theta_{7} &= d_{74} + d_{77} + d_{82} + d_{83} ,\\
\theta_{8} &=  d_{12} + d_{14} + d_{26} + d_{29} + d_{66} + d_{67},\\
\theta_{9} &= d_{40} + d_{42} + d_{78} + d_{81} + d_{86} + d_{87},\\
\theta_{10} &=  d_{10} + d_{15} + d_{24} + d_{27} + d_{46} + d_{47} + d_{64} + d_{65},\\
\theta_{11} &= d_{38} + d_{43} + d_{46} + d_{47} + d_{76} + d_{79} + d_{84} + d_{85},\\
\theta_{12} &=  d_{62} + d_{67} + d_{68} + d_{71} + d_{88} + d_{89} + d_{92} + d_{93},\\
\theta_{13} &= d_{47} + d_{54} + d_{57} + d_{62} + d_{69} + d_{82} + d_{85} + d_{88} + d_{90},\\
\theta_{14} &=   d_{12} + d_{15} + d_{19} + d_{20} + d_{46} + d_{47} + d_{51} + d_{52} + d_{58} + d_{61}\\  &\quad+ d_{64} + d_{66} + d_{67} + d_{70} + d_{84} + d_{87} + d_{89} + d_{91},\\
\theta_{15} &=   d_{11} + d_{12} + d_{18} + d_{20} + d_{24} + d_{25} + d_{26} + d_{27} + d_{38} + d_{40} + d_{45}\\  &\quad + d_{47} + d_{48} + d_{50} + d_{52} + d_{57} + d_{61} + d_{63} + d_{64} + d_{65} + d_{66}\\  &\quad + d_{67} + d_{69} + d_{77} + d_{78} + d_{83} + d_{85} + d_{86} + d_{87} +  d_{89} + d_{90}.
\end{align*}

Now, we show that $c_i = 0$ for $i = 1, 2,\ldots , 15$. The proof is divided into 6 steps.

{\it Step 1.} Set $\theta = \theta_1 + \sum_{i=2}^{15}\beta_i\theta_i$ for $\beta_i \in \mathbb F_2, i= 2, 3, \ldots, 15$. We prove that $[\theta] \ne 0$. Suppose the contrary that $\theta$ is hit. Then we have
$$\theta = Sq^1(A) + Sq^2(B) +Sq^4(C) + Sq^8(D)$$
for some polynomials $A, B, C, D \in P_4^+$. Let $(Sq^2)^3$ act to the both sides of the above equality, we obtain
$$(Sq^2)^3(\theta) = (Sq^2)^3Sq^4(C) + (Sq^2)^3Sq^8(D).$$
By a similar computation as in the proof of Proposition \ref{mdc42}, we see that the monomial $x_1^8x_2^4x_3^2x_4^{15}$ is a term of $(Sq^2)^3(\theta)$. This monomial is not a term of $(Sq^2)^3(Sq^4(C)+Sq^8(D))$ for all polynomials $C, D$ and we have a contradiction. So, $[\theta] \ne 0$  and we get $c_1 = \gamma_1 = 0$. 

By an argument analogous to the previous one, we get $c_2 = c_3 = c_4 = 0$. Now, the relation (\ref{ct913}) becomes
\begin{equation} \sum_{i=5}^{15}c_i[\theta_i] = 0. \label{ct914}
\end{equation}

{\it Step 2.} The homomorphisms $$\varphi_1, \varphi_1\varphi_3, \varphi_1\varphi_3 \varphi_4, \varphi_1\varphi_3 \varphi_2,  \varphi_1\varphi_3 \varphi_2\varphi_4, \varphi_1 \varphi_3\varphi_4\varphi_2 \varphi_3$$ send (\ref{ct914}) respectively  to
\begin{align*}
 c_{10}[\theta_3]  &= 0    \quad \text{mod}\langle [\theta_5],[\theta_6], \ldots , [\theta_{15}]\rangle,\\
 c_{9}[\theta_3]  &= 0    \quad \text{mod}\langle [\theta_5],[\theta_6], \ldots , [\theta_{15}]\rangle,\\
 c_{7}[\theta_3]&  = 0    \quad \text{mod}\langle [\theta_5],[\theta_6], \ldots , [\theta_{15}]\rangle,\\
 c_{8}[\theta_3]&  = 0    \quad \text{mod}\langle [\theta_5],[\theta_6], \ldots , [\theta_{15}]\rangle,\\
 c_{6}[\theta_3]  &= 0    \quad \text{mod}\langle [\theta_5], [\theta_6],\ldots , [\theta_{15}]\rangle,\\
 c_{5}[\theta_3]  &= 0    \quad \text{mod}\langle [\theta_5], [\theta_6],\ldots , [\theta_{15}]\rangle.
\end{align*}
Using the results in Step 1, we get $c_{5} = c_{6} = c_{7} =  c_{8} = c_{9} = c_{10} = 0$.
So, the relation (\ref{ct914}) becomes
\begin{equation} 
c_{11}[\theta_{11}] + c_{12}[\theta_{12}] + c_{13}[\theta_{13}] + c_{14}[\theta_{14}] + c_{15}[\theta_{15}] = 0. \label{ct924}
\end{equation}

{\it Step 3.} The homomorphism $\varphi_1$ sends (\ref{ct924}) to
\begin{align*}
&c_{13}[\theta_6]+ (c_{14}+c_{15})[\theta_7]+(c_{11}+c_{12})[\theta_{11}]\\
&\quad + c_{12}[\theta_{12}] + c_{13}[\theta_{13}] + c_{14}[\theta_{14}] + c_{15}[\theta_{15}] = 0. 
\end{align*}
By Step 2, we get $c_{13} = 0$ and $c_{14}=c_{15}$. So, the relation (\ref{ct924}) becomes
\begin{equation} 
c_{11}[\theta_{11}] + c_{12}[\theta_{12}] + c_{14}[\theta_{14}] + c_{14}[\theta_{15}] = 0. \label{ct925}
\end{equation}

{\it Step 4.} The homomorphism $\varphi_3$ sends (\ref{ct925}) to
$$c_{11}[\theta_{11}] + c_{14}[\theta_{12}] + (c_{12}+c_{14})[\theta_{13}] +c_{14}[\theta_{14}] + c_{14}[\theta_{15}] = 0.$$
By Step 3, we get $c_{12}=c_{14}$. Then the relation (\ref{ct925}) becomes
\begin{equation} 
c_{11}[\theta_{11}] + c_{12}[\theta_{12}] + c_{12}[\theta_{14}] + c_{12}[\theta_{15}] = 0. \label{ct926}
\end{equation}

{\it Step 5.} The homomorphism $\varphi_2$ sends (\ref{ct926}) to
$$(c_{11}+c_{12})[\theta_{12}] +  c_{12}[\theta_{14}] + c_{12}[\theta_{15}] = 0.$$
From the result in Step 4, we get $c_{11}=0$. Then the relation (\ref{ct926}) becomes
\begin{equation} 
c_{12}([\theta_{12}] + [\theta_{14}] + [\theta_{15}]) = 0. \label{ct927}
\end{equation}
{\it Step 6.} The homomorphism $\varphi_1$ sends (\ref{ct927}) to
$$c_{12}[\theta_{11}] + c_{12}([\theta_{12}] + [\theta_{14}] + [\theta_{15}]) = 0. 
$$
By the result in Step 5, we have $c_{12} = 0.$
The case $u=1$ of the proposition is completely proved.

For $u > 1$, we have $|C_4(n)^+| = 141$. Suppose that there is a linear relation
\begin{equation}\mathcal S =\sum_{i=1}^{141}\gamma_id_i = 0, \label{ct931}
\end{equation}
with $\gamma_i \in \mathbb F_2$ and $d_i = d_{n,i}\in B_4^+(n)$. By a direct computation from the relations $p_{(j;J)}(\mathcal S) \equiv 0$ with $(j;J) \in \mathcal N_4,$ we obtain $\gamma_i = 0$ for all $i \notin E$, with some $E  \subset \mathbb N_{141}$ and the relation (\ref{ct931}) becomes
\begin{equation}\sum_{i=1}^{15}c_i[\theta_i] = 0,\label{ct932}\\
\end{equation}
where $c_1= \gamma_1, c_2=\gamma_6, c_3=\gamma_{51}, c_4= \gamma_{136}, c_5 = \gamma_2,  c_6 = \gamma_{31}, c_7 = \gamma_{107}, c_8 = \gamma_{40}, c_9 = \gamma_{116}, c_{10} = \gamma_{101}, c_{11} = \gamma_{14}, c_{12} = \gamma_{56}, c_{13} = \gamma_{79}, c_{14} = \gamma_{23}, c_{15} = \gamma_{15}$ and
\begin{align*}
\theta_{1} &=  d_{1} + d_{25} + d_{55} + d_{73},\\ 
\theta_{2}  &=  d_{6} + d_{30} + d_{66} + d_{78}, \\ 
\theta_{3} & =   d_{51} + d_{54} + d_{105} + d_{106}, \\ 
\theta_{4}  &= d_{7} + d_{8} + d_{47} + d_{48}, \\  
\theta_{5} & =   d_{2} + d_{27} + d_{58} + d_{75} , \\  
\theta_{6}  &=  d_{31} + d_{34} + d_{89} + d_{90}, \\  
\theta_{7} & =  + d_{107} + d_{110} + d_{117} + d_{118}, \\
\theta_{8}  &=  d_{16} + d_{22} + d_{35} + d_{40} + d_{94} + d_{95}, \\ 
\theta_{9} & =   d_{58} + d_{64} + d_{111} + d_{116} + d_{122} + d_{123},\\
\theta_{10} & =  d_{89} + d_{95} + d_{101} + d_{104} + d_{124} + d_{127} + d_{129} + d_{130},\\ 
\theta_{11} & =   d_{14} + d_{19} + d_{33} + d_{36} + d_{68} + d_{69} + d_{91} + d_{92},\\
\theta_{12} & =   d_{56} + d_{61} + d_{68} + d_{69} + d_{109} + d_{112} + d_{119} + d_{120},\\
\theta_{13} & =   d_{67} + d_{69} + d_{79} + d_{82} + d_{89} + d_{90} + d_{117} + d_{118} + d_{124} + d_{125},\\
\theta_{14} & =  d_{16} + d_{23} + d_{27} + d_{29} + d_{70} + d_{71} + d_{72} + d_{75} + d_{77}\\ &\quad + d_{83} + d_{88} + d_{94} + d_{95} + d_{122} + d_{123} + d_{126} + d_{127}, \\  
\theta_{15} & =   d_{15} + d_{19} + d_{26} + d_{27} + d_{33} + d_{34} + d_{35} + d_{36} + d_{58}\\ &\quad  + d_{61} + d_{68} + d_{69} + d_{70} + d_{74} + d_{75} + d_{82} + d_{83} + d_{91}\\ &\quad  + d_{92} + d_{109} + d_{110} + d_{111} + d_{112} + d_{119} + d_{120} + d_{125}.
\end{align*}

Now, we prove $c_i = 0$ for $ i = 1,2, \ldots , 15$. The proof is divided into 6 steps.

{\it Step 1.} First, we prove $c_1 = 0$. Set $\theta = \theta_1 + \sum_{j=2}^{15}c_j\theta_j$. We show that $[\theta] \ne 0$ for all $c_j \in \mathbb F_2, j = 2,3, \ldots, 15$. Suppose the contrary that $\theta$ is hit. Then we have
$$\theta = \sum_{m=0}^{u+2}Sq^{2^m}(A_m),$$
 for some polynomials $A_m, m= 0,1,\ldots, u+2$. Let $(Sq^2)^3$ act on the both sides of this equality. Since $(Sq^2)^3Sq^1 = 0, (Sq^2)^3Sq^2 =0$, we get
$$(Sq^2)^3(\theta) = \sum_{m=2}^{u+2}(Sq^2)^3Sq^{2^m}(A_m).$$
It is easy to see that the monomial $x=x_1^8x_2^4x_3^2x_4^{2^{u+3}-1}$ is a term of $(Sq^2)^3(\theta)$. Hence, it is a term of $(Sq^2)^3Sq^{2^m}(y)$ for some monomial $y$ of degree $2^{u+3}-2^m+7$ with $ m \geqslant 2$. Then $y = x_2^{2^{u+3} - 1}f_2(z)$ with $z$ a monomial of degree $8 - 2^m\leqslant 4$ in $P_3$ and $x$ is a term of $x_2^{2^{u+3} - 1}(Sq^2)^3Sq^{2^m}(z)$. If $m > 2$ then $Sq^{2^m}(z)  = 0$. If $m=2$, then $Sq^{2^2}(z)  = z^2$. Hence,  $(Sq^2)^3Sq^{2^m}(z) = (Sq^2)^3(z^2) = 0$. So, $x$ is not a term of 
$$(Sq^2)^3(\theta) = \sum_{m=2}^{u+2}(Sq^2)^3Sq^{2^m}(A_m),$$
for all polynomial $A_m$ with $m >1$. This is a contradiction. So, we get $c_1 = 0$.

By an argument analogous to the previous one, we get $c_2 = c_3 = c_4 = 0$. Then the relation (\ref{ct932}) becomes
\begin{equation} \sum_{i=5}^{15}c_i[\theta_i] = 0.\label{ct933}
\end{equation}

{\it Step 2.} The homomorphisms $$\varphi_1, \varphi_1\varphi_3, \varphi_1\varphi_3 \varphi_4, \varphi_1\varphi_3 \varphi_2,  \varphi_1\varphi_3 \varphi_2\varphi_4, \varphi_1 \varphi_3\varphi_4\varphi_2 \varphi_3$$ send (\ref{ct914}) respectively  to
\begin{align*}
 c_{10}[\theta_3]  &= 0    \quad \text{mod}\langle [\theta_5],[\theta_6], \ldots , [\theta_{15}]\rangle,\\
 c_{9}[\theta_3]  &= 0    \quad \text{mod}\langle [\theta_5],[\theta_6], \ldots , [\theta_{15}]\rangle,\\
 c_{7}[\theta_3]&  = 0    \quad \text{mod}\langle [\theta_5],[\theta_6], \ldots , [\theta_{15}]\rangle,\\
 c_{8}[\theta_3]&  = 0    \quad \text{mod}\langle [\theta_5],[\theta_6], \ldots , [\theta_{15}]\rangle,\\
 c_{6}[\theta_3]  &= 0    \quad \text{mod}\langle [\theta_5], [\theta_6],\ldots , [\theta_{15}]\rangle,\\
 c_{5}[\theta_3]  &= 0    \quad \text{mod}\langle [\theta_5], [\theta_6],\ldots , [\theta_{15}]\rangle.
\end{align*}
By  Step 1, we get $c_{5} =  c_{6} = c_{7} =  c_{8} =  c_{9} =  c_{10} =  0$.
So, the relation (\ref{ct914}) becomes
\begin{equation} 
c_{11}[\theta_{11}] + c_{12}[\theta_{12}] + c_{13}[\theta_{13}] + c_{14}[\theta_{14}] + c_{15}[\theta_{15}] = 0. \label{ct934}
\end{equation}

{\it Step 3.} Applying the homomorphism $\varphi_1$ to (\ref{ct934}), we get
$$c_{13}[\theta_6] + c_{14}[\theta_8] + (c_{11} + c_{12} + c_{15})[\theta_{11}] +  c_{12}[\theta_{12}] + c_{13}[\theta_{13}] + c_{14}[\theta_{14}] + c_{15}[\theta_{15}] = 0. $$

By the results in Step 2, we obtain $c_{13}= c_{14} = 0.$ Then the relation (\ref{ct934}) becomes
\begin{equation} 
c_{11}[\theta_{11}] + c_{12}[\theta_{12}] + c_{14}[\theta_{15}] = 0. \label{ct935}
\end{equation}

{\it Step 4.} Applying the homomorphism $\varphi_3$ to the relation (\ref{ct935}) we obtain
 $$c_{11}[\theta_{11}] + c_{12}[\theta_{13}] + c_{15}[\theta_{15}] = 0.$$

By the results in Step 3, we get $c_{12} = 0$. So, the relation (\ref{ct935}) becomes
\begin{equation} 
c_{11}[\theta_{11}]  + c_{15}[\theta_{15}] = 0. \label{ct936}
\end{equation}

{\it Step 5.} Applying the homomorphism $\varphi_2$ to the relation (\ref{ct935}) one gets
$$c_{11}[\theta_{13}] + c_{15}[\theta_{15}] = 0.$$

By Step 4, we get $c_{10} = \gamma_{41} = 0$. So, the relation (\ref{ct936}) becomes
\begin{equation} 
 c_{15}[\theta_{15}] = 0. \label{ct937}
\end{equation}

{\it Step 6.} Applying the homomorphism $\varphi_1$ to the relation (\ref{ct937}) we obtain
$$ c_{15}[\theta_{11}] + c_{15}[\theta_{15}] = 0.$$

By Step 5, we get $c_{15}$. 
 The proposition is completely proved.
\end{proof}

\medskip\noindent
 {5.6.3. \bf  The subcase $s = 1,\ t > 2$.}\label{s93}\ 
\setcounter{equation}{0} 

\medskip
For  $s=1, t > 2,$ we have $n = 2^{t + u+1}+ 2^{t+1} - 1= 2m+3$ with $m = 2^{t + u}+ 2^{t} - 2$. From Theorem \ref{mdkmk}, we have $B_3(n) = \psi(\Phi(B_2(m))). $ 

\begin{props}\label{dl94} \

{\rm i)} $C_4(n)=\Phi(B_3(n))\cup \{x_1^3x_2^4x_3^{2^{t+1}-5}x_4^{2^{t+2}-3},\ x_1^3x_2^4x_3^{2^{t+2}-5}x_4^{2^{t+1}-3}\}$ is the set of all the admissible monomials for $\mathcal A$-module $P_4$ in degree $n =  2^{t + 2}+ 2^{t+1} - 1$ with any positive integer $t > 2$.

{\rm ii)} $C_4(n)=\Phi(B_3(n))\cup A(t,u)$ is the set of all the admissible monomials for $\mathcal A$-module $P_4$ in degree $n =  2^{t + u+ 1}+ 2^{t+1} - 1$ with any positive integers $t > 2, u > 1$, where $A(t,u)$ is the set consisting of 3 monomials:
$$x_1^3x_2^4x_3^{2^{t+1}-5}x_4^{2^{t+u+1}-3},\ x_1^3x_2^4x_3^{2^{t+u+1}-5}x_4^{2^{t+1}-3},\ x_1^3x_2^4x_3^{2^{t+2}-5}x_4^{2^{t+u+1}-2^{t+1}-3}.$$ 
\end{props}

By a direct computation, we can easy obtain the following lemma.  

\begin{lems}\label{b943} The following monomials are strictly inadmissible:
$$X_3x_1^2x_2^2x_3^8x_4^{28}x_i^4,\ X_3x_1^2x_2^2x_3^8x_{4}^{12}x_i^4 ,\ i = 1, 2, \ X_4x_1^6x_2^{10}x_3^{12}x_4^{16}.$$
\end{lems}

\begin{proof}[Proof of Proposition \ref{dl94}] Let $x\in P_4$ be an admissible monomial of degree $n=2^{t+ u+1} + 2^{t+1} - 1$.  
By Lemma \ref{b81}, $\omega_1(x) = 3$. So, $x = X_iy^2$ with $y$ a monomial of degree $2^{t+u}+2^t - 2$. Since $x$ is admissible, by Theorem \ref{dlcb1}, $y \in B_4(2^{t+u}+2^t - 2)$. 

By a direct computation, we see that if $x = X_iy^2$ with $y \in B_4(2^{t+u}+2^t - 2)$ and $x $ does not belongs to the set $C_4(n)$ as given in the proposition, then there is a monomial $w$ which is given in one of Lemmas \ref{b943}  and \ref{b53} such that $x = wy^{2^r}$ for some monomial $y$ and integer $r > 1$. 
By Theorem \ref{dlcb1}, $x$ is inadmissible. Hence,  $(QP_4)_{n}$ is spanned by the set $[C_4(n)]$.

Now, we prove that set $[C_4(n)]$ is linearly independent in $QP_4$.  

We set $|C_4(n)^+| = m(t,u)$ with $m(t,1) = 84$ and $m(t,u) = 126$ for $u > 1$.   Suppose that there is a linear relation
\begin{equation*}\mathcal S =\sum_{i=1}^{m(t,u)}\gamma_id_i = 0,
\end{equation*}
with $\gamma_i \in \mathbb F_2$ and $d_i = d_{n,i}$. By a direct computation from the relations $p_{(j;J)}(\mathcal S) \equiv 0$ with $(j;J) \in \mathcal N_4,$ we obtain $\gamma_i = 0$ for all $i$.
\end{proof}

\medskip\noindent
{5.6.4. \bf The subcase $s = 2,\ t = 1$.}\label{s94}\ 
\setcounter{equation}{0}
\medskip 

For  $s=2, t = 1,$ we have $n = 2^{u+3}+ 9$.  
According to Theorem \ref{mdkmk}, we have
$$B_3(n) = \begin{cases}\psi^2(\Phi(B_2(2^{u+1}))), & \text{ if } \ u \ne 2,\\
\psi^2(\Phi(B_2(8)))\cup \{x_1^{15}x_2^{19}x_3^7\}, & \text{ if } \ u = 2.
\end{cases}$$

Denote by $G(u)$ the set of 7 monomials:
\begin{align*}
&x_1^{3}x_2^{7}x_3^{2^{u+3}-5}x_4^{4},\  x_1^{7}x_2^{3}x_3^{2^{u+3}-5}x_4^{4},\  x_1^{7}x_2^{2^{u+3}-5}x_3^{3}x_4^{4},\\
&x_1^{3}x_2^{7}x_3^{7}x_4^{2^{u+3}-8},\  x_1^{7}x_2^{3}x_3^{7}x_4^{2^{u+3}-8},\  x_1^{7}x_2^{7}x_3^{3}x_4^{2^{u+3}-8}, \  x_1^{7}x_2^{7}x_3^{2^{u+3}-8}x_4^{3},
\end{align*}
\begin{props}\label{dl95} \

{\rm i)} $C_4(25)=\Phi(B_3(25))\cup G(1)\cup\{x_1^7x_2^9x_3^3x_4^6\}$ is the set of all the admissible monomials for $\mathcal A$-module $P_4$ in degree $25$.

 {\rm ii)} $C_4(n) = \Phi(B_3(n))\cup G(u)\cup H(u)$ is the set of all the admissible monomials for $\mathcal A$-module $P_4$ in degree $n = 2^{u+3}+ 9$ with any positive integer $u>1$, where $H(u)$ is the set consisting of $5$ monomials:
\begin{align*}
&x_1^{3}x_2^{7}x_3^{11}x_4^{2^{u+3}-12},\  x_1^{7}x_2^{3}x_3^{11}x_4^{2^{u+3}-12},\  x_1^{7}x_2^{11}x_3^{3}x_4^{2^{u+3}-12},\\
&x_1^{7}x_2^{7}x_3^{8}x_4^{2^{u+3}-13},\  x_1^{7}x_2^{7}x_3^{11}x_4^{2^{u+3}-16}.
\end{align*}
\end{props}

The following is proved by a direct computation.  

\begin{lems}\label{b953} The following monomials are strictly inadmissible:

\medskip
{\rm i)} $X_3X_2^2x_1^4x_2^8x_4^4,\ X_jX_2^2x_1^4x_2^8x_4^4,\ X_3^3x_i^4x_3^8x_4^4,\  X_2^3x_1^4x_2^8x_j^4, \ i = 1, 2,\ j = 3, 4.$

{\rm ii)} $X_4X_3^2x_1^{12}x_2^{16}x_3^4, X_4X_2^2x_1^{4}x_2^{24}x_4^4,\ X_4^3x_i^{12}x_3^{16}x_4^4, X_4X_2^2x_1^{12}x_2^{16}x_4^4, X_4X_3x_1^4x_2^4x_i^8x_3^{16}, $

\quad \ $X_jX_2^2x_1^{12}x_2^{16}x_3^4,\ X_jX_2^2x_1^{12}x_2^{16}x_4^4,\ X_4X_2^2x_1^{4}x_2^{8}x_4^{20},\  X_j^3x_1^4x_2^4x_i^8x_j^{16}, \ X_2^3x_1^{12}x_2^{16}x_j^4, $ 

\quad \  $X_4^3x_i^4x_3^{12}x_4^{16}, \ X_4^3x_i^{12}x_3^4x_4^{16},\ X_3^3x_i^{12}x_3^{16}x_4^4,\ X_j^3x_1^4x_2^8x_3^{16}x_4^4, \ X_4X_2^2x_1^4x_2^8x_3^{16}x_4^4$, 

\quad \  $X_4^3x_1^4x_2^8x_3^4x_4^{16},\ i = 1, 2,\ j = 3, 4.$
\end{lems}

\begin{proof}[Proof of Proposition \ref{dl95}] Let $x$ be an admissible monomial of degree $n=2^{u+3} + 9$ in  $P_4$.  

By Lemma \ref{b81}, $\omega_1(x) = \omega_2(x)= 3$. So, $x = X_iX_j^2y^4$ with $y$ a monomial of degree $2^{u+1}$. Since $x$ is admissible, by Theorem \ref{dlcb1}, $y \in B_4(2^{t+u}+2^t - 2)$. 

By a direct computation, we see that if $x = X_iX_j^2y^4$ with $y \in B_4(2^{t+u} + 2^t - 2)$ and $x$ does not belongs to the set $C_4(n)$ given in the proposition, then there is a monomial $w$ which is given in one of Lemmas \ref{b953}, \ref{b53} such that $x = wy^{2^r}$ for some monomial $y$ and integer $r > 1$. 
 By Theorem \ref{dlcb1}, $x$ is inadmissible. Hence,  $(QP_4)_{n}$ is spanned by the set $[C_4(n)]$.

Now, we prove that set $[C_4(n)]$ is linearly independent in $QP_4$.  

We denote  $|C_4(n)^+| = m(u)$ with $m(1) = 88$, $m(2) = 165$  and $m(u) = 154$ for $u \geqslant 3$.   Suppose that there is a linear relation
\begin{equation*}\mathcal S =\sum_{i=1}^{m(u)}\gamma_id_i = 0, 
\end{equation*}
with $\gamma_i \in \mathbb F_2$ and $d_i = d_{n,i}$. By a direct computation from the relations $p_{(j;J)}(\mathcal S) \equiv 0$ with $(j;J) \in \mathcal N_4,$ we obtain $\gamma_i = 0$ for all $i$.
\end{proof}

\medskip\noindent
{5.6.5. \bf The subcase $s = 2,\ t \geqslant 2$.}\label{s95}\ 
\setcounter{equation}{0} 

\medskip
For  $s=2, t \geqslant 2,$ we have $n = 2^{t+u+2}+ 2^{t+2} +1 = 4m +9$ with $m = 2^{t+u} + 2^t - 2$.  
From Theorem \ref{dl1}, we have
$$B_3(n) = \psi^2(\Phi(B_2(m))).$$

Denote by $B(t,u)$ the set of 8 monomials:
\begin{align*}
&x_1^{3}x_2^{7}x_3^{2^{t+2}-5}x_4^{2^{t+u+2}-4},\  x_1^{7}x_2^{3}x_3^{2^{t+2}-5}x_4^{2^{t+u+2}-4},\ x_1^{7}x_2^{2^{t+2}-5}x_3^{3}x_4^{2^{t+u+2}-4},\\
&x_1^{3}x_2^{7}x_3^{2^{t+u+2}-5}x_4^{2^{t+2}-4},\  x_1^{7}x_2^{3}x_3^{2^{t+u+2}-5}x_4^{2^{t+2}-4},\ x_1^{7}x_2^{2^{t+u+2}-5}x_3^{3}x_4^{2^{t+2}-4},\\
&x_1^{7}x_2^{7}x_3^{2^{t+2}-8}x_4^{2^{t+u+2}-5},\  x_1^{7}x_2^{7}x_3^{2^{t+u+2}-8}x_4^{2^{t+2}-5},
\end{align*}
and by $C(t,u)$ the set of 4 monomials:
\begin{align*}
&x_1^{3}x_2^{7}x_3^{2^{t+3}-5}x_4^{2^{t+u+2}-2^{t+2}-4},\  x_1^{7}x_2^{3}x_3^{2^{t+3}-5}x_4^{2^{t+u+2}-2^{t+2}-4},\\
&x_1^{7}x_2^{2^{t+3}-5}x_3^{3}x_4^{2^{t+u+2}-2^{t+2}-4},\  x_1^{7}x_2^{7}x_3^{2^{t+3}-8}x_4^{2^{t+u+2}-2^{t+2}-5}.
\end{align*}
\begin{props}\label{dl96}\

 {\rm i)} $C_4(n)= \Phi(B_3(n))\cup B(t,1)$ is the set of all the admissible monomials for $\mathcal A$-module $P_4$ in degree $n = 2^{t+3}+ 2^{t+2} +1$.

{\rm ii)} For any positive integer $t, u > 1$, $C_4(n)=\Phi(B_3(n))\cup B(t,u)\cup C(t,u)$ is the set of all the admissible monomials for $\mathcal A$-module $P_4$ in degree $n = 2^{t+u+2}+ 2^{t+2}+ 1$.
\end{props}

By a direct computation, we get the following.
\begin{lems}\label{b914} The following monomials are strictly inadmissible:

\medskip
 $X_jX_3^2x_1^{12}x_2^{12}x_3^{16}, X_4^3x_i^{12}x_3^{12}x_4^{16},\ X_4^3x_1^{12}x_2^{12}x_4^{16}, X_4^3x_1^4x_2^4x_3^{8}x_4^8x_j^{16},  X_4X_3^2x_3^4x_1^{12}x_4^{8}x_2^{16}$, 
 
 $X_4X_3^2x_1^{4}x_2^{4}x_4^{8}x_i^8x_3^{16} ,  X_j^3x_1^{4}x_2^{4}x_3^8x_i^8x_4^{16} , X_4^3x_1^4x_3^4x_2^8x_3^8x_4^{16},\  i = 1, 2,\ j = 3, 4.$
\end{lems}

\begin{proof}[Proof of Proposition \ref{dl96}] Let $x\in P_4$ be an admissible monomial  of degree $n = 2^{t+u+2}+ 2^{t+2}+ 1$.  By Lemma \ref{b81}, $\omega_1(x) = \omega_2(x)= 3$. So, $x = X_iX_j^2y^4$ with $y$ a monomial of degree $2^{t+u} + 2^t - 2$. 

Since $x$ is admissible, by Theorem \ref{dlcb1}, $y \in B_4(2^{t+u}+2^t - 2)$. By a direct computation, we see that if $x = X_iX_j^2y^4$ with $y \in B_4(2^{t+u} + 2^t - 2)$ and $x$ does not belongs to the set $C_4(n)$ as given in the proposition, then there is a monomial $w$ which is given in one of Lemmas \ref{b914}, \ref{3.2} such that $x = wy^{2^r}$ for some monomial $y$ and integer $r > 1$. 
By Theorem \ref{dlcb1}, $x$ is inadmissible. Hence,  $(QP_4)_{n}$ is spanned by the set $[C_4(n)]$.

Now, we prove that set $[C_4(n)]$ is linearly independent in $QP_4$.  

We set $|C_4(n)^+| = m(t,u)$ with $m(t,1) = 154$ and $m(t,u) = 231$ for $t \geqslant 2$.   Suppose that there is a linear relation
\begin{equation*}\mathcal S =\sum_{i=1}^{m(t,u)}\gamma_id_i = 0,
\end{equation*}
with $\gamma_i \in \mathbb F_2$ and $d_i = d_{n,i}$. By a direct computation from the relations $p_{(j;J)}(\mathcal S) \equiv 0$ with $(j;J) \in \mathcal N_4,$ we obtain $\gamma_i = 0$ for all $i$.  
\end{proof}

Theorem \ref{dl3} follows from the results in Subsections \ref{sub1}-\ref{sub9}.

\medskip
\noindent
{\bf Acknowledgment.}
I would like to thank Prof. Nguy\~\ecircumflex n H. V. H\uhorn ng for  helpful suggestions and constant encouragement. My thanks also go to all colleagues at the Department of Mathematics, Quy Nh\ohorn n University for many conversations. 

The original version of this work was completed while the author was visiting the Vietnam Institute for Advanced Study in Mathematics (VIASM) in July, 2013. He would like to thank the VIASM for supporting the visit and hospitality.
The work was also supported in part by the Research Project Grant No. B2013.28.129.

I would like to express my warmest thanks to the referees for the careful  reading and detailed comments with many helpful suggestions.

\bigskip
{}

\medskip\noindent
{Department of Mathematics, Quy Nh\ohorn n University,

\noindent
170 An D\uhorn \ohorn ng V\uhorn \ohorn ng, Quy Nh\ohorn n, B\`inh \DD \d inh, Vi\^{\d e}t Nam.}

\medskip\noindent
E-mail: nguyensum@qnu.edu.vn

\newpage
\centerline{\bf APPENDIX}

\bigskip
In the appendix, we explicitly determine $(QP_4)_{45}$ by using the algorism presented in the proof of Proposition \ref{mdc1}. For $k=4$, the degree $n=45$ is entry with $d_1= 5, d_2= d_3 =3$ and $m = 3$. It is well known that $\dim (QP_3)_3 = 7$ and 
$$B_3(3) = \{x_3^3,\ x_2x_3^2,\ x_2^3,\ x_1x_3^2,\ x_1x_2x_3,\ x_1x_2^2,\ x_1^3\}.$$
For simplycity, we denote the monomial $y = x_1^ax_2^bx_3^c \in B_3(3)$ by $(abc)$ with $0\leqslant a,b,c \leqslant 3$. By Theorem \ref{dl1}, $\dim (QP_4)_{45} = (2^4-1)\dim (QP_3)_3 = 105$ and $\Phi(B_3(n))$ is the minimal set of generators for $\mathcal A$-module $P_4$ in degree $n =45$. From the proof of Proposition \ref{mdc1}, $\Phi(B_3(n)) = \{\phi_{(i;I)}(X^7y^8) :(i;I) \in \mathcal N_4,\ y \in B_3(3)\}$. The monomials $\phi_{(i;I)}(X^7y^8)$ are determine by the following table.

\medskip
\centerline{\begin{tabular}{llllllll}
\hline
$\vdq\ \  (i;I)$  &\vdq $\quad y$ & $\vdq\phi_{(i;I)}(X^7y^8)$&\vdq\qquad &\vdq \ \  $(i;I)$  &\vdq $\quad y$ & $\vdq\phi_{(i;I)}(X^7y^8)$&\vdq\cr
\hline
\vdq $(1;\emptyset)$&\vdq $ (003)$ &\vdq $x_2^{7}x_3^{7}x_4^{31}$& \vdq \qquad  &\vdq $(1;\emptyset)$&\vdq $ (012)$ &\vdq $x_2^{7}x_3^{15}x_4^{23}$&\vdq\cr    
\vdq $(1;\emptyset)$&\vdq $ (030)$ &\vdq $x_2^{7}x_3^{31}x_4^{7}$& \vdq \qquad  &\vdq $(1;\emptyset)$&\vdq $ (102)$ &\vdq $x_2^{15}x_3^{7}x_4^{23}$&\vdq\cr    
\vdq $(1;\emptyset)$&\vdq $ (111)$ &\vdq $x_2^{15}x_3^{15}x_4^{15}$& \vdq \qquad  &\vdq $(1;\emptyset)$&\vdq $ (120)$ &\vdq $x_2^{15}x_3^{23}x_4^{7}$&\vdq\cr    
\vdq $(1;\emptyset)$&\vdq $ (300)$ &\vdq $x_2^{31}x_3^{7}x_4^{7}$& \vdq \qquad  &\vdq $(1;2)$&\vdq $ (003)$ &\vdq $x_1x_2^{6}x_3^{7}x_4^{31}$&\vdq\cr    
\vdq $(1;2)$&\vdq $ (012)$ &\vdq $x_1x_2^{6}x_3^{15}x_4^{23}$& \vdq \qquad  &\vdq $(1;2)$&\vdq $ (030)$ &\vdq $x_1x_2^{6}x_3^{31}x_4^{7}$&\vdq\cr   
 \vdq $(1;2)$&\vdq $ (102)$ &\vdq $x_1x_2^{14}x_3^{7}x_4^{23}$& \vdq \qquad  &\vdq $(1;2)$&\vdq $ (111)$ &\vdq $x_1x_2^{14}x_3^{15}x_4^{15}$&\vdq\cr    
\vdq $(1;2)$&\vdq $ (120)$ &\vdq $x_1x_2^{14}x_3^{23}x_4^{7}$& \vdq \qquad  &\vdq $(1;2)$&\vdq $ (300)$ &\vdq $x_1x_2^{30}x_3^{7}x_4^{7}$&\vdq\cr    
\vdq $(1;3)$&\vdq $ (003)$ &\vdq $x_1x_2^{7}x_3^{6}x_4^{31}$& \vdq \qquad  &\vdq $(1;3)$&\vdq $ (012)$ &\vdq $x_1x_2^{7}x_3^{14}x_4^{23}$&\vdq\cr    
\vdq $(1;3)$&\vdq $ (030)$ &\vdq $x_1x_2^{7}x_3^{30}x_4^{7}$& \vdq \qquad  &\vdq $(1;3)$&\vdq $ (102)$ &\vdq $x_1x_2^{15}x_3^{6}x_4^{23}$&\vdq\cr    
\vdq $(1;3)$&\vdq $ (111)$ &\vdq $x_1x_2^{15}x_3^{14}x_4^{15}$& \vdq \qquad  &\vdq $(1;3)$&\vdq $ (120)$ &\vdq $x_1x_2^{15}x_3^{22}x_4^{7}$&\vdq\cr    
\vdq $(1;3)$&\vdq $ (300)$ &\vdq $x_1x_2^{31}x_3^{6}x_4^{7}$& \vdq \qquad  &\vdq $(1;4)$&\vdq $ (003)$ &\vdq $x_1x_2^{7}x_3^{7}x_4^{30}$&\vdq\cr    
\vdq $(1;4)$&\vdq $ (012)$ &\vdq $x_1x_2^{7}x_3^{15}x_4^{22}$& \vdq \qquad  &\vdq $(1;4)$&\vdq $ (030)$ &\vdq $x_1x_2^{7}x_3^{31}x_4^{6}$&\vdq\cr    
\vdq $(1;4)$&\vdq $ (102)$ &\vdq $x_1x_2^{15}x_3^{7}x_4^{22}$& \vdq \qquad  &\vdq $(1;4)$&\vdq $ (111)$ &\vdq $x_1x_2^{15}x_3^{15}x_4^{14}$&\vdq\cr    
\vdq $(1;4)$&\vdq $ (120)$ &\vdq $x_1x_2^{15}x_3^{23}x_4^{6}$& \vdq \qquad  &\vdq $(1;4)$&\vdq $ (300)$ &\vdq $x_1x_2^{31}x_3^{7}x_4^{6}$&\vdq\cr    
\vdq $(2;\emptyset)$&\vdq $ (003)$ &\vdq $x_1^{7}x_3^{7}x_4^{31}$& \vdq \qquad  &\vdq $(2;\emptyset)$&\vdq $ (012)$ &\vdq $x_1^{7}x_3^{15}x_4^{23}$&\vdq\cr    
\vdq $(2;\emptyset)$&\vdq $ (030)$ &\vdq $x_1^{7}x_3^{31}x_4^{7}$& \vdq \qquad  &\vdq $(2;\emptyset)$&\vdq $ (102)$ &\vdq $x_1^{15}x_3^{7}x_4^{23}$&\vdq\cr    
\vdq $(2;\emptyset)$&\vdq $ (111)$ &\vdq $x_1^{15}x_3^{15}x_4^{15}$& \vdq \qquad  &\vdq $(2;\emptyset)$&\vdq $ (120)$ &\vdq $x_1^{15}x_3^{23}x_4^{7}$&\vdq\cr    
\vdq $(2;\emptyset)$&\vdq $ (300)$ &\vdq $x_1^{31}x_3^{7}x_4^{7}$& \vdq \qquad  &\vdq $(3;\emptyset)$&\vdq $ (003)$ &\vdq $x_1^{7}x_2^{7}x_4^{31}$&\vdq\cr    
\vdq $(3;\emptyset)$&\vdq $ (012)$ &\vdq $x_1^{7}x_2^{15}x_4^{23}$& \vdq \qquad  &\vdq $(3;\emptyset)$&\vdq $ (030)$ &\vdq $x_1^{7}x_2^{31}x_4^{7}$&\vdq\cr    
\vdq $(3;\emptyset)$&\vdq $ (102)$ &\vdq $x_1^{15}x_2^{7}x_4^{23}$& \vdq \qquad  &\vdq $(3;\emptyset)$&\vdq $ (111)$ &\vdq $x_1^{15}x_2^{15}x_4^{15}$&\vdq\cr    
\vdq $(3;\emptyset)$&\vdq $ (120)$ &\vdq $x_1^{15}x_2^{23}x_4^{7}$& \vdq \qquad  &\vdq $(3;\emptyset)$&\vdq $ (300)$ &\vdq $x_1^{31}x_2^{7}x_4^{7}$&\vdq\cr    
\vdq $(2;3)$&\vdq $ (003)$ &\vdq $x_1^{7}x_2x_3^{6}x_4^{31}$& \vdq \qquad  &\vdq $(2;3)$&\vdq $ (012)$ &\vdq $x_1^{7}x_2x_3^{14}x_4^{23}$&\vdq\cr    
\vdq $(2;3)$&\vdq $ (030)$ &\vdq $x_1^{7}x_2x_3^{30}x_4^{7}$& \vdq \qquad  &\vdq $(2;3)$&\vdq $ (102)$ &\vdq $x_1^{15}x_2x_3^{6}x_4^{23}$&\vdq\cr    
\vdq $(2;3)$&\vdq $ (111)$ &\vdq $x_1^{15}x_2x_3^{14}x_4^{15}$& \vdq \qquad  &\vdq $(2;3)$&\vdq $ (120)$ &\vdq $x_1^{15}x_2x_3^{6}x_4^{23}$&\vdq\cr    
\vdq $(2;3)$&\vdq $ (300)$ &\vdq $x_1^{31}x_2x_3^{6}x_4^{7}$& \vdq \qquad  &\vdq $(2;4)$&\vdq $ (003)$ &\vdq $x_1^{7}x_2x_3^{7}x_4^{30}$&\vdq\cr    
\vdq $(2;4)$&\vdq $ (012)$ &\vdq $x_1^{7}x_2x_3^{14}x_4^{23}$& \vdq \qquad  &\vdq $(2;4)$&\vdq $ (030)$ &\vdq $x_1^{7}x_2x_3^{31}x_4^{6}$&\vdq\cr    
\hline  
\end{tabular}}

\centerline{\begin{tabular}{llllllll}
\hline
$\vdq\ \  (i;I)$  &\vdq $\quad y$ & $\vdq\phi_{(i;I)}(X^7y^8)$&\vdq\qquad &\vdq \ \  $(i;I)$  &\vdq $\quad y$ & $\vdq\phi_{(i;I)}(X^7y^8)$&\vdq\cr
\hline
\vdq $(2;4)$&\vdq $ (102)$ &\vdq $x_1^{15}x_2x_3^{7}x_4^{22}$& \vdq \qquad  &\vdq $(2;4)$&\vdq $ (111)$ &\vdq $x_1^{15}x_2x_3^{15}x_4^{14}$&\vdq\cr    
\vdq $(2;4)$&\vdq $ (120)$ &\vdq $x_1^{15}x_2x_3^{23}x_4^{14}$& \vdq \qquad  &\vdq $(2;4)$&\vdq $ (300)$ &\vdq $x_1^{31}x_2x_3^{7}x_4^{6}$&\vdq\cr  
\vdq $(3;4)$&\vdq $ (003)$ &\vdq $x_1^{7}x_2^{7}x_3x_4^{30}$& \vdq \qquad  &\vdq $(3;4)$&\vdq $ (012)$ &\vdq $x_1^{7}x_2^{15}x_3x_4^{22}$&\vdq\cr    
\vdq $(3;4)$&\vdq $ (030)$ &\vdq $x_1^{7}x_2^{31}x_3x_4^{6}$& \vdq \qquad  &\vdq $(3;4)$&\vdq $ (102)$ &\vdq $x_1^{15}x_2^{7}x_3x_4^{14}$&\vdq\cr    
\vdq $(3;4)$&\vdq $ (111)$ &\vdq $x_1^{15}x_2^{15}x_3x_4^{14}$& \vdq \qquad  &\vdq $(3;4)$&\vdq $ (120)$ &\vdq $x_1^{15}x_2^{23}x_3x_4^{6}$&\vdq\cr    
\vdq $(3;4)$&\vdq $ (300)$ &\vdq $x_1^{31}x_2^{7}x_3x_4^{6}$& \vdq \qquad  &\vdq $(1;2,3)$&\vdq $ (003)$ &\vdq $x_1^{3}x_2^{5}x_3^{6}x_4^{31}$&\vdq\cr    
\vdq $(1;2,3)$&\vdq $ (012)$ &\vdq $x_1^{3}x_2^{5}x_3^{14}x_4^{23}$& \vdq \qquad  &\vdq $(1;2,3)$&\vdq $ (030)$ &\vdq $x_1^{3}x_2^{5}x_3^{30}x_4^{7}$&\vdq\cr    
\vdq $(1;2,3)$&\vdq $ (102)$ &\vdq $x_1^{3}x_2^{13}x_3^{6}x_4^{23}$& \vdq \qquad  &\vdq $(1;2,3)$&\vdq $ (111)$ &\vdq $x_1^{3}x_2^{13}x_3^{14}x_4^{15}$&\vdq\cr    
\vdq $(1;2,3)$&\vdq $ (120)$ &\vdq $x_1^{3}x_2^{13}x_3^{22}x_4^{7}$& \vdq \qquad  &\vdq $(1;2,3)$&\vdq $ (300)$ &\vdq $x_1^{3}x_2^{29}x_3^{6}x_4^{7}$&\vdq\cr    
\vdq $(1;2,4)$&\vdq $ (003)$ &\vdq $x_1^{3}x_2^{5}x_3^{7}x_4^{30}$& \vdq \qquad  &\vdq $(1;2,4)$&\vdq $ (012)$ &\vdq $x_1^{3}x_2^{5}x_3^{15}x_4^{22}$&\vdq\cr    
\vdq $(1;2,4)$&\vdq $ (030)$ &\vdq $x_1^{3}x_2^{5}x_3^{31}x_4^{6}$& \vdq \qquad  &\vdq $(1;2,4)$&\vdq $ (102)$ &\vdq $x_1^{3}x_2^{13}x_3^{7}x_4^{22}$&\vdq\cr    
\vdq $(1;2,4)$&\vdq $ (111)$ &\vdq $x_1^{3}x_2^{13}x_3^{15}x_4^{14}$& \vdq \qquad  &\vdq $(1;2,4)$&\vdq $ (120)$ &\vdq $x_1^{3}x_2^{13}x_3^{23}x_4^{6}$&\vdq\cr    
\vdq $(1;2,4)$&\vdq $ (300)$ &\vdq $x_1^{3}x_2^{29}x_3^{7}x_4^{6}$& \vdq \qquad  &\vdq $(1;3,4)$&\vdq $ (003)$ &\vdq $x_1^{3}x_2^{7}x_3^{5}x_4^{30}$&\vdq\cr    
\vdq $(1;3,4)$&\vdq $ (012)$ &\vdq $x_1^{3}x_2^{7}x_3^{13}x_4^{22}$& \vdq \qquad  &\vdq $(1;3,4)$&\vdq $ (030)$ &\vdq $x_1^{3}x_2^{7}x_3^{29}x_4^{6}$&\vdq\cr    
\vdq $(1;3,4)$&\vdq $ (102)$ &\vdq $x_1^{3}x_2^{15}x_3^{5}x_4^{22}$& \vdq \qquad  &\vdq $(1;3,4)$&\vdq $ (111)$ &\vdq $x_1^{3}x_2^{15}x_3^{13}x_4^{14}$&\vdq\cr    
\vdq $(1;3,4)$&\vdq $ (120)$ &\vdq $x_1^{3}x_2^{15}x_3^{21}x_4^{6}$& \vdq \qquad  &\vdq $(1;3,4)$&\vdq $ (300)$ &\vdq $x_1^{3}x_2^{31}x_3^{5}x_4^{6}$&\vdq\cr    
\vdq $(2;3,4)$&\vdq $ (003)$ &\vdq $x_1^{7}x_2^{3}x_3^{5}x_4^{30}$& \vdq \qquad  &\vdq $(2;3,4)$&\vdq $ (012)$ &\vdq $x_1^{7}x_2^{3}x_3^{13}x_4^{22}$&\vdq\cr    
\vdq $(2;3,4)$&\vdq $ (030)$ &\vdq $x_1^{7}x_2^{3}x_3^{29}x_4^{6}$& \vdq \qquad  &\vdq $(2;3,4)$&\vdq $ (102)$ &\vdq $x_1^{15}x_2^{3}x_3^{5}x_4^{22}$&\vdq\cr    
\vdq $(2;3,4)$&\vdq $ (111)$ &\vdq $x_1^{15}x_2^{3}x_3^{13}x_4^{14}$& \vdq \qquad  &\vdq $(2;3,4)$&\vdq $ (120)$ &\vdq $x_1^{15}x_2^{3}x_3^{21}x_4^{6}$&\vdq\cr    
\vdq $(2;3,4)$&\vdq $ (300)$ &\vdq $x_1^{31}x_2^{3}x_3^{5}x_4^{6}$& \vdq \qquad  &\vdq $(4;\emptyset)$&\vdq $ (003)$ &\vdq $x_1^{7}x_2^{7}x_3^{31}$&\vdq\cr    
\vdq $(4;\emptyset)$&\vdq $ (012)$ &\vdq $x_1^{7}x_2^{15}x_3^{23}$& \vdq \qquad  &\vdq $(4;\emptyset)$&\vdq $ (030)$ &\vdq $x_1^{7}x_2^{31}x_3^{7}$&\vdq\cr    
\vdq $(4;\emptyset)$&\vdq $ (102)$ &\vdq $x_1^{15}x_2^{7}x_3^{23}$& \vdq \qquad  &\vdq $(4;\emptyset)$&\vdq $ (111)$ &\vdq $x_1^{15}x_2^{15}x_3^{15}$&\vdq\cr    
\vdq $(4;\emptyset)$&\vdq $ (120)$ &\vdq $x_1^{15}x_2^{23}x_3^{7}$& \vdq \qquad  &\vdq $(4;\emptyset)$&\vdq $ (300)$ &\vdq $x_1^{31}x_2^{7}x_3^{7}$&\vdq\cr    
\vdq $(1;I_1)$&\vdq $ (003)$ &\vdq $x_1^{7}x_2^{7}x_3^{7}x_4^{24}$& \vdq \qquad  &\vdq $(1;I_1)$&\vdq $ (012)$ &\vdq $x_1^{7}x_2^{7}x_3^{9}x_4^{22}$&\vdq\cr    
\vdq $(1;I_1)$&\vdq $ (030)$ &\vdq $x_1^{7}x_2^{7}x_3^{25}x_4^{6}$& \vdq \qquad  &\vdq $(1;I_1)$&\vdq $ (102)$ &\vdq $x_1^{7}x_2^{11}x_3^{5}x_4^{22}$&\vdq\cr    
\vdq $(1;I_1)$&\vdq $ (111)$ &\vdq $x_1^{7}x_2^{11}x_3^{13}x_4^{14}$& \vdq \qquad  &\vdq $(1;I_1)$&\vdq $ (120)$ &\vdq $x_1^{7}x_2^{11}x_3^{21}x_4^{6}$&\vdq\cr    
\vdq $(1;I_1)$&\vdq $ (300)$ &\vdq $x_1^{7}x_2^{27}x_3^{5}x_4^{6}$& \vdq \qquad  &\vdq &\vdq $ $ &\vdq &\vdq\cr    
\hline
\end{tabular}}

\bigskip
Now, we compute $\phi_{(i;I)}(X^7)\bar y^8$ in terms of the monomials in $B_4(45) = \Phi(B_3(45))$ with $\bar y$ a monomial in $B_4(3)$ such that $\nu_i(\bar y) >0$ and $(i;I) \in \mathcal N_4$. It is easy to see that $B_4(3)= \Phi^0(B_3(3))$ is the set consisting of all the following monomials
\begin{align*}
&x_4^{3},\  x_3x_4^{2},\  x_3^{3},\  x_2x_4^{2},\  x_2x_3x_4,\  x_2x_3^{2},\  x_2^{3},\  x_1x_4^{2},\\ &  x_1x_3x_4,\  x_1x_3^{2},\  x_1x_2x_4,\  x_1x_2x_3,\  x_1x_2^{2},\  x_1^{3}.
\end{align*}
In the following table, we denote the monomial $\bar y = x_1^ax_2^bx_3^cx_4^d$ by $(abcd)$ with $0\leqslant a,b,c,d \leqslant 3$. If the monomial $\bar y$ satisfies the conditions of Case 3.1.$u$ in the proof of Proposition \ref{mdc1}, then we denote $\phi_{(i;I)}(X^7)\bar y^8$ by $\phi_{(i;I)}^{(u)}\bar y^8.$ Here $1 \leqslant u \leqslant 14$.

\centerline{\begin{tabular}{lllll}
\hline
$\vdq\ \  (i;I)$  &\vdq\quad \  $\bar y$ &\vdq Case & $\vdq\hskip2cm\phi_{(i;I)}(X^7)\bar y^8 \equiv $&\vdq\cr
\hline
$\vdq (1;\emptyset)$  &\vdq $(1002)$ &\vdq\quad 4 & \vdq$x_1x_2^{14}x_3^{7}x_4^{23}$ + $x_1x_2^{7}x_3^{14}x_4^{23}$ + $x_1x_2^{7}x_3^{7}x_4^{30}$&\vdq\cr      
$\vdq (1;\emptyset)$  &\vdq $(1011)$ &\vdq\quad 4 & \vdq$x_1x_2^{14}x_3^{15}x_4^{15}$ + $x_1x_2^{7}x_3^{14}x_4^{23}$ + $x_1x_2^{7}x_3^{15}x_4^{22}$&\vdq\cr      
$\vdq (1;\emptyset)$  &\vdq $(1020)$ &\vdq\quad 4 & \vdq$x_1x_2^{14}x_3^{23}x_4^{7}$ + $x_1x_2^{7}x_3^{30}x_4^{7}$ + $x_1x_2^{7}x_3^{15}x_4^{22}$&\vdq\cr      
$\vdq (1;\emptyset)$  &\vdq $(1101)$ &\vdq\quad 4 & \vdq$x_1x_2^{14}x_3^{7}x_4^{23}$ + $x_1x_2^{15}x_3^{14}x_4^{15}$ + $x_1x_2^{15}x_3^{7}x_4^{22}$&\vdq\cr      
$\vdq (1;\emptyset)$  &\vdq $(1110)$ &\vdq\quad 4 & \vdq$x_1x_2^{14}x_3^{23}x_4^{7}$ + $x_1x_2^{15}x_3^{22}x_4^{7}$ + $x_1x_2^{15}x_3^{15}x_4^{14}$&\vdq\cr      
$\vdq (1;\emptyset)$  &\vdq $(1200)$ &\vdq\quad 4 & \vdq$x_1x_2^{30}x_3^{7}x_4^{7}$ + $x_1x_2^{15}x_3^{22}x_4^{7}$ + $x_1x_2^{15}x_3^{7}x_4^{22}$&\vdq\cr      
$\vdq (1;2)$  &\vdq $(1002)$ &\vdq\quad 4 & \vdq$x_1x_2^{14}x_3^{7}x_4^{23}$ + $x_1^{3}x_2^{5}x_3^{14}x_4^{23}$ + $x_1^{3}x_2^{5}x_3^{7}x_4^{30}$&\vdq\cr      
$\vdq (1;2)$  &\vdq $(1011)$ &\vdq\quad 4 & \vdq$x_1x_2^{14}x_3^{15}x_4^{15}$ + $x_1^{3}x_2^{5}x_3^{14}x_4^{23}$ + $x_1^{3}x_2^{5}x_3^{15}x_4^{22}$&\vdq\cr      
$\vdq (1;2)$  &\vdq $(1020)$ &\vdq\quad 4 & \vdq$x_1x_2^{14}x_3^{23}x_4^{7}$ + $x_1^{3}x_2^{5}x_3^{30}x_4^{7}$ + $x_1^{3}x_2^{5}x_3^{15}x_4^{22}$&\vdq\cr      
$\vdq (1;2)$  &\vdq $(1101)$ &\vdq\quad 4 & \vdq$x_1x_2^{14}x_3^{7}x_4^{23}$ + $x_1^{3}x_2^{13}x_3^{14}x_4^{15}$ + $x_1^{3}x_2^{13}x_3^{7}x_4^{22}$&\vdq\cr      
$\vdq (1;2)$  &\vdq $(1110)$ &\vdq\quad 4 & \vdq$x_1x_2^{14}x_3^{23}x_4^{7}$ + $x_1^{3}x_2^{13}x_3^{22}x_4^{7}$ + $x_1^{3}x_2^{13}x_3^{15}x_4^{14}$&\vdq\cr      
$\vdq (1;2)$  &\vdq $(1200)$ &\vdq\quad 4 & \vdq$x_1x_2^{30}x_3^{7}x_4^{7}$ + $x_1^{3}x_2^{13}x_3^{22}x_4^{7}$ + $x_1^{3}x_2^{13}x_3^{7}x_4^{22}$&\vdq\cr      
$\vdq (1;3)$  &\vdq $(1002)$ &\vdq\quad 4 & \vdq$x_1^{3}x_2^{13}x_3^{6}x_4^{23}$ + $x_1x_2^{7}x_3^{14}x_4^{23}$ + $x_1^{3}x_2^{7}x_3^{5}x_4^{30}$&\vdq\cr      
$\vdq (1;3)$  &\vdq $(1011)$ &\vdq\quad 4 & \vdq$x_1^{3}x_2^{13}x_3^{14}x_4^{15}$ + $x_1x_2^{7}x_3^{14}x_4^{23}$ + $x_1^{3}x_2^{7}x_3^{13}x_4^{22}$&\vdq\cr      
$\vdq (1;3)$  &\vdq $(1020)$ &\vdq\quad 4 & \vdq$x_1^{3}x_2^{13}x_3^{22}x_4^{7}$ + $x_1x_2^{7}x_3^{30}x_4^{7}$ + $x_1^{3}x_2^{7}x_3^{13}x_4^{22}$&\vdq\cr      
$\vdq (1;3)$  &\vdq $(1101)$ &\vdq\quad 4 & \vdq$x_1^{3}x_2^{13}x_3^{6}x_4^{23}$ + $x_1x_2^{15}x_3^{14}x_4^{15}$ + $x_1^{3}x_2^{15}x_3^{5}x_4^{22}$&\vdq\cr      
$\vdq (1;3)$  &\vdq $(1110)$ &\vdq\quad 4 & \vdq$x_1^{3}x_2^{13}x_3^{22}x_4^{7}$ + $x_1x_2^{15}x_3^{22}x_4^{7}$ + $x_1^{3}x_2^{15}x_3^{13}x_4^{14}$&\vdq\cr      
$\vdq (1;3)$  &\vdq $(1200)$ &\vdq\quad 4 & \vdq$x_1^{3}x_2^{29}x_3^{6}x_4^{7}$ + $x_1x_2^{15}x_3^{22}x_4^{7}$ + $x_1^{3}x_2^{15}x_3^{5}x_4^{22}$&\vdq\cr      
$\vdq (1;4)$  &\vdq $(1002)$ &\vdq\quad 4 & \vdq$x_1^{3}x_2^{13}x_3^{7}x_4^{22}$ + $x_1^{3}x_2^{7}x_3^{13}x_4^{22}$ + $x_1x_2^{7}x_3^{7}x_4^{30}$&\vdq\cr      
$\vdq (1;4)$  &\vdq $(1011)$ &\vdq\quad 4 & \vdq$x_1^{3}x_2^{13}x_3^{15}x_4^{14}$ + $x_1^{3}x_2^{7}x_3^{13}x_4^{22}$ + $x_1x_2^{7}x_3^{15}x_4^{22}$&\vdq\cr      
$\vdq (1;4)$  &\vdq $(1020)$ &\vdq\quad 4 & \vdq$x_1^{3}x_2^{13}x_3^{23}x_4^{6}$ + $x_1^{3}x_2^{7}x_3^{29}x_4^{6}$ + $x_1x_2^{7}x_3^{15}x_4^{22}$&\vdq\cr      
$\vdq (1;4)$  &\vdq $(1101)$ &\vdq\quad 4 & \vdq$x_1^{3}x_2^{13}x_3^{7}x_4^{22}$ + $x_1^{3}x_2^{15}x_3^{13}x_4^{14}$ + $x_1x_2^{15}x_3^{7}x_4^{22}$&\vdq\cr      
$\vdq (1;4)$  &\vdq $(1110)$ &\vdq\quad 4 & \vdq$x_1^{3}x_2^{13}x_3^{23}x_4^{6}$ + $x_1^{3}x_2^{15}x_3^{21}x_4^{6}$ + $x_1x_2^{15}x_3^{15}x_4^{14}$&\vdq\cr      
$\vdq (1;4)$  &\vdq $(1200)$ &\vdq\quad 4 & \vdq$x_1^{3}x_2^{29}x_3^{7}x_4^{6}$ + $x_1^{3}x_2^{15}x_3^{21}x_4^{6}$ + $x_1x_2^{15}x_3^{7}x_4^{22}$&\vdq\cr      
$\vdq (2;\emptyset)$  &\vdq $(0102)$ &\vdq\quad 4 & \vdq$\phi_{(1;2)}^{(4)}(1002)^8$ + $x_1^{7}x_2x_3^{14}x_4^{23}$ + $x_1^{7}x_2x_3^{7}x_4^{30}$&\vdq\cr      
$\vdq (2;\emptyset)$  &\vdq $(0111)$ &\vdq\quad 4 & \vdq$\phi_{(1;2)}^{(4)}(1011)^8$ + $x_1^{7}x_2x_3^{14}x_4^{23}$ + $x_1^{7}x_2x_3^{15}x_4^{22}$&\vdq\cr      
$\vdq (2;\emptyset)$  &\vdq $(0120)$ &\vdq\quad 4 & \vdq$\phi_{(1;2)}^{(4)}(1020)^8$ + $x_1^{7}x_2x_3^{30}x_4^{7}$ + $x_1^{7}x_2x_3^{15}x_4^{22}$&\vdq\cr      
$\vdq (3;\emptyset)$  &\vdq $(0012)$ &\vdq\quad 5 & \vdq$\phi_{(1;\emptyset)}^{(4)}(1002)^8$ + $\phi_{(2;\emptyset)}^{(4)}(0102)^8$ + $x_1^{7}x_2^{7}x_3^{7}x_4^{24}$&\vdq\cr      
$\vdq (4;\emptyset)$  &\vdq $(0012)$ &\vdq\quad 6 & \vdq$\phi_{(1;\emptyset)}^{(4)}(1011)^8$ + $\phi_{(2;\emptyset)}^{(4)}(0111)^8$ + $\phi_{(3;\emptyset)}^{(5)}(0012)^8$&\vdq\cr      
$\vdq (3;\emptyset)$  &\vdq $(0030)$ &\vdq\quad 7 & \vdq$\phi_{(1;3)}^{(4)}(1020)^8$ + $\phi_{(2;3)}^{(4)}(0120)^8$ + $x_1^{7}x_2^{7}x_3^{9}x_4^{22}$&\vdq\cr      
$\vdq (2;3)$  &\vdq $(0102)$ &\vdq\quad 8 & \vdq$\phi_{(1;3)}^{(4)}(1002)^8$ + $\phi_{(3;\emptyset)}^{(5)}(0012)^8$ + $x_1^{7}x_2^{7}x_3^{9}x_4^{22}$&\vdq\cr      
$\vdq (2;3)$  &\vdq $(0111)$ &\vdq\quad 8 & \vdq$\phi_{(1;3)}^{(4)}(1011)^8$ + $\phi_{(3;\emptyset)}^{(5)}(0012)^8$ + $x_1^{7}x_2^{7}x_3^{9}x_4^{22}$&\vdq\cr      
$\vdq (2;3)$  &\vdq $(0120)$ &\vdq\quad 8 & \vdq$\phi_{(1;3)}^{(4)}(1020)^8$ + $\phi_{(3;\emptyset)}^{(7)}(0030)^8$ + $x_1^{7}x_2^{7}x_3^{9}x_4^{22}$&\vdq\cr      
$\vdq (2;4)$  &\vdq $(0102)$ &\vdq\quad 8 & \vdq$\phi_{(1;4)}^{(4)}(1002)^8$ + $x_1^{7}x_2^{7}x_3^{9}x_4^{22}$ + $x_1^{7}x_2^{7}x_3^{7}x_4^{24}$&\vdq\cr      
$\vdq (2;4)$  &\vdq $(0111)$ &\vdq\quad 8 & \vdq$\phi_{(1;4)}^{(4)}(1011)^8$ + $\phi_{(4;\emptyset)}^{(6)}(0012)^8$ + $x_1^{7}x_2^{7}x_3^{9}x_4^{22}$&\vdq\cr      
$\vdq (2;4)$  &\vdq $(0120)$ &\vdq\quad 8 & \vdq$\phi_{(1;4)}^{(4)}(1020)^8$ + $\phi_{(4;\emptyset)}^{(6)}(0012)^8$ + $x_1^{7}x_2^{7}x_3^{25}x_4^{6}$&\vdq\cr      
\hline
\end{tabular}}

\centerline{\begin{tabular}{lllll}
\hline
$\vdq\ \  (i;I)$  &\vdq\quad \  $\bar y$ &\vdq  Case & $\vdq\hskip2cm\phi_{(i;I)}(X^7)\bar y^8 \equiv $&\vdq\cr
\hline
$\vdq (3;\emptyset)$  &\vdq $(0111)$ &\vdq\quad 9 & \vdq$\phi_{(1;3)}^{(4)}(1101)^8$ + $\phi_{(2;3)}^{(8)}(0102)^8$ + $x_1^{7}x_2^{15}x_3x_4^{22}$&\vdq\cr      
$\vdq (3;4)$  &\vdq $(0111)$ &\vdq\quad 9 & \vdq$\phi_{(1;4)}^{(4)}(1101)^8$ + $\phi_{(2;4)}^{(8)}(0102)^8$ + $\phi_{(1;\emptyset)}^{(4)}(1101)^8$&\vdq\cr \vdq&\vdq &\vdq &\vdq \hskip1.5cm  + $\phi_{(2;\emptyset)}^{(4)}(0102)^8$ + $\phi_{(3;\emptyset)}^{(9)}(0111)^8$&\vdq\cr      
$\vdq (3;\emptyset)$  &\vdq $(0120)$ &\vdq\quad 9 & \vdq$\phi_{(1;3)}^{(4)}(1110)^8$ + $\phi_{(2;3)}^{(8)}(0120)^8$ + $\phi_{(3;4)}^{(9)}(0111)^8$&\vdq\cr      
$\vdq (3;4)$  &\vdq $(0120)$ &\vdq\quad 9 & \vdq$\phi_{(1;4)}^{(4)}(1110)^8$ + $\phi_{(2;4)}^{(8)}(0120)^8$ + $\phi_{(1;\emptyset)}^{(4)}(1110)^8$&\vdq\cr \vdq&\vdq &\vdq &\vdq \hskip1.5cm  + $\phi_{(2;\emptyset)}^{(4)}(0120)^8$ + $\phi_{(3;\emptyset)}^{(9)}(0120)^8$&\vdq\cr      
$\vdq (2;\emptyset)$  &\vdq $(0300)$ &\vdq\quad 10 & \vdq$\phi_{(1;2)}^{(4)}(1020)^8$ + $\phi_{(2;3)}^{(8)}(0102)^8$ + $\phi_{(2;4)}^{(8)}(0102)^8$&\vdq\cr      
$\vdq (2;3)$  &\vdq $(0300)$ &\vdq\quad 10 & \vdq$\phi_{(1;3)}^{(4)}(1200)^8$ + $\phi_{(3;\emptyset)}^{(9)}(0120)^8$ + $x_1^{7}x_2^{15}x_3x_4^{22}$&\vdq\cr      
$\vdq (2;4)$  &\vdq $(0300)$ &\vdq\quad 10 & \vdq$\phi_{(1;4)}^{(4)}(1200)^8$ + $\phi_{(3;4)}^{(9)}(0120)^8$ + $\phi_{(1;\emptyset)}^{(4)}(1200)^8$&\vdq\cr \vdq&\vdq &\vdq &\vdq \hskip1.5cm  + $\phi_{(2;\emptyset)}^{(10)}(0300)^8$ + $\phi_{(3;\emptyset)}^{(9)}(0120)^8$&\vdq\cr      
$\vdq (1;2,3)$  &\vdq $(1002)$ &\vdq\quad 11 & \vdq$\phi_{(2;3)}^{(8)}(0102)^8$ + $x_1^{7}x_2x_3^{14}x_4^{23}$ + $x_1^{7}x_2^{3}x_3^{5}x_4^{30}$&\vdq\cr      
$\vdq (1;2,3)$  &\vdq $(1011)$ &\vdq\quad 11 & \vdq$\phi_{(2;3)}^{(8)}(0111)^8$ + $x_1^{7}x_2x_3^{14}x_4^{23}$ + $x_1^{7}x_2^{3}x_3^{13}x_4^{22}$&\vdq\cr      
$\vdq (1;2,3)$  &\vdq $(1020)$ &\vdq\quad 11 & \vdq$\phi_{(2;3)}^{(8)}(0120)^8$ + $x_1^{7}x_2x_3^{30}x_4^{7}$ + $x_1^{7}x_2^{3}x_3^{13}x_4^{22}$&\vdq\cr      
$\vdq (1;2,3)$  &\vdq $(1101)$ &\vdq\quad 11 & \vdq$\phi_{(2;3)}^{(8)}(0102)^8$ + $\phi_{(2;3)}^{(8)}(0111)^8$ + $x_1^{7}x_2^{11}x_3^{5}x_4^{22}$&\vdq\cr      
$\vdq (1;2,3)$  &\vdq $(1110)$ &\vdq\quad 11 & \vdq$x_1^{7}x_2^{11}x_3^{13}x_4^{14}$&\vdq\cr      
$\vdq (1;2,3)$  &\vdq $(1200)$ &\vdq\quad 11 & \vdq$\phi_{(2;3)}^{(8)}(0300)^8$ + $\phi_{(2;3)}^{(8)}(0120)^8$ + $x_1^{7}x_2^{11}x_3^{5}x_4^{22}$&\vdq\cr      
$\vdq (1;2,4)$  &\vdq $(1002)$ &\vdq\quad 11 & \vdq$\phi_{(2;4)}^{(8)}(0102)^8$ + $x_1^{7}x_2^{3}x_3^{13}x_4^{22}$ + $x_1^{7}x_2x_3^{7}x_4^{30}$&\vdq\cr      
$\vdq (1;2,4)$  &\vdq $(1011)$ &\vdq\quad 11 & \vdq$\phi_{(2;4)}^{(8)}(0111)^8$ + $x_1^{7}x_2^{3}x_3^{13}x_4^{22}$ + $x_1^{7}x_2x_3^{15}x_4^{22}$&\vdq\cr      
$\vdq (1;2,4)$  &\vdq $(1020)$ &\vdq\quad 11 & \vdq$\phi_{(2;4)}^{(8)}(0120)^8$ + $x_1^{7}x_2^{3}x_3^{29}x_4^{6}$ + $x_1^{7}x_2x_3^{15}x_4^{22}$&\vdq\cr      
$\vdq (1;2,4)$  &\vdq $(1101)$ &\vdq\quad 11 & \vdq$x_1^{7}x_2^{11}x_3^{13}x_4^{14}$&\vdq\cr      
$\vdq (1;2,4)$  &\vdq $(1110)$ &\vdq\quad 11 & \vdq$\phi_{(2;4)}^{(8)}(0120)^8$ + $\phi_{(2;4)}^{(8)}(0111)^8$ + $x_1^{7}x_2^{11}x_3^{21}x_4^{6}$&\vdq\cr      
$\vdq (1;2,4)$  &\vdq $(1200)$ &\vdq\quad 11 & \vdq$\phi_{(2;4)}^{(10)}(0300)^8$ + $\phi_{(2;4)}^{(8)}(0102)^8$ + $x_1^{7}x_2^{11}x_3^{21}x_4^{6}$&\vdq\cr      
$\vdq (1;3,4)$  &\vdq $(1002)$ &\vdq\quad 11 & \vdq$x_1^{7}x_2^{11}x_3^{5}x_4^{22}$ + $x_1^{7}x_2^{7}x_3^{9}x_4^{22}$ + $x_1^{7}x_2^{7}x_3x_4^{30}$&\vdq\cr      
$\vdq (1;3,4)$  &\vdq $(1011)$ &\vdq\quad 11 & \vdq$x_1^{7}x_2^{11}x_3^{13}x_4^{14}$&\vdq\cr      
$\vdq (1;3,4)$  &\vdq $(1020)$ &\vdq\quad 11 & \vdq$x_1^{7}x_2^{11}x_3^{21}x_4^{6}$ + $x_1^{7}x_2^{7}x_3^{25}x_4^{6}$ + $x_1^{7}x_2^{7}x_3^{9}x_4^{22}$&\vdq\cr      
$\vdq (1;3,4)$  &\vdq $(1101)$ &\vdq\quad 11 & \vdq$\phi_{(3;4)}^{(9)}(0111)^8$ + $x_1^{7}x_2^{11}x_3^{5}x_4^{22}$ + $x_1^{7}x_2^{15}x_3x_4^{22}$&\vdq\cr      
$\vdq (1;3,4)$  &\vdq $(1110)$ &\vdq\quad 11 & \vdq$\phi_{(3;4)}^{(9)}(0120)^8$ + $\phi_{(3;4)}^{(9)}(0111)^8$ + $x_1^{7}x_2^{11}x_3^{21}x_4^{6}$&\vdq\cr      
$\vdq (1;3,4)$  &\vdq $(1200)$ &\vdq\quad 11 & \vdq$\phi_{(3;4)}^{(9)}(0120)^8$ + $x_1^{7}x_2^{27}x_3^{5}x_4^{6}$ + $x_1^{7}x_2^{15}x_3x_4^{22}$&\vdq\cr      
$\vdq (2;\emptyset)$  &\vdq $(1101)$ &\vdq\quad 12 & \vdq$\phi_{(1;2)}^{(4)}(1020)^8$ + $x_1^{15}x_2x_3^{14}x_4^{15}$ + $x_1^{15}x_2x_3^{7}x_4^{22}$&\vdq\cr      
$\vdq (2;\emptyset)$  &\vdq $(1110)$ &\vdq\quad 12 & \vdq$\phi_{(1;2)}^{(4)}(1020)^8$ + $x_1^{15}x_2x_3^{22}x_4^{7}$ + $x_1^{15}x_2x_3^{15}x_4^{14}$&\vdq\cr      
$\vdq (2;3)$  &\vdq $(1101)$ &\vdq\quad 12 & \vdq$\phi_{(1;2,3)}^{(11)}(1002)^8$ + $x_1^{15}x_2x_3^{14}x_4^{15}$ + $x_1^{15}x_2^{3}x_3^{5}x_4^{22}$&\vdq\cr      
$\vdq (2;3)$  &\vdq $(1110)$ &\vdq\quad 12 & \vdq$\phi_{(1;2,3)}^{(11)}(1020)^8$ + $x_1^{15}x_2x_3^{22}x_4^{7}$ + $x_1^{15}x_2^{3}x_3^{13}x_4^{14}$&\vdq\cr      
$\vdq (2;4)$  &\vdq $(1101)$ &\vdq\quad 12 & \vdq$\phi_{(1;2,4)}^{(11)}(1002)^8$ + $x_1^{15}x_2^{3}x_3^{13}x_4^{14}$ + $x_1^{15}x_2x_3^{7}x_4^{22}$&\vdq\cr      
$\vdq (2;4)$  &\vdq $(1110)$ &\vdq\quad 12 & \vdq$\phi_{(1;2,4)}^{(11)}(1020)^8$ + $x_1^{15}x_2^{3}x_3^{21}x_4^{6}$ + $x_1^{15}x_2x_3^{15}x_4^{14}$&\vdq\cr      
\hline
\end{tabular}}

\

\bigskip
\centerline{\begin{tabular}{lllll}
\hline
$\vdq\ \ (i;I)$  &\vdq\quad \  $\bar y$ &\vdq  Case & $\vdq\hskip2cm\phi_{(i;I)}(X^7)\bar y^8 \equiv $&\vdq\cr
\hline
$\vdq (2;3,4)$  &\vdq $(1101)$ &\vdq\quad 12 & \vdq$\phi_{(1;3,4)}^{(11)}(1002)^8$ + $\phi_{(1;\emptyset)}^{(4)}(1002)^8$ + $\phi_{(1;3)}^{(4)}(1002)^8$&\vdq\cr 
\vdq&\vdq &\vdq &\vdq \hskip1.5cm  + $\phi_{(1;4)}^{(4)}(1002)^8$ + $\phi_{(2;\emptyset)}^{(12)}(1101)^8$&\vdq\cr 
\vdq&\vdq &\vdq &\vdq \hskip1.5cm  + $\phi_{(2;3)}^{(12)}(1101)^8$ + $\phi_{(2;4)}^{(12)}(1101)^8$&\vdq\cr      
$\vdq (2;3,4)$  &\vdq $(1110)$ &\vdq\quad 12 & \vdq$\phi_{(1;3,4)}^{(11)}(1020)^8$ + $\phi_{(1;\emptyset)}^{(4)}(1020)^8$ + $\phi_{(1;3)}^{(4)}(1020)^8$&\vdq\cr 
\vdq&\vdq &\vdq &\vdq \hskip1.5cm  + $\phi_{(1;4)}^{(4)}(1020)^8$ + $\phi_{(2;\emptyset)}^{(12)}(1110)^8$&\vdq\cr 
\vdq&\vdq &\vdq &\vdq \hskip1.5cm  + $\phi_{(2;3)}^{(12)}(1110)^8$ + $\phi_{(2;4)}^{(12)}(1110)^8$&\vdq\cr      
$\vdq (2;\emptyset)$  &\vdq $(1200)$ &\vdq\quad 12 & \vdq$\phi_{(1;2)}^{(4)}(1200)^8$ + $\phi_{(2;3)}^{(12)}(1110)^8$ + $\phi_{(2;4)}^{(12)}(1101)^8$&\vdq\cr      
$\vdq (2;3)$  &\vdq $(1200)$ &\vdq\quad 12 & \vdq$\phi_{(1;2,3)}^{(11)}(1200)^8$ + $\phi_{(2;3)}^{(12)}(1110)^8$ + $\phi_{(2;3,4)}^{(12)}(1101)^8$&\vdq\cr      
$\vdq (2;4)$  &\vdq $(1200)$ &\vdq\quad 12 & \vdq$\phi_{(1;2,4)}^{(11)}(1200)^8$ + $\phi_{(2;3,4)}^{(12)}(1110)^8$ + $\phi_{(2;4)}^{(12)}(1101)^8$&\vdq\cr      
$\vdq (2;3,4)$  &\vdq $(1200)$ &\vdq\quad 12 & \vdq$\phi_{(1;3,4)}^{(11)}(1200)^8$ + $\phi_{(1;\emptyset)}^{(4)}(1200)^8$ + $\phi_{(1;3)}^{(4)}(1200)^8$&\vdq\cr 
\vdq&\vdq &\vdq &\vdq \hskip1.5cm  + $\phi_{(1;4)}^{(4)}(1200)^8$ + $\phi_{(2;\emptyset)}^{(12)}(1200)^8$&\vdq\cr 
\vdq&\vdq &\vdq &\vdq \hskip1.5cm  + $\phi_{(2;3)}^{(12)}(1200)^8$ + $\phi_{(2;4)}^{(12)}(1200)^8$&\vdq\cr      
$\vdq (1;\emptyset)$  &\vdq $(3000)$ &\vdq\quad 13 & \vdq$\phi_{(1;2)}^{(4)}(1200)^8$ + $\phi_{(1;3)}^{(4)}(1020)^8$ + $\phi_{(1;4)}^{(4)}(1002)^8$&\vdq\cr      
$\vdq (1;2)$  &\vdq $(3000)$ &\vdq\quad 13 & \vdq$\phi_{(1;2)}^{(4)}(1200)^8$ + $\phi_{(1;2,3)}^{(11)}(1020)^8$ + $\phi_{(1;2,4)}^{(11)}(1002)^8$&\vdq\cr      
$\vdq (1;3)$  &\vdq $(3000)$ &\vdq\quad 13 & \vdq$\phi_{(1;2,3)}^{(11)}(1200)^8$ + $\phi_{(1;3)}^{(4)}(1020)^8$ + $\phi_{(1;3,4)}^{(11)}(1002)^8$&\vdq\cr      
$\vdq (1;4)$  &\vdq $(3000)$ &\vdq\quad 13 & \vdq$\phi_{(1;2,4)}^{(11)}(1200)^8$ + $\phi_{(1;3,4)}^{(11)}(1020)^8$ + $\phi_{(1;4)}^{(4)}(1002)^8$&\vdq\cr      
$\vdq (1;2,3)$  &\vdq $(3000)$ &\vdq\quad 13 & \vdq$\phi_{(2;3)}^{(12)}(1200)^8$ + $x_1^{15}x_2x_3^{14}x_4^{15}$ + $x_1^{15}x_2^{3}x_3^{5}x_4^{22}$&\vdq\cr      
$\vdq (1;2,4)$  &\vdq $(3000)$ &\vdq\quad 13 & \vdq$\phi_{(2;4)}^{(12)}(1200)^8$ + $x_1^{15}x_2^{3}x_3^{21}x_4^{6}$ + $x_1^{15}x_2x_3^{7}x_4^{22}$&\vdq\cr      
$\vdq (1;3,4)$  &\vdq $(3000)$ &\vdq\quad 13 & \vdq$\phi_{(2;3,4)}^{(12)}(1200)^8$ + $\phi_{(1;\emptyset)}^{(4)}(3000)^8$ + $\phi_{(1;3)}^{(4)}(3000)^8$&\vdq\cr 
\vdq&\vdq &\vdq &\vdq \hskip1.5cm + $\phi_{(1;4)}^{(4)}(3000)^8$ + $\phi_{(2;\emptyset)}^{(12)}(1200)^8$&\vdq\cr 
\vdq&\vdq &\vdq &\vdq \hskip1.5cm + $\phi_{(2;3)}^{(12)}(1200)^8$ + $\phi_{(2;4)}^{(12)}(1200)^8$&\vdq\cr      
$\vdq (3;\emptyset)$  &\vdq $(1011)$ &\vdq\quad 14 & \vdq$\phi_{(1;3)}^{(4)}(1002)^8$ + $\phi_{(2;3)}^{(12)}(1101)^8$ + $x_1^{15}x_2^{7}x_3x_4^{22}$&\vdq\cr      
$\vdq (3;\emptyset)$  &\vdq $(1110)$ &\vdq\quad 14 & \vdq$\phi_{(1;3)}^{(4)}(1200)^8$ + $\phi_{(2;3)}^{(12)}(1200)^8$ + $x_1^{15}x_2^{15}x_3x_4^{14}$&\vdq\cr      
$\vdq (3;4)$  &\vdq $(1011)$ &\vdq\quad 14 & \vdq$\phi_{(1;3,4)}^{(11)}(1002)^8$ + $\phi_{(2;3,4)}^{(12)}(1101)^8$ + $x_1^{15}x_2^{7}x_3x_4^{22}$&\vdq\cr      
$\vdq (3;4)$  &\vdq $(1020)$ &\vdq\quad 14 & \vdq$\phi_{(1;3,4)}^{(11)}(1020)^8$ + $\phi_{(2;3,4)}^{(12)}(1110)^8$ + $\phi_{(3;4)}^{(14)}(1011)^8$&\vdq\cr      
$\vdq (3;4)$  &\vdq $(1110)$ &\vdq\quad 14 & \vdq$\phi_{(1;3,4)}^{(11)}(1200)^8$ + $\phi_{(2;3,4)}^{(12)}(1200)^8$ + $x_1^{15}x_2^{15}x_3x_4^{14}$&\vdq\cr      
$\vdq (4;\emptyset)$  &\vdq $(0111)$ &\vdq\quad 14 & \vdq$\phi_{(1;\emptyset)}^{(4)}(1110)^8$ + $\phi_{(2;\emptyset)}^{(4)}(0120)^8$ + $\phi_{(3;\emptyset)}^{(9)}(0120)^8$&\vdq\cr      
$\vdq (4;\emptyset)$  &\vdq $(1101)$ &\vdq\quad 14 & \vdq$\phi_{(1;4)}^{(4)}(1200)^8$ + $\phi_{(2;4)}^{(12)}(1200)^8$ + $\phi_{(3;4)}^{(14)}(1110)^8$&\vdq\cr      
$\vdq (3;\emptyset)$  &\vdq $(1020)$ &\vdq\quad 14 & \vdq$\phi_{(1;3)}^{(4)}(1020)^8$ + $\phi_{(2;3)}^{(12)}(1110)^8$ + $\phi_{(3;4)}^{(14)}(1011)^8$&\vdq\cr      
$\vdq (4;\emptyset)$  &\vdq $(0102)$ &\vdq\quad 14 & \vdq$\phi_{(1;4)}^{(4)}(1101)^8$ + $\phi_{(2;4)}^{(8)}(0102)^8$ + $\phi_{(3;4)}^{(9)}(0111)^8$&\vdq\cr      
$\vdq (4;\emptyset)$  &\vdq $(1002)$ &\vdq\quad 14 & \vdq$\phi_{(1;4)}^{(4)}(1002)^8$ + $\phi_{(2;4)}^{(12)}(1101)^8$ + $\phi_{(3;4)}^{(14)}(1011)^8$&\vdq\cr   
$\vdq (4;\emptyset)$  &\vdq $(1011)$ &\vdq\quad 14 & \vdq$\phi_{(1;4)}^{(4)}(1020)^8$ + $\phi_{(2;4)}^{(12)}(1110)^8$ + $\phi_{(3;4)}^{(14)}(1020)^8$&\vdq\cr      
\hline
\end{tabular}}

\end{document}